\documentclass[11pt]{amsart}
\usepackage{amsmath,amssymb,latexsym,amsthm,enumerate}
\usepackage{hyperref} %To be able to click on the links in the paper
\usepackage{verbatim} 

\theoremstyle{plain}
\newtheorem{theorem}{Theorem}[section]
\newtheorem{proposition}[theorem]{Proposition}
\newtheorem{lemma}[theorem]{Lemma}
\newtheorem{corollary}[theorem]{Corollary}
\theoremstyle{definition}

\newtheorem*{remark}{Remark}
\newtheorem*{remarks}{Remarks}
\theoremstyle{remark}

\newcounter{numpar}[section]

\numberwithin{equation}{section}

\newcommand*{\dd}{\mathrm{d}}     %For dx, dy, etc.
\newcommand*{\te}{\mathrm{e}}

\newcommand*{\cC}{\mathcal C}

\newcommand*{\cE}{\mathcal E}
\newcommand*{\cF}{\mathcal F}
\newcommand*{\cG}{\mathcal G}
\newcommand*{\cH}{\mathcal H}
\newcommand*{\cK}{\mathcal K}
\newcommand*{\cL}{\mathcal L}
\newcommand*{\cV}{\mathcal V}

\newcommand*{\bbE}{\mathbb E}
\newcommand*{\bbN}{\mathbb N}

\newcommand*{\bbR}{\mathbb R}
\newcommand*{\bbZ}{\mathbb Z}

\newcommand*{\EE}{\mathsf E}
\newcommand*{\PP}{\mathsf P}

\DeclareMathOperator{\sgn}{sgn}

\begin{document}
\title[Coupling and tracking of regime-switching martingales]{Coupling and tracking of regime-switching martingales}

\author{Saul D. Jacka}
\address{Department of Statistics, University of Warwick, UK}
\email{s.d.jacka@warwick.ac.uk}

\author{Aleksandar Mijatovi\'{c}}
\address{Department of Mathematics, Imperial College London, UK}
\email{a.mijatovic@imperial.ac.uk}

\thanks{We would like to thank Peter Bank, David Hobson and Dimitry Kramkov for 
useful comments. We are grateful to two anonymous referees whose suggestions and comments 
greatly improved the paper.}

\keywords{generalised mirror and synchronous coupling
of Brownian motion, 
coupling time and tracking error
of regime-switching martingales, 
Bellman principle, 
continuous-time Markov chains, stochastic integrals} 

\subjclass[2010]{60H05, 60J27, 93E20}

\begin{abstract}
This paper describes two explicit couplings of standard Brownian motions
$B$
and
$V$,
which naturally extend the mirror coupling and the synchronous coupling
and respectively maximise and minimise
(uniformly over all time horizons)
the coupling time and the tracking error 
of two regime-switching martingales. 
The generalised mirror coupling minimises the 
coupling time of the two martingales while
simultaneously maximising the tracking error
for all time horizons.  The generalised synchronous 
coupling maximises the coupling time and 
minimises the tracking error over all co-adapted
couplings. The proofs are based on the Bellman principle.
We give counterexamples to the conjectured optimality of the two couplings
amongst a wider classes of stochastic integrals.
\end{abstract}

\maketitle

%========================================================================================
\section{Introduction}
\label{sec:i}

Let
$(\Omega, (\cF_t)_{t\geq0}, \cF, \PP)$
be a filtered probability space that supports 
a standard
$(\cF_t)$-Brownian motion
$B=(B_t)_{t\geq0}$
and let
$$
\cV:=\left\{V=(V_t)_{t\geq0}\,:\,V\text{ is an $(\cF_t)$-Brownian motion with }\,V_0=0 \right\}
$$
be the set of all $(\cF_t)$-Brownian motions on this probability space.
It is well-known that for any time horizon 
$T>0$
the Brownian motion in
$\cV$
which minimises the probability that 
the processes
$X=x+B$
and
$Y(V)=y+V$
couple
after time 
$T$
(for any starting points 
$x,y\in\bbR$),
i.e. 
the Brownian motion that solves the problem
$$
\text{minimise}\quad\PP\left[\tau_0(X-Y(V))>T\right]\qquad \text{over $V\in\cV$,}
$$
where
$\tau_0(X-Y(V)):=\inf\{t\geq0:X_t=Y_t(V)\}$,
is given by the \textit{mirror coupling}
$V=-B$
(see e.g.~\cite{HsuSturm:03}).
Furthermore it is easy to see that the Brownian motion
which minimises the tracking error of
$Y(V)$
with respect to the target
$X$
at time 
$T$,
i.e. solves 
$$\text{minimise}\quad\EE\left[\left(X_T-Y_T(V)\right)^2\right]\qquad \text{over $V\in\cV$,}$$
is given by the 
\textit{synchronous coupling}
$V=B$.
This paper investigates the following generalisations of these questions.

\subsection{Problems}
\label{subsec:Problem}
Let
$Z=(Z_t)_{t\geq0}$
be an $(\cF_t)$-Feller process,
i.e. a Feller process 
on our probability space,
which is 
$(\cF_t)$-Markov.
Let
the state space
$\bbE$
of 
$Z$
be a subset 
of a Euclidean space
$\bbR^d$
for some
$d\in\bbN$.
For real Borel measurable functions
$\sigma_i:\bbE\to\bbR$,
$i=1,2$,
define the stochastic integrals 
$X=(X_t)_{t\geq0}$
and
$Y(V)=(Y_t(V))_{t\geq0}$
by
\begin{eqnarray}
\label{eq:Proc_Def_X}
X_t  :=   x+\int_0^t\sigma_1(Z_s)\,\dd B_s & \text{ and } &
Y_t(V)  :=  y+\int_0^t\sigma_2(Z_s)\,\dd V_s,
\end{eqnarray}
where
$x,y\in\bbR$
and
$V\in\cV$.
Throughout the paper we assume that 
for each starting point the process
$Z$
is a semimartingale (in particular, it is non-explosive and
has c\`adl\`ag paths) 
%without instantaneous states)
and
\begin{eqnarray}
\label{eq:Mart_Integrab_Quad_VAr}
\EE\int_0^t\sigma_i^2(Z_s)\dd s<\infty\qquad\text{for all}\qquad
t>0,\,
i=1,2.
\end{eqnarray}
This implies that the processes
$X$
and
$Y(V)$
in~\eqref{eq:Proc_Def_X}
are well-defined 
true martingales 
(e.g.  see~\cite[Cor~IV.1.25]{RevuzYor:99}).
In the case the state space 
$\bbE$
of 
$Z$
is embedded in a multidimensional space,
a natural choice for the volatility functions 
$\sigma_1$
and
$\sigma_2$
are the projections 
resulting in 
$\sigma_1(Z)$
and
$\sigma_2(Z)$
being coordinate processes of
$Z$
in 
$\bbR^d$.
Furthermore, to avoid degenerate situations, 
we assume throughout the paper that 
$(|\sigma_1|+|\sigma_2|)(z)>0$
for
all
$z\in\bbE$.
The class of stochastic integrals in~\eqref{eq:Proc_Def_X},
with the integrand  
$Z$
typically a jump-diffusion (i.e. a Feller process),
arises frequently and is of interest 
in the theory and practice of mathematical
finance in the guise of stochastic volatility models
(see e.g.~\cite{Gatheral}).

We are interested in the ``distance'' between the 
two processes 
$X$
and
$Y(V)$
for any
$V\in\cV$.
In other words we seek to understand how large and small the following  quantities 
can be
\begin{equation}
\label{eq:Distance}
\EE\left[\phi(X_T-Y_T(V))\right]
\qquad\text{and}\qquad
\PP\left[\tau_0(X-Y(V))>T\right],
\end{equation}
for
$T>0$
a fixed time horizon,
\begin{equation}
\label{eq:Ass_phi}
\phi:\bbR\to\bbR %\in\cC^2(\bbR) 
\text{ a convex function 
%bounded from below  
%that satisfies
%the inequality 
satisfying
$|\phi(x)|\leq a |x|^p+b$
for some 
$a,b>0$,
$p\geq 2$
and 
$\forall x\in\bbR$,}
\end{equation}
and 
$\tau_0(X-Y(V)):=\inf\{t\geq0:X_t=Y_t(V)\}$
the \textit{coupling time} of the processes
$X$
and
$Y(V)$.
Since 
$V$
is an arbitrary
$(\cF_t)$-Brownian motion,
the law of the difference
$X-Y(V)$
is in general not easy to describe. Therefore we cannot expect 
to be able to identify the quantities in~\eqref{eq:Distance}
explicitly. Our goal is to establish sharp upper and lower bounds
for the expectations in~\eqref{eq:Distance}, 
which hold for 
any choice of Brownian motion
$V\in\cV$
\textit{and} are based on a natural generalisations
of the mirror and synchronous couplings of Brownian motions 
described in Section~\ref{subsection:Solution}.
%in~\eqref{eq:Def_VI_VII}.
More precisely, we are looking for Brownian motions 
$V^{M}, V^{S}\in\cV$
%and
%$V^{Cu}, V^{Cd}\in\cV$
such that the following inequalities 
hold 
for all 
$V\in\cV$:
\begin{eqnarray*}
&\hspace{-12mm}
\text{\textbf{(T)}}\hspace{10mm}  
\EE\left[\phi(X_T-Y_T(V^{S}))\right] \hspace{3mm}
\leq  
\hspace{3mm} \EE\left[\phi(X_T-Y_T(V))\right] \hspace{3mm}
\leq \hspace{3mm}\EE\Big[\phi(X_T-Y_T(V^{M}))\Big],&\\
&\text{\textbf{(C)}}\hspace{5mm} 
\PP\left[\tau_0(X-Y(V^{M}))>T\right] \hspace{3mm}\leq \hspace{3mm}
\PP\left[\tau_0(X-Y(V))>T\right] \hspace{3mm} \leq \hspace{3mm}
\PP\Big[\tau_0(X-Y(V^{S}))>T\Big],& 
\end{eqnarray*}
where 
the generalised mirror
(resp. synchronous)
coupling holds for
$B$
and
$V^{M}$
(resp. 
$V^{S}$).
%Similarly in the case of coupling, we seek 
%Brownian motions 
%$V^{Cu}, V^{Cd}\in\cV$
%such that 
%\begin{eqnarray*}
%\text{\textbf{(C)}}\hspace{5mm} 
%\PP\left[\tau_0(X-Y(V^{Cd}))>T\right] \hspace{3mm}\leq \hspace{3mm}
%\PP\left[\tau_0(X-Y(V))>T\right] \hspace{3mm} \leq \hspace{3mm}
%\PP\Big[\tau_0(X-Y(V^{Cu}))>T\Big]
%\end{eqnarray*}
%for all 
%$V\in\cV$.

In Problems~\textbf{(T)} and~\textbf{(C)},
the goal is not merely to prove the existence in an abstract sense of the
integrators
$V^{M}, V^{S}\in\cV$,
\textit{but} primarily to understand for which classes of  
$(\cF_t)$-Feller processes 
$Z$
are the generalised mirror and synchronous couplings of Brownian motions,
described in Section~\ref{subsection:Solution},  
extremal in the inequalities of Problems~\textbf{(T)} and~\textbf{(C)}.
In particular, for the volatility processes 
$Z$ with the property that the generalised mirror and synchronous couplings 
satisfy the inequalities above for all 
Brownian motions 
$V\in\cV$,
the following holds:
maximising the coupling time 
of the stochastic integrals
minimises the ``convex distance'' 
of the two processes
and vice versa uniformly over all time horizons
$T>0$. 

\subsection{Results}
\label{subsection:Solution}
In the setting of processes~\eqref{eq:Proc_Def_X},
it is natural to define generalised synchronous 
and mirror couplings of Brownian motions in the following way.
Let the functions
$\hat c_I,\hat c_{II}:\bbE\to\bbR$
be given by the formulae
$$
\hat c_I(z):=\sgn(\sigma_1(z)\sigma_2(z)),
\qquad 
\hat c_{II}(z):=-\sgn(\sigma_1(z)\sigma_2(z))
$$
for any
$z\in\bbE$,
and define the 
Brownian motions 
$V^I=(V^I_t)_{t\geq0}$
and
$V^{II}=(V^{II}_t)_{t\geq0}$
in
$\cV$
by
\begin{equation}
\label{eq:Def_VI_VII}
V^I_t :=\int_0^t \hat c_I(Z_s)\,\dd B_s
\qquad\text{and}\qquad
V^{II}_t :=\int_0^t \hat c_{II}(Z_s)\,\dd B_s. % = - V^I_t.
\end{equation}
Note that 
$\hat c_{II}=-\hat c_I$
and hence 
$V^{II}=-V^I$.
It is clear from~\eqref{eq:Def_VI_VII}
that 
$B$
and
$V^I$
generalise the synchronous coupling of Brownian motions,
while the pair 
$B$
and
$V^{II}$
extends the notion of the mirror coupling.
A natural conjecture, based on the 
case where 
$X$
and
$Y(V)$
are Brownian motions,
goes as follows.

\begin{flushleft}
\textit{Conjecture.} For any $(\cF_t)$-Feller process
$Z$
and 
$V\in\cV$,
the inequalities in~\textbf{(T)}
and~\textbf{(C)}
are satisfied 
by
$V^{S}=V^{I}$
and
$V^{M}=V^{II}=-V^{I}$.
\end{flushleft}

\subsubsection{The conjecture fails in the class of general $(\cF_t)$-Feller processes}
%and a non-symmetric cost function $\phi$}
\label{subsubsec:GEneral_Fails}
Let the Feller process
$Z$,
with state space
$\bbE:=(0,\infty)$,
be defined as
\begin{equation}
\label{eq:Z_dependent_BM_Counter}
Z_t := z_0 M_t,\qquad \text{where $M_t:=\exp(B_t-t/2)$ and $z_0>0$,}
\end{equation}
and the volatility functions 
$\sigma_1, \sigma_2:\bbE\to \bbR$
given by
$\sigma_i(z):=-iz$
for any $z\in\bbE$ and $i=1,2$.
The corresponding candidate extremal Brownian motions
$V^{I}$
and
$V^{II}$,
defined in~\eqref{eq:Def_VI_VII},
are in this case given by
the classical synchronous 
$V^I=B$
and mirror 
$V^{II}=-B$
couplings.
The fact that
$
M_t = 1+\int_0^tM_s\dd B_s
$
yields
$\int_0^t \sigma_i(Z_s)\dd B_s=-iz_0(M_t-1)$,
for
$i=1,2$,
which in particular implies
the following for all
$t\geq0$:
\begin{eqnarray}
\label{eq:Explicit_Expressions_Intro}
X_t-Y_t(V^I) = x-y + z_0(M_t-1)
\quad\text{and}\quad
X_t-Y_t(V^{II}) = x-y -3z_0(M_t-1).
\end{eqnarray}

Fix a time horizon
$T>0$
and note that,
since~\eqref{eq:Explicit_Expressions_Intro} implies the supports of the random variables 
$X_T-Y_T(V^{I})$
and
$X_T-Y_T(V^{II})$
are given by
$$
\mathrm{supp}\left(X_T-Y_T(V^{I})\right) = (x-y-z_0,\infty)
\quad\text{and}\quad
\mathrm{supp}\left(X_T-Y_T(V^{II})\right) = (-\infty,x-y+3z_0),
$$
any non-negative non-zero convex function 
$\phi:\bbR\to\bbR$
that satisfies the assumptions in~\eqref{eq:Ass_phi},
%Section~\ref{subsec:Problem},
with support
(i.e. the closure of $\phi^{-1}(0,\infty)$)
contained in the half-line  
$(x-y+3z_0,\infty)$,
clearly yields
$$
0=\EE\left[\phi\left(X_T-Y_T(V^{II})\right)\right]  <
\EE\left[\phi\left(X_T-Y_T(V^{I})\right)\right]. 
$$
Hence the tracking part of the conjecture fails
for 
$Z=z_0M$.
%However, it is easy to see that
%for the quadratic cost function 
%we have
%$\EE\left[\left(X_T-Y_T(V^{I})\right)^2\right] \leq
%\EE\left[\left(X_T-Y_T(V^{II})\right)^2\right]$ 
%as suggested in Section~\ref{subsec:Non_Markov_Trackin}.

Assume that the starting points in~\eqref{eq:Proc_Def_X}
satisfy
$x-y<-3z_0$
and define the stopping time 
$\tau:=\inf\{t\geq0:B_t-t/2=\log(1-(x-y)/z_0)\}$.
Note that the representations 
in~\eqref{eq:Explicit_Expressions_Intro}
imply 
$\PP\left[\tau_0(X-Y(V^{II}))=\infty\right]=1$
and
$\PP\left[\tau_0(X-Y(V^{I}))>T\right] =\PP\left[\tau>T\right]<1$
for any time horizon
$T>0$.
Therefore the coupling part of the conjecture also fails:
$$
\PP\left[\tau_0(X-Y(V^{I}))>T\right] <
\PP\left[\tau_0(X-Y(V^{II}))>T\right]=1.
$$
%which implies that Proposition~\ref{prop:Time_Determ}
%does not hold in general beyond deterministic volatility processes. 

\subsubsection{The generalised mirror and synchronous couplings are optimal if $Z$ is a continuous-time Markov chain}
\label{subsubsec:adapted_CTMC}
Unless otherwise stated, in the rest of the paper 
$Z$
denotes an %continuous-time 
$(\cF_t)$-Markov semimartingale with a countable state space. 
More precisely, we assume that 
\begin{equation}
\label{eq:Z_Chain_Assumption}
\text{$Z$ is a non-explosive, irreducible, c\`adl\`ag $(\cF_t)$-Feller process on a discrete space
$\bbE\subset\bbR^d$.}
%and c\'adl\'ag trajectories.}
%$Z$
%that satisfies the assumptions in Section~\ref{subsec:Problem}
\end{equation}
Assumption~\eqref{eq:Z_Chain_Assumption} makes 
$\bbE$
a countable set
(i.e. the cardinality of $\bbE$
is at most that of $\bbN$)
and
$Z$
a continuous-time 
$(\cF_t)$-Markov chain on $\bbE$.
The following assumptions on the semigroup
$P$
and the 
$Q$-matrix 
$Q$
of the chain 
$Z$
ensure the finiteness of the expectations in~\textbf{(T)} 
%are finite 
(see Section~\ref{sec:Tracking}) 
and the regularity of the law of the coupling time 
(see Section~\ref{sec:Coupling}) 
respectively:\footnote{It is not hard to show that
neither of the conditions~\eqref{eq:Int_Assump} 
and~\eqref{eq:Coupling_Int_Assump}
implies the other, see Appendix~\ref{app:Examples}.}
\begin{eqnarray}
\label{eq:Int_Assump}
\forall z\in\bbE:\quad (P_T(|\sigma_1|^p+|\sigma_2|^p))(z)<\infty,\\
\forall z\in\bbE:\quad (Q(|\sigma_1|^2+|\sigma_2|^2))(z)<\infty.
\label{eq:Coupling_Int_Assump}
\end{eqnarray}
%We define the action of 
%$Q$
%on the space of bounded functions on
%$\bbE$
%in the standard way:
%for a bounded
%$g:\bbE\to\bbR$,
%let
%\begin{eqnarray*}
%%\label{eq:Q_Action}
%Qg:\bbE\to\bbR\qquad\text{be given by the formula}\quad
%(Qg)(z):=\sum_{z'\in\bbE}Q(z,z')g(z'),
%\end{eqnarray*}
%since the series %in~\eqref{eq:Q_Action}
%converges absolutely for every
%$z\in\bbE$.
%

%We can now state Theorem~\ref{thm:tracking}, 
%which describes the extremal couplings of 
%pairs of standard Brownian motions
%$(B,V^{Tu})$
%and
%$(B,V^{Td})$
%that feature in the inequalities of Problem~\textbf{(T)}.

\begin{theorem}
\label{thm:tracking}
Let a 
%$(\cF_t)$-adapted 
Markov chain 
$Z$
%be as above. 
%If 
%$Z$
%also 
satisfy~\eqref{eq:Mart_Integrab_Quad_VAr},~\eqref{eq:Z_Chain_Assumption} and~\eqref{eq:Int_Assump}
and $\phi$
be as in~\eqref{eq:Ass_phi}.
Then 
\begin{eqnarray*}
\EE\left[\phi(X_T-Y_T(V^{I}))\right] \hspace{3mm}
\leq  
\hspace{3mm} \EE\left[\phi(X_T-Y_T(V))\right] \hspace{3mm}
\leq \hspace{3mm}\EE\left[\phi(X_T-Y_T(V^{II}))\right]
\quad \text{for any $V\in\cV$.}
\end{eqnarray*}
\end{theorem}

The integrability condition in~\eqref{eq:Int_Assump}
is not necessary for the solution of 
Problem~\textbf{(C)}. 

\begin{theorem}
\label{thm:coupling}
Let an
%$(\cF_t)$-adapted 
$(\cF_t)$-Markov chain
$Z$
satisfy~\eqref{eq:Mart_Integrab_Quad_VAr},~\eqref{eq:Z_Chain_Assumption} and~\eqref{eq:Coupling_Int_Assump}.
%~\eqref{eq:Mart_Integrab_Quad_VAr} and~\eqref{eq:Z_Chain_Assumption}. 
%the assumptions in 
%Section~\ref{subsec:Problem}.
Then 
\begin{eqnarray*}
\PP\left[\tau_0(X-Y(V^{II}))>T\right] \hspace{1mm}\leq \hspace{1mm}
\PP\left[\tau_0(X-Y(V))>T\right] \hspace{1mm} \leq \hspace{1mm}
\PP\left[\tau_0(X-Y(V^{I}))>T\right] \quad 
\text{for any $V\in\cV$.}
\end{eqnarray*}
\end{theorem}

\begin{remarks} 
\begin{enumerate}[(i)]
\item 
The function
$\hat c_I = -\hat c_{II}$,
and hence the Brownian motions
$V^I=-V^{II}$,
that feature in the solution of 
Problems~\textbf{(T)} and~\textbf{(C)}
depend neither on the maturity 
$T$ 
nor on the precise form of the convex cost function
$\phi$.
No local regularity (e.g. differentiability) of 
$\phi$
is required for Theorem~\ref{thm:tracking} to hold.
Note also that essentially no restriction 
on the volatility functions  
$\sigma_1$
and 
$\sigma_2$
in the stochastic integrals in~\eqref{eq:Proc_Def_X}
is necessary, for the two theorems to hold.
Furthermore, the assumptions in 
Theorems~\ref{thm:tracking} and~\ref{thm:coupling}
place no restrictions on
the filtration
$(\cF_t)_{t\geq0}$;
in particular 
$(\cF_t)_{t\geq0}$
need not be generated by the processes
$B$
and 
$Z$. 
\item 
\label{Rem:Randon_Nikodym_DEr}
Brownian motion
$V^I$ (resp. $V^{II}$)
is chosen to minimise (resp. maximise) at each moment in time 
the Radon-Nikodym derivative of the quadratic variation of the process
$X-Y(V)$
over the set
$\cV$.
It is clear that 
$V^I$ and $V^{II}$
can also be defined for much more general 
integrands than the ones considered in~\eqref{eq:Proc_Def_X}
and that the generalisations will still be locally extremal.
\item Section~\ref{subsec:Non_Markov_Trackin}
shows that local maximisation/minimisation 
of the Radon-Nikodym derivative mentioned in item~\eqref{Rem:Randon_Nikodym_DEr}
is 
also globally optimal in a 
non-Markovian setting in the special case 
of the quadratic tracking (i.e. where the cost function is
$\phi(x)=x^2$).
Section~\ref{subsec:Non_Markov_Coupling}
establishes a coupling result, analogous to
Theorem~\ref{thm:coupling},
in the case where the volatility processes are 
time-inhomogeneous but
deterministic.
However, Sections~\ref{subsubsec:GEneral_Fails}
and~\ref{subsection:Independence_Also_Does_NOt_Work}
show that the generalisations of 
Theorems~\ref{thm:tracking} and~\ref{thm:coupling}
do not hold for general $(\cF_t)$-Feller processes.
\item The key fact, established in 
Lemma~\ref{lem:_Zand_W_indep},
that enables us to prove 
Theorems~\ref{thm:tracking} and~\ref{thm:coupling}
is that the chain 
$Z$
is in fact independent 
of the driving 
Brownian motion 
$B$
(see Section~\ref{subsec:Indepdnednt_Z_B}).
It is therefore natural to ask whether the results
in the theorems above hold for a general $(\cF_t)$-Feller process
$Z$,
which is 
independent of 
$B$.
The example in 
Section~\ref{subsection:Independence_Also_Does_NOt_Work}
shows that 
Theorem~\ref{thm:tracking}
cannot be generalised even if 
such independence is assumed.
\item The results in 
Theorems~\ref{thm:tracking} and~\ref{thm:coupling}
are likely to remain valid in the generalised
setting %if Lemma~\ref{lem:_Zand_W_indep} is assumed:
given by the filtered space 
$(\Omega, (\cF_t)_{t\geq0}, \cF, \PP)$
supporting an additional filtration 
$(\cG_t)_{t\geq0}$,
such that 
$\cF_t\subseteq\cG_t$
for
$t\geq0$,
with properties that every Brownian motion in
$V\in\cV$
is also a 
$(\cG_t)$-Brownian motion
and the continuous
$(\cG_t)$-Feller process $Z$
is independent of any $V\in\cV$.
These conditions are satisfied 
for example by 
$\cG_t:=\cF_t\otimes\cH_t$,
where the filtration
$(\cH_t)_{t\geq0}$
is independent of 
$(\cF_t)_{t\geq0}$
and supports a continuous
$(\cH_t)$-Feller (and hence 
$(\cG_t)$-Feller) process
$Z$,
e.g. $Z$ is a stochastic volatility process
(i.e. a solution of an SDE)
driven by 
an $(\cH_t)$-Brownian motion.
The reason why such a generalisation is likely to remain true
lies in the fact that the representation in~\eqref{eq:Def_cL}
still holds in this setting and the continuity of
the paths of the process 
$Z$
could be used to perform the necessary localisations
in the proofs of 
Theorems~\ref{thm:tracking} and~\ref{thm:coupling}.
Note that by Lemma~\ref{lem:_Zand_W_indep} 
the setting of the paper is given by 
$\cG_t:=\cF_t$
and $Z$ a continuous-time $(\cF_t)$-Markov chain.\footnote{We thank
one of the referees for this remark.}
%be a filtered probability space that supports 
%a standard
%$(\cF_t)$-Brownian motion
%$B=(B_t)_{t\geq0}$
%and let
%$$
%\cV:=\left\{V=(V_t)_{t\geq0}\,:\,V\text{ is an $(\cF_t)$-Brownian motion with }\,V_0=0 \right\}
%$$
%However, the assumptions
%required (every $\mathcal{F}_t$-Brownian motion in $\mathcal{V}$ is independent of the
%$\mathcal{G}_t$-Feller process $Z$, where $\mathcal{F}_t\subseteq\mathcal{G}_t$, \textit{and}
%every $V\in\mathcal{V}$ is also a  $\mathcal{G}_t$-Brownian motion) appear to be
%slightly unnatural, with the main example where they are satisfied being $\mathcal{F}_t=\mathcal{G}_t$
%and $Z$ an $\mathcal{F}_t$-Markov chain, which is the setting of the paper.
\item The volatility functions 
$\sigma_1$
and 
$\sigma_2$
are typically distinct, which makes the maximal coupling time
$\tau_0(X-Y(V^I))$
finite.
Hence the upper bound 
in Theorem~\ref{thm:coupling}
is
non-trivial 
(i.e. smaller than $1$).
\item Recall that %\label{rem:Sign}
$\sgn(x)$
is $1$
if
$x>0$
and 
$-1$
if
$x<0$.
In the setting of Theorems~\ref{thm:tracking} and~\ref{thm:coupling}, 
the choice of 
$\sgn(0)$
in
$\{1,-1\}$
can be arbitrary, since 
by~\cite[Prop~IV.1.13]{RevuzYor:99}
it influences neither the laws of the processes
$\phi(X-Y(V^I))$,
$\phi(X-Y(V^{II}))$
nor of the variables
$\tau_0(X-Y(V^{I}))$, $\tau_0(X-Y(V^{II}))$.  
\item In~\cite{BarlowJacka:86}
the authors establish an inequality,
analogous to 
the first inequality of
Theorem~\ref{thm:tracking},
in the case 
$X$
and
$Y(V)$
are solutions of driftless SDEs. 
A related inverse question to the tracking problem
is studied in~\cite{McNamara:83}.
A general reference on the theory of coupling is given in~\cite{Lindvall:02}.
\item In the case $Z$
is a continuous-time Markov chain, the processes in~\eqref{eq:Proc_Def_X}
are regime-switching martingales as they evolve as Brownian motions
with varying values of the instantaneous volatility, determined by the current state of the chain $Z$
and the functions
$\sigma_i$, $i=1,2$.
The seminal paper~\cite{Hamilton:89} 
introduced such regime-switching models  
to economics and finance. Since then,
such models  
have found a plethora of applications in areas as diverse as
macroeconomics, term-structure modelling
and option pricing (see e.g.~\cite{LaiXing:08}
and the references therein). 
\end{enumerate}
\end{remarks}

\subsection{Structure of the paper}
Sections~\ref{subsec:BM_rep} and~\ref{subsec:Q_MArt}
state two well-known lemmas 
that allow us to relate the coupling inequalities above
to problems in stochastic control. 
Section~\ref{subsec:Indepdnednt_Z_B}
proves that the $(\cF_t)$-Markov chain
$Z$
and the Brownian motion $B$
are independent. 
Sections~\ref{sec:Tracking}
and~\ref{subsec:Proof_TRacking}
prove Theorem~\ref{thm:tracking}. 
Section~\ref{subsec:Non_Markov_Trackin}
discusses Problem~\textbf{(T)}
in a non-Markovian setting
and establishes a generalisation of 
Theorem~\ref{thm:tracking}
in the case of a quadratic cost function.
In Sections~\ref{sec:Coupling},~\ref{subsec:StochTime_Change}
and~\ref{subsec:Proof_Coupling},
we establish Theorem~\ref{thm:coupling}.
Section~\ref{subsec:Non_Markov_Coupling}
proves an analogue of Theorem~\ref{thm:coupling}
in the case the volatility processes are time-inhomogeneous
but deterministic.
Section~\ref{sec:CounterEx} discusses four counterexamples
to the Conjecture above in the case where certain assumptions of 
Theorems~\ref{thm:tracking} and~\ref{thm:coupling}
are violated.
Appendix~\ref{sec:App}
contains the proofs of Lemmas~\ref{lem:R^V_rep}
and~\ref{lem:Q_mart}.
of Section~\ref{sec:Setting}.

%========================================================================================
\section{Preliminaries}
\label{sec:Setting}

\subsection{The set of Brownian motions on a probability space}
\label{subsec:BM_rep}
Without loss of generality we may assume that
the probability space
$(\Omega, (\cF_t)_{t\geq0}, \cF, \PP)$,
where the $(\cF_t)$-Brownian motion 
$B$
and the chain
$Z$
in~\eqref{eq:Proc_Def_X}
are defined, 
supports a further 
$(\cF_t)$-Brownian motion 
$B^\perp\in\cV$,
which is independent of 
$B$.
If this were not the case, we could enlarge the 
probability space and note that this only increases
the set 
$\cV$
of all 
$(\cF_t)$-Brownian motions.
Since the extremal Brownian motions
$V^I,V^{II}$
in Problems~\textbf{(T)} and~\textbf{(C)}
are constructed from 
$B$
and
$Z$
alone, they must also be extremal in the original problem.
We shall henceforth assume that 
$B^\perp\in\cV$
exists.
Any 
$V\in\cV$
and the process
$X-Y(V)$,
which plays a key role in all that follows, 
therefore possess the following representation.

\begin{lemma} 
\label{lem:R^V_rep}
For any
$V\in\cV$
there exist $(\cF_t)$-Brownian motion 
$W\in\cV$
and
$C=(C_t)_{t\geq0}$,
such that 
$W$
and
$B$
are independent,
$C$
is progressively measurable with
$-1\leq C_t\leq 1$
for all
$t\geq0$
$\PP$-a.s.,
and the following representations hold:
\begin{equation}
\label{eq:V_Rep}
V_t=\int_0^tC_s\,\dd B_s + \int_0^t(1-C^2_s)^{1/2}\,\dd W_s, 
\end{equation}
and 
$X-Y(V)=R(V)$,
where
$R(V)=(R_t(V))_{t\geq0}$
is given by 
$R_0(V)=r:=x-y$
and
\begin{equation}
\label{eq:R^V_Rep}
R_t(V):=r+\int_0^t(\sigma_1(Z_s)-C_s\sigma_2(Z_s))\,\dd B_s 
- \int_0^t(1-C^2_s)^{1/2}\sigma_2(Z_s)\,\dd W_s.
\end{equation}
\end{lemma}

\begin{remarks}
\begin{enumerate}[(i)]
\item Equality~\eqref{eq:V_Rep} in Lemma~\ref{lem:R^V_rep} 
is a well-known representation for a 
Brownian motion 
$V\in\cV$
in terms of 
$B$
(see e.g.~\cite{BarlowJacka:86} and the references
therein).
For completeness and because of the importance
of the representation in~\eqref{eq:R^V_Rep},
which follows directly from~\eqref{eq:V_Rep},
the proof of Lemma~\ref{lem:R^V_rep}
is given in the appendix (see Section~\ref{sec:App_BM_representation_sketch});
it is this proof that requires the existence of
$B^\perp\in\cV$
independent  
of 
$B$.
\item Note that 
$W$
and
$B$
in Lemma~\ref{lem:R^V_rep}
are independent, but the process 
$C$
may depend on either (or both)
Brownian motions 
$B,W$.
\end{enumerate}
\end{remarks}

\subsection{$Q$-matrices, related operators and martingales}
\label{subsec:Q_MArt}
Let
$Q$
denote the 
$Q$-matrix 
of the continuous-time
%$(\cF_t)$-adapted 
$(\cF_t)$-Markov chain
$Z$.
We define the action of 
$Q$
on the space of bounded functions on
$\bbE$
in the standard way:
for a bounded
$g:\bbE\to\bbR$,
let
\begin{eqnarray*}
%\label{eq:Q_Action}
Qg:\bbE\to\bbR\qquad\text{be given by the formula}\quad
(Qg)(z):=\sum_{z'\in\bbE}Q(z,z')g(z'),
\end{eqnarray*}
since the series %in~\eqref{eq:Q_Action}
converges absolutely for every
$z\in\bbE$.

Let the function 
$H:\bbE\times\bbR\to\bbR$
satisfy the assumptions:
$H(\cdot,z)\in\cC^{2}(\bbR)$
and %the restriction 
$H(r,\cdot):\bbE\to\bbR$
is bounded
for any
$r\in\bbR$.
Then, for any
$c\in[-1,1]$,
we define 
$\cL^cH:\bbE\times\bbR\to\bbR$ 
by the formula:
\begin{eqnarray}
\label{eq:Def_cL}
(\cL^c H)(r,z) & := &  \frac{1}{2}\left(\sigma_1^2-c2\sigma_1\sigma_2+\sigma_2^2\right)(z)
\frac{\partial^2 H}{\partial r^2}(r,z) + (QH(r,\cdot))(z).
\end{eqnarray}
The operator 
$\cL^c$
is closely related to a generator of the process
$(R(V),Z)$
and will play an important role in the solution of 
the stochastic control problems.

The next lemma describes a class of martingales related to the
chain 
$Z$.

\begin{lemma}
\label{lem:Q_mart}  
Let 
$F:\bbR_+\times\bbR\times\bbE\to\bbR$
be a bounded function,
such that for any 
$z\in\bbE$ 
the restriction to the first two coordinates 
$F(\cdot,\cdot,z):\bbR_+\times\bbR\to\bbR$
is continuous. 
Assume that the generator 
$Q$
satisfies
\begin{equation}
\label{eq:sup_Assum_Q}
\sup\{ -Q(z,z)\>:\> z\in\bbE\}<\infty.
\end{equation}
Let 
$U=(U_t)_{t\geq0}$
be any continuous semimartingale, adapted to the filtration $(\cF_t)_{t\geq0}$.  
Then the process 
$M^{U}=(M^{U}_t)_{t\geq0}$, 
given by 
\begin{eqnarray*} 
M^{U}_t & := & \sum_{0<s\leq t}\left[F(s,U_s,Z_s)-F(s,U_s,Z_{s-})\right]
-\int_0^t(QF(s,U_s,\cdot))(Z_{s-})\,\dd s,
\end{eqnarray*}
is a true $(\cF_t,\PP_z)$-martingale
for any starting point
$z\in\bbE$.
\end{lemma}

\begin{remarks}
\begin{enumerate}[(i)]
\item The key point in Lemma~\ref{lem:Q_mart}
is that we do not assume 
that the process
$(U,Z)$
is Markov, 
since all that is required of 
$U$ 
is that it has continuous paths and is adapted to
the underlying filtration on the original probability space. 
This fact plays a crucial role in the solution of our 
optimisation problems, as it allows us to eliminate 
all the (suboptimal) non-Markovian couplings of 
the Brownian motions 
$V$
and
$B$, the laws
of which are not tractable.
\item Assumption~\eqref{eq:sup_Assum_Q} 
on 
$Q$
is equivalent to stipulating that 
$Q$
is a bounded linear operator. This is clearly
satisfied when the state space
$\bbE$
is finite. 
\item The result in Lemma~\ref{lem:Q_mart} is well-known
but a precise reference appears difficult to find.
For this reason, and because of its importance in the proofs of 
Theorems~\ref{thm:tracking} and~\ref{thm:coupling}, 
a proof of Lemma~\ref{lem:Q_mart}
is given in Appendix~\ref{sec:App_Lemma_Q_mart}.
\end{enumerate}
\end{remarks}

\subsection{$(\cF_t)$-Brownian motion and continuous-time
$(\cF_t)$-Markov chain are independent}
\label{subsec:Indepdnednt_Z_B}
Intuitively, 
the independence of
the chain 
$Z$
and a Brownian motion 
$W\in\cV$
follows from the fact that 
any 
$(\cF_t)$-martingale of the form
$(\psi(Z_t,t))_{t\geq0}$,
where 
$\psi$
is a real function defined on 
the product
$\bbE\times\bbR_+$,
is equal to the sum of its jumps 
minus an absolutely continuous compensator
and therefore 
has constant covariation 
with any continuous semimartingale
adapted to $(\cF_t)_{t\geq0}$. 
The key fact underpinning   
this argument is that 
$Z$
is a Markov process on the filtration 
$(\cF_t)_{t\geq0}$ 
(see Section~\ref{subsection:NOn_MArkov_Chain}
for counterexamples to 
Theorems~\ref{thm:tracking} and~\ref{thm:coupling}
when this assumption is relaxed).

\begin{lemma}
\label{lem:_Zand_W_indep}
An 
$(\cF_t)$-Markov chain 
$Z$
is independent of any 
$(\cF_t)$-Brownian motion 
$W$
in
$\cV$.
\end{lemma}

\begin{proof}
We first show that the random variables 
$W_T$ and $Z_T$ are independent
for any
$T>0$. 
Let the functions
$f:\bbR\to\bbR$ 
and 
$g:\bbE\to\bbR$ 
be bounded and measurable with 
$f$
suitably smooth.
We need to establish the
equality 
$\EE[f(W_T)g(Z_T)]=\EE[f(W_T)]\EE[g(Z_T)]$.
Define the $(\cF_t)$-martingales 
$M^f=(M^f_t)_{t\in[0,T]}$ and  
$N^g=(N^g_t)_{t\in[0,T]}$ 
by
$$M^f_t:=\EE[f(W_T)|\cF_t]\qquad\text{and}\qquad N^g_t:=\EE[g(Z_T)|\cF_t].$$
Note that it is sufficient to prove that 
the product
$M^fN^g=(M^f_tN^g_t)_{t\in[0,T]}$ 
is a martingale since in that case we have
\begin{eqnarray}
\label{eq:Indep_at_T}
\EE[f(W_T)]\EE[g(Z_T)]=M^f_0N^g_0=\EE[M^f_TN^g_T]=\EE[f(W_T)g(Z_T)].
\end{eqnarray}
Now
$M^f_t=(P^W_{T-t}f)(W_t)$,
where $P^W$ is the Brownian semigroup,
and hence
$M^f$
is a continuous martingale. Similarly
we have 
$N^g_t=(P_{T-t}g)(Z_t)$,
where $P$ denotes the semigroup for $Z$,
and hence It\^o's lemma for general semimartingales~\cite[Sec~II.7,~Thm.~33]{Protter:05}
and the Kolmogorov backward equation imply
$\dd N^g_t=(P_{T-t}g)(Z_t)-(P_{T-t}g)(Z_{t-})-(Q(P_{T-t}g))(Z_t)\dd t$
($Q$ 
denotes the generator matrix for 
$Z$).
In particular, the quadratic variation of 
$N^g$
is equal to the sum of its jumps, i.e.
the continuous part of the process
$[N^g,N^g]$ is almost surely zero. Hence the continuity of 
$M^f$
and~\cite[Sec~II.6,~Thm.~28]{Protter:05}
imply that the covariation satisfies
$\dd [M^f,N^g]_t=0$.
Therefore, by the product rule, 
the infinitesimal increment  
of the process
$M^fN^g$
equals
$$\dd (M^f_tN^g_t)=N^g_{t-}\dd M^f_t+M^f_{t-}\dd N^g_t+\dd [M^f,N^g]_t =N^g_t\dd M^f_t+M^f_t\dd N^g_t$$
(the subscripts $t-$ can be change to $t$ since $M^f$ is continuous),
making 
$M^fN^g$
a martingale, since both
$M^f$ and $N^g$ are bounded martingales, 
and equality~\eqref{eq:Indep_at_T} follows.
By an approximation argument and
the Dominated Convergence Theorem
we conclude that~\eqref{eq:Indep_at_T} 
holds for arbitrary 
bounded measurable functions
$f$
and
$g$
and 
the independence of 
$W_T$
and
$Z_T$
follows.

To prove independence of random vectors
$(W_{t_1},\ldots, W_{t_n})$ 
and 
$(Z_{t_1},\ldots, Z_{t_n})$
for any
$n\in\bbN$
and a sequence of times
$0=t_0<t_1<\cdots< t_n$,
pick any bounded measurable functions 
$f:\bbR^n\to \bbR$
and
$g:\bbE^n\to \bbR$
and define 
recursively 
the functions 
$f_k:\bbR^{k\vee1}\to \bbR$
and
$g_k:\bbE^{k\vee1}\to \bbR$
for 
$k=n,\ldots,0$,
which are again 
bounded and measurable, 
by
$f_n:=f, g_n:=g$ 
and
$$f_{k-1}(W_{t_1},\ldots, W_{t_{k-1}}):=\EE[f_k(W_{t_1},\ldots, W_{t_k})|\cF_{t_{k-1}}], \quad
g_{k-1}(Z_{t_1},\ldots, Z_{t_{k-1}}):=\EE[g_k(Z_{t_1},\ldots, Z_{t_k})|\cF_{t_{k-1}}].$$
Note that
$f_0$
and
$g_0$
are constant functions.
Equality~\eqref{eq:Indep_at_T}
applied to the bounded measurable functions
$x\mapsto f(W_{t_1},\ldots, W_{t_{n-1}},x)$
and
$z\mapsto g(Z_{t_1},\ldots, Z_{t_{n-1}},z)$
shows that the following conditional expectation factorises:
$$\EE[f(W_{t_1},\ldots, W_{t_n})g(Z_{t_1},\ldots, Z_{t_n})|\cF_{t_{n-1}}]=f_{n-1}(W_{t_1},\ldots, W_{t_{n-1}})g_{n-1}(Z_{t_1},\ldots, Z_{t_{n-1}}).$$
Therefore, by iteration and the tower property,
we see that the following holds
$$\EE[f(W_{t_1},\ldots, W_{t_n})g(Z_{t_1},\ldots, Z_{t_n})]=f_0g_0=\EE[f(W_{t_1},\ldots, W_{t_n})]\EE[g(Z_{t_1},\ldots, Z_{t_n})].$$
Since 
$f$
and
$g$
were arbitrary,
the processes 
$W$
and
$Z$
are independent. 
\end{proof}

It follows from
Lemma~\ref{lem:_Zand_W_indep}
that an $(\cF_t)$-adapted volatility process, 
given by a strong solution of an SDE, cannot be
approximated pathwise by a 
continuous-time $(\cF_t)$-Markov chain.

\begin{corollary}
Let 
$Z'$
be an 
$(\cF_t)$-adapted 
Feller semimartingale, which solves a scalar 
SDE with Lipschitz drift
and diffusion coefficients
$\mu, \sigma$
such that
$\sigma>c>0$.
Then there exists no sequence of 
continuous-time $(\cF_t)$-Markov chains 
that converges 
to
$Z'$
almost surely on compacts. 
\end{corollary}

\begin{proof}
The process
$W=(W_t)_{t\geq0}$,
where
$W_t:=\int_0^t(\dd Z'_s-\mu(Z'_s)\dd t)/\sigma(Z'_s)$, %-\int_0^t(\mu(Z'_s)/\sigma(Z'_s))\dd t$,
is an $(\cF_t)$-adapted
continuous local martingale
with
$[W,W]_t=t$. $W$ is therefore an
$(\cF_t)$-Brownian motion 
(by L\'evy's characterisation theorem)
and 
$Z'$
is a strong solution of the SDE
$\dd Z'_t = \mu(Z'_t)\dd t+\sigma(Z'_t)\dd W_t$.
By Lemma~\ref{lem:_Zand_W_indep},
any sequence of 
continuous-time $(\cF_t)$-Markov chains is independent 
of 
$W$
and therefore also independent of 
$Z'$.
Therefore, 
since 
$Z'$
is non-deterministic,
the sequence cannot converge to 
$Z'$
almost surely on compacts.
\end{proof}

%========================================================================================
\section{Tracking}
\label{sec:Tracking}

In this section we consider the problem of \textit{tracking}
$X$
by the process 
$Y(V)$,
defined in~\eqref{eq:Proc_Def_X},
where the control is being exercised solely
by choosing the driving Brownian motion $V$.
Recall that the tracking criterion, stated for 
a convex function 
$\phi$ 
%\in\cC^2(\bbR)$
%bounded from below,
%which satisfies 
%$|\phi(x)|\leq a |x|^p$
%for large 
%$x$
in~\eqref{eq:Ass_phi}
%and
%some
%$a>0$,
%$p\in[2,\infty)$.
%More precisely, 
%for
and a time horizon 
$T>0$,
can be equivalently expressed 
in terms of the following problems:
\begin{eqnarray*}
%\text{\textbf{(Td)}}\hspace{5mm} & & 
\text{minimise}\quad \EE\left[\phi(X_T-Y_T(V))\right]\quad\text{over}\quad V\in\cV, & &\\
%\text{\textbf{(Tu)}}\hspace{5mm} & & 
\text{maximise}\quad \EE\left[\phi(X_T-Y_T(V))\right]\quad\text{over}\quad V\in\cV. & &
\end{eqnarray*}

\begin{theorem}
\label{thm:tracking_control}
Let the Brownian motions 
$V^I$
and 
$V^{II}$
be as in~\eqref{eq:Def_VI_VII}.
Assume 
$Z$
%satisfies~\eqref{eq:Int_Assump}.
satisfies~\eqref{eq:Mart_Integrab_Quad_VAr},~\eqref{eq:Z_Chain_Assumption} and~\eqref{eq:Int_Assump}
and that the function $\phi$
is as in~\eqref{eq:Ass_phi}.
Then for any positive
$T$
we have
\begin{eqnarray}
\label{eq:sol_Td}
\inf_{V\in\cV} \EE\left[\phi(X_T-Y_T(V))\right] & = &
\EE\left[\phi(X_T-Y_T(V^I))\right], \\
\label{eq:sol_Tu}
\sup_{V\in\cV} \EE\left[\phi(X_T-Y_T(V))\right] & = &
\EE\left[\phi(X_T-Y_T(V^{II}))\right]. 
\end{eqnarray}
\end{theorem}

In this section we prove 
Theorem~\ref{thm:tracking_control},
which clearly implies
Theorem~\ref{thm:tracking},
and hence 
solves
Problem~\textbf{(T)}. 
The proof 
of Theorem~\ref{thm:tracking_control}
is based on Bellman's 
principle, a martingale verification argument
and an approximation scheme. 
The first stage %in the proof of Theorem~\ref{thm:tracking_control}
consists of ``approximating'' Problems~\eqref{eq:sol_Td}-\eqref{eq:sol_Tu}.
More precisely, we proceed in two steps: we first 
introduce a stopped chain 
$Z^n$
and, in the second step, the stopped process
$R^{K,n}(V)$.

To this end
let 
$U_n\subset\bbR^d$,
$n\in\bbN$,
be a family of compact subsets 
such that 
$\cup_{n\in\bbN}U_n=\bbR^d$
and
$U_n\subset  U^\circ_{n+1}$,
for all 
$n\in\bbN$,
where
$U^\circ_{n+1}$
denotes the interior 
of
$U_{n+1}$
in 
$\bbR^d$.
For each 
$n\in\bbN$,
define a stopping time 
$\tau_n$
and the stopped $(\cF_t)$-Markov chain
$Z^n$
by
\begin{eqnarray}
\label{eq:Stopped_Chain}
Z^n_t:= Z_{t\wedge\tau_n},& &
\text{where}\qquad
\tau_n:=\inf\{t\geq0\,:\,Z_t\in\bbE\setminus U_n\}
\qquad (\inf \emptyset = \infty).
\end{eqnarray}
Hence,
$Z^n$
is an $(\cF_t)$-Markov chain with the state space
$\bbE$
and
a $Q$-matrix 
$Q_n$
given by
\begin{eqnarray}
\label{eq:Stopped_Chain_Gen}
Q_n(z,z') & = & I_{U_n}(z) Q(z,z'),\qquad z,z'\in\bbE,
\end{eqnarray}
where
$I_{\{\cdot\}}$
denotes the indicator function.
In particular, 
since 
$U_n$
is compact and hence
$U_n\cap\bbE$
must be finite 
by~\eqref{eq:Z_Chain_Assumption},
$Q_n$
satisfies assumption~\eqref{eq:sup_Assum_Q}  
in Lemma~\ref{lem:Q_mart}.
Since the chain 
$Z$
has %no instantaneous states by assumption,
c\`adl\`ag paths,
the sequence of positive random variables 
$(\tau_n)_{n\in\bbN}$
is non-decreasing and the following holds
\begin{equation*}
\tau_\infty:=\lim_{n\to\infty}\tau_n=\infty\qquad \text{$\PP_z$-a.s.}\qquad\text{for any $z\in\bbE$.}
\end{equation*}
Hence, we can extend the definition in~\eqref{eq:Stopped_Chain}
in a natural way to the case
$n=\infty$
by
$Z^\infty:=Z$.

Fix a large
$K>0$
and 
%in Lemma~\ref{lem:R^C_rep}
%and 
%$K>0$,
define,
for any  
$V\in\cV$,
the stopping time
$$
\tau^K(V):=\inf\{t\geq0:|R_t(V)|\geq K\} \qquad (\inf \emptyset = \infty),
$$
where 
$R(V)$
is given in~\eqref{eq:R^V_Rep}.
The stopped process of interest
$R^{K,n}(V)=(R^{K,n}_t(V))_{t\geq0}$
can now be defined by
\begin{equation}
\label{eq:Def_process_K_n}
R^{K,n}_t(V):= R_{t\wedge\tau_n\wedge \tau^K(V)}(V).  
\end{equation}
%We first establish the Markov property 
%for the processes 
%$(R^K(V^I),Z)$
%and
%$(R^K(V^{II}),Z)$.
%This simple but important fact will
%allow us to express the Bellman process
%in Problems~\eqref{eq:sol_Td}
%and~\eqref{eq:sol_Tu}
%as a deterministic function of the current 
%level of the process
%$(R(V),Z)$. 

For 
given
$\phi$
satisfying~\eqref{eq:Ass_phi},
$T>0$
%as in 
%Section~\ref{sec:i}
and any
$K\in(0,\infty)$
and
$n\in\bbN\cup\{\infty\}$,
consider the problems
\begin{eqnarray}
\label{eq:sol_Td_K}
%\text{\textbf{(Ia)}}^K\hspace{5mm} 
  \text{minimise}\quad \EE\left[\phi(R^{K,n}_T(V))\right]\quad\text{over}\quad V\in\cV, &  & \\
\label{eq:sol_Tu_K}
%\text{\textbf{(Ib)}}^K\hspace{5mm} 
\text{maximise}\quad
\EE\left[\phi(R^{K,n}_{T}(V))\right]\quad\text{over}\quad V\in\cV. & &
\end{eqnarray}
By Lemma~\ref{lem:_Zand_W_indep},
the processes
$(R(V^I),Z)$
and
$(R(V^{II}),Z)$
are Markov. 
%and hence so are the stopped process
%$(R^{K,n}(V^I),Z)$
%and
%$(R^{K,n}(V^{II}),Z)$.
%In the case of the Brownian motion
%$V^I$
%defined in 
%Theorem~\ref{thm:tracking},
%the stochastic correlation process
%in Lemma~\ref{lem:R^V_rep}
%takes the form
%$C^I_t=\hat c(Z_t)$.
Therefore we can define the candidate value functions
$\psi^{(I)}_{K,n},\psi^{(II)}_{K,n}:\bbR\times\bbE\times[0,T]\to \bbR_+$
%and
%$\psi^{(II)}_K:\bbR\times\bbE\to \bbR_+$
for Problems~\eqref{eq:sol_Td_K}
and~\eqref{eq:sol_Tu_K}
by
\begin{eqnarray}
\label{eq:Val_Fun_T}
\psi_{K,n}^{(I)}(r,z,t):=\EE_{r,z}\left[\phi(R^{K,n}_t(V^I))\right]
\qquad\text{and}\qquad 
\psi_{K,n}^{(II)}(r,z,t):=\EE_{r,z}\left[\phi(R^{K,n}_t(V^{II}))\right],
%\psi_K^{(II)}(r,z)&:=&\EE_{r,z}\left[\phi(R^K_{e_q}(V^I))\right]
%=\EE_{r,z}\left[q\int_0^\infty\te^{-qt}\phi(R^K_t(V^I))\,\dd t\right].
%\label{eq:Val_Fun_IK}
\end{eqnarray}
respectively.
Note that by definition we have
$\psi_{K,n}^{(I)}(r,z,t)=\psi_{K,n}^{(II)}(r,z,t)=\phi(r)$
if 
$r\in\bbR\setminus(-K,K)$
or 
$z\in\bbR\setminus U_n$.

\begin{lemma}
\label{lem:Key_Props_I}
Assume that
$\phi$,
given in~\eqref{eq:Ass_phi},
is bounded from below and
$\phi\in\cC^2(\bbR)$.
For any
$K\in(0,\infty)$
and
$n\in\bbN\cup\{\infty\}$,
the functions 
$\psi_{K,n}^{(I)}$
and
$\psi_{K,n}^{(II)}$,
defined in~\eqref{eq:Val_Fun_T},
have the following properties.
%be a continuous-time Markov chain with the 
%on the state space 
%$\bbE$.
%The followign properties hold.
\begin{enumerate}[(i)]
\item 
\label{Lem:Key_Props_I_i}
For all 
$r\in\bbR$,
$z\in\bbE$
and
$t\in[0,T]$,
there exists a constant 
$\ell\in\bbR$,
such that 
%we have
$$
\ell\leq \psi_{K,n}^{(I)}(r,z,t),
\psi_{K,n}^{(II)}(r,z,t)
\leq \max\{\phi(\max\{K,r\}),\phi(\min\{-K,r\})\}.
$$
\item For each 
$z\in\bbE$ we have
$\psi_{K,n}^{(I)}(\cdot,z,\cdot), \psi_{K,n}^{(II)}(\cdot,z,\cdot)\in\cC^{2,1}(\bbR\times(0,T])$.
%and
%$\psi_K^{(II)}(\cdot,z)\in\cC^{2}(\bbR)$.
\label{Lem:Key_Props_I_ii}
\item 
\label{Lem:Key_Props_I_iii}
For any
$r\in\bbR$,
$z\in\bbE$
and
$t\in(0,T]$,
the derivatives satisfy the
following 
inequalities:
\begin{eqnarray}
\label{eq:Firs_Der_Bound_I}
\left|\frac{\partial \psi_{K,n}^{(I)}}{\partial r}\right|(r,z,t),\>\>
\left|\frac{\partial \psi_{K,n}^{(II)}}{\partial r}\right|(r,z,t)
& \leq & \max\{\phi'(\max\{K,r\}),-\phi'(\min\{-K,r\})\},\\
\frac{\partial^2 \psi_{K,n}^{(I)}}{\partial r^2}(r,z,t),\>\>
\frac{\partial^2 \psi_{K,n}^{(II)}}{\partial r^2}(r,z,t) & \geq & 0.
\label{eq:Sec_Der_Bound_I}
%\geq 0;\\
%\leq \max\{\phi'(K),-\phi'(-K),1\} & \text{and}& 
\end{eqnarray}
\end{enumerate}
\end{lemma}

\begin{proof}
Part~\eqref{Lem:Key_Props_I_i}
follows from~\eqref{eq:Val_Fun_T}
and the properties of 
$\phi$.
To prove that  
$\psi_{K,n}^{(I)}$
is differentiable in $r$,
define 
$S:=R^{K,n}_t(V^I)- R^{K,n}_0(V^I)$
and note that its distribution does not depend on the starting
point of $R^{K,n}(V^I)$.
Since 
$\phi \in\cC^2(\bbR)$,
Lagrange's mean value theorem implies that, 
for any small
$h>0$,
there exists a random variable 
$\xi_{S,h}$
such that 
\begin{eqnarray}
\label{eq:Lagrange}
\phi(r+h+S)-\phi(r+S)=h\phi'(r+\xi_{S,h})\qquad\text{and}\qquad
\xi_{S,h}\in(S,h+S).
\end{eqnarray}
Since 
$|S|\leq K$
almost surely and 
$r$
is fixed, the continuity of 
$\phi'$
yields that the random variable 
$|\phi'(r+\xi_{S,h})|$
is bounded above by a constant. 
Equation~\eqref{eq:Lagrange},
almost sure convergence 
of
$\xi_{S,h}$
to
$S$,
as
$h\to0$,
and the Dominated Convergence Theorem imply that 
$\psi_{K,n}^{(I)}(\cdot,z,t)$
is differentiable in 
$r$
and
\begin{eqnarray}
\label{eq:psi_I_prime}
\frac{\partial \psi_{K,n}^{(I)}}{\partial r}(r,z,t)= \EE_{r,z}\left[\phi'(R^{K,n}_t(V^I))\right].
\end{eqnarray}
Furthermore, the convexity of  
$\phi$
and~\eqref{eq:psi_I_prime}
yield the first inequality in~\eqref{eq:Firs_Der_Bound_I}.
An identical argument applied to the function	
%completely analogous argument implies that
$\psi_{K,n}^{(II)}(\cdot,z,t)$
implies its
differentiability in 
$r$
and yields~\eqref{eq:Firs_Der_Bound_I}. 
%the formula
%\begin{eqnarray}
%\label{eq:psi_II_prime}
%\frac{\partial \psi_K^{(II)}}{\partial r}(r,z)= \EE_{r,z}\left[\phi'(R^K_{t}(V^{II}))\right],
%\end{eqnarray}
%which in turn proves~\eqref{eq:Firs_Der_Bound_I}.

Since 
$\phi''$
is continuous by assumption, we can apply an analogous argument to 
the one above, 
now using formula~\eqref{eq:psi_I_prime} %--\eqref{eq:psi_II_prime}
instead of~\eqref{eq:Val_Fun_T},
to conclude that the functions
$\psi_{K,n}^{(I)}(\cdot,z,t)$
and
$\psi_{K,n}^{(II)}(\cdot,z,t)$
are in
$\cC^2(\bbR)$
with 
$$
\frac{\partial^2 \psi_{K,n}^{(I)}}{\partial r^2}(r,z,t)= \EE_{r,z}\left[\phi''(R^{K,n}_t(V^I))\right],\qquad
\frac{\partial^2 \psi_{K,n}^{(II)}}{\partial r^2}(r,z,t)= \EE_{r,z}\left[\phi''(R^{K,n}_{t}(V^{II}))\right].
$$
The convexity of  
$\phi$
now implies part~\eqref{Lem:Key_Props_I_iii}
of the lemma.
Differentiability of  
$\psi_{K,n}^{(I)}(r,z,\cdot)$
in 
$t$
follows from the smoothness of
$\phi$
and the standard properties of It\^o integrals.
\end{proof}

Pick a function
$F:\bbR\times\bbE\times[0,T)\to\bbR$
such that 
$F(\cdot,z,\cdot)\in\cC^{2,1}(\bbR\times[0,T))$
for each 
$z\in\bbE$,
and
for each
$r\in\bbR$,
$t\in[0,T)$
the restriction to the second coordinate
$F(r,\cdot,t):\bbE\to\bbR$
is bounded.
Then for any constant
$c\in[-1,1]$
we define the function
$\cK^cF:\bbR\times\bbE\times[0,T)\to\bbR$
by the formula:
\begin{eqnarray*}
(\cK^cF)(r,z,t) & = & (\cL^c F(\cdot,\cdot,t))(r,z) +
\frac{\partial F}{\partial t}(r,z,t),
\end{eqnarray*}
where the operator  
$\cL^c$
is as defined in~\eqref{eq:Def_cL}.

\begin{lemma}[HJB equation]
\label{lem:HJB_I}
Let
$\phi$
in~\eqref{eq:Ass_phi}
be bounded from below and
satisfy
$\phi\in\cC^2(\bbR)$.
Let
$n\in\bbN$
and
$K\in(0,\infty)$.
Then the functions 
$$F^{(I)}(r,z,t):=\psi_{K,n}^{(I)}(r,z,T-t)\qquad\text{and}\qquad
F^{(II)}(r,z,t):=\psi_{K,n}^{(II)}(r,z,T-t),$$
(see~\eqref{eq:Val_Fun_T}
for the definition of 
$\psi_{K,n}^{(I)}$
and
$\psi_{K,n}^{(II)}$)
satisfy the HJB equations:

for any triplet 
$(r,z,t)\in(-K,K)\times\left(\bbE\cap U_n\right)\times[0,T)$
(see~\eqref{eq:Stopped_Chain} for the role of the set
$U_n$)
we have
\begin{eqnarray}
\label{eq:HJB_IK}
\inf_{c\in[-1,1]}\left(\cK^c F^{(I)}\right)(r,z,t) & = & 0,\\
%\qquad \text{on $(-K,K)\times\bbE\times[0,T)$ with
%$F^{(I)}(r,z,T)=\phi(r)$},\\
\sup_{c\in[-1,1]}\left(\cK^c F^{(II)}\right)(r,z,t) & = & 0.
%\qquad \text{on $(-K,K)\times\bbE\times[0,T)$ with
%$F^{(II)}(r,z,T)=\phi(r)$}.
\label{eq:HJB_IIK}
\end{eqnarray}
Furthermore,  if
at least one of the conditions
$|r|\geq K$
or
$z\in\bbE\setminus U_n$
or
$t=T$
is satisfied,
we have
\begin{eqnarray}
\label{eq:Tracking_Boundary_Beh}
F^{(I)}(r,z,t)=F^{(II)}(r,z,t)=\phi(r).
\end{eqnarray}
\end{lemma}

\begin{remark}
Unlike Lemma~\ref{lem:Key_Props_I},
the proof of Lemma~\ref{lem:HJB_I}
depends on Lemma~\ref{lem:Q_mart}
and so requires the assumption 
$n<\infty$.
\end{remark}

\begin{proof}
Note first that the definitions in~\eqref{eq:Val_Fun_T} 
imply the boundary behaviour stated in~\eqref{eq:Tracking_Boundary_Beh}.

We now focus on the proof of~\eqref{eq:HJB_IK}.
Recall that 
for any starting point 
$z\in\bbE$
and 
$t\in[0,T)$,
on the event
$\{\tau_n\geq t\}$
we have
$Z^n_t=Z_t$.
The Markov property of the process
$(R(V^I),Z)$
and the equality in~\eqref{eq:Tracking_Boundary_Beh}
now imply
\begin{eqnarray*}
\EE\left[\phi(R^{K,n}_T(V^{I}))\vert \cF_t\right] 
& = & 
\EE\left[\phi(R^{K,n}_T(V^{I}))I_{\{\tau_n<t\}} \vert \cF_t\right] +
\EE\left[\phi(R^{K,n}_T(V^{I}))I_{\{\tau_n\geq t\}} \vert \cF_t\right] \\
& = & \phi(R^{K,n}_{\tau_n}(V^{I}))I_{\{\tau_n<t\}}  +
\psi_{K,n}^{(I)}(R^{K,n}_t(V^{I}),Z^n_t,T-t)I_{\{\tau_n\geq t\}} \\
& = & 
\psi_{K,n}^{(I)}(R^{K,n}_t(V^{I}),Z^n_t,T-t).
\end{eqnarray*}
The following observations are key:
\begin{enumerate}
\item[$\bullet$] the quadratic covariation
$[R^{K,n}(V^I),Z^{n,i}]_t$
vanishes for all $t\geq0$
and
$i=1,\ldots,d$,
where
$Z^{n,i}$
is the 
$i$-th component of 
$Z^n$
(recall that we are assuming 
$\bbE\subset\bbR^d$);
\item[$\bullet$]
the chain 
$Z^n$
satisfies the assumptions
of Lemma~\ref{lem:Q_mart}
and hence the process 
$M^{U}=(M^{U}_t)_{t\in[0,T]}$, 
given by 
\begin{eqnarray*} 
M^{U}_t & := & \sum_{0<s\leq t}\left[\psi_{K,n}^{(I)}(R^{K,n}_s(V^I),Z^n_s,T-s)-\psi_{K,n}^{(I)}(R^{K,n}_s(V^I),Z^n_{s-},T-s)\right]\\
& & -\int_0^t(Q_n\psi_{K,n}^{(I)}(R^{K,n}_s(V^I),\cdot,T-s))(Z^n_{s-})\,\dd s,
\end{eqnarray*}
where 
$Q_n$
is the generator of the chain
$Z^n$
given in~\eqref{eq:Stopped_Chain_Gen},
is a true $(\cF_t,\PP_z)$-martingale
for any starting point
$z\in\bbE$.
\end{enumerate}
By Lemma~\ref{lem:Key_Props_I},
the function 
$\psi_{K,n}^{(I)}$
possesses the necessary smoothness so that 
It\^o's lemma for general semimartingales~\cite[Sec~II.7,~Thm.~33]{Protter:05}
can be applied to the process 
$(\psi_{K,n}^{(I)}(R^{K,n}_t(V^{I}),Z^n_t,T-t))_{t\in[0,T]}$,
which is itself a bounded martingale.
Since 
$Q_n(z,z')=Q(z,z')$
for any
$z\in\bbE\cap U_n$,
$z'\in\bbE$
and
on the event
$\{t\leq\tau_n\}$
we have
$Z_t=Z^n_t\in U_n$,
the pathwise representation of this bounded martingale 
implies that the following process
$N=(N_t)_{t\in[0,T]}$,
\begin{eqnarray*}
N_t & = & \int_0^{t\wedge \tau_n\wedge\tau^K(V)}\left[\frac{1}{2}(|\sigma_1|-|\sigma_2|)^2(Z_s) 
\frac{\partial^2 \psi_{K,n}^{(I)}}{\partial r^2}(R^{K,n}_s(V^I),Z_s,T-s)\right.\\
&  & \qquad + \left. (Q\psi_{K,n}^{(I)}(R^{K,n}_s(V^I),\cdot,T-s))(Z_s)-
\frac{\partial \psi_{K,n}^{(I)}}{\partial t}(R^{K,n}_s(V^I),Z_s,T-s)
\right]\,\dd s, 
\end{eqnarray*}
is a continuous martingale. 
The quadratic variation of
$N$
is clearly equal to zero
and hence
$N_t=0$
for all 
$t\in[0,T]$
and starting
points
$(r,z)$.
For any
$z\in\bbE\cap U_n$
we have 
$\PP_z[Z_t=z,\> \forall t\leq T]>0$.
On this event the following holds:
$\tau_n\geq T$
$\PP_z$-a.s.
and the process
$R^{K,n}(V^{I})$
is by~\eqref{eq:Def_VI_VII},~\eqref{eq:R^V_Rep}
and~\eqref{eq:Def_process_K_n}  
either  equal to the constant 
$r$
(if $\sigma_1(z)=\sigma_2(z)$)
or a Brownian motion stopped when it exits 
$(-K,K)$.
Since,  
with positive probability,
Brownian motion visits a neighbourhood of any
point in 
$(-K,K)$
and stays in this interval until time 
$T$,
the fact that 
$N_t=0$
for all 
$t\in[0,T]$
and starting
points
$(r,z)$
implies
%the equality 
\begin{eqnarray}
\label{eq:PDE_HJB_Tracking_I}
& &
\frac{1}{2}(|\sigma_1|-|\sigma_2|)^2(z)
\frac{\partial^2 \psi_{K,n}^{(I)}}{\partial r^2}(r,z,T-t)
+(Q\psi_{K,n}^{(I)} (r,\cdot,T-t))(z)-
\frac{\partial\psi_{K,n}^{(I)}}{\partial t}(r,z,T-t)  =  0
\end{eqnarray}
for all 
$(r,z,t)\in(-K,K)\times\left(\bbE\cap U_n\right)\times[0,T)$.

To prove~\eqref{eq:HJB_IK}, 
observe that 
$(|\sigma_1|-|\sigma_2|)^2 =
 \inf_{c\in[-1,1]} (\sigma_1^2-2c\sigma_1\sigma_2+\sigma_2^2)$.
Then~\eqref{eq:Sec_Der_Bound_I}
of Lemma~\ref{lem:Key_Props_I}
implies that 
\begin{eqnarray*}
(\sigma_1^2-2c\sigma_1\sigma_2+\sigma_2^2)(z)
\frac{\partial^2\psi_{K,n}^{(I)}}{\partial r^2}(r,z,T-t) 
& \geq &
(|\sigma_1|-|\sigma_2|)^2(z) 
\frac{\partial^2\psi_{K,n}^{(I)}}{\partial r^2}(r,z,T-t)
\end{eqnarray*}
for any
$c\in[-1,1]$
and each 
$(r,z,t)\in(-K,K)\times\left(\bbE\cap U_n\right)\times[0,T)$.
This inequality
%the definition of
%$\cL^c\zeta^{(II)}$
%in~\eqref{eq:Def_cL}
and identity~\eqref{eq:PDE_HJB_Tracking_I}
imply~\eqref{eq:HJB_IK}.
The proof of~\eqref{eq:HJB_IIK}
is analogous and therefore left to the reader.
\end{proof}

\subsection{Proof of Theorem~\ref{thm:tracking_control}} 
\label{subsec:Proof_TRacking}
Assume that 
$\phi$
satisfies condition~\eqref{eq:Ass_phi}
as well as
\begin{eqnarray}
\label{eq:Add_Cond_on_phi}
%\exists \ell\in\bbR\text{ such that } 
\ell \leq \phi(x) \text{ $\forall x\in\bbR$, $\ell\in\bbR$,} \qquad\text{and}\qquad
\phi\in\cC^2(\bbR).
\end{eqnarray}
%twice continuously differentiable
Pick 
$V\in\cV$
and,
for any
$t\in[0,T]$,
define Brownian motions
$V^{It}=(V_s^{It})_{s\geq0}\in\cV$ and $V^{IIt}=(V_s^{IIt})_{s\geq0}\in\cV$
by
\begin{eqnarray}
\label{eq:Bellman_BMs}
V_s^{It}  :=  \left\{ \begin{array}{ll}
V_s  & \textrm{ if $s\leq t$,}\\
V_t + V^I_{s}-V^I_{t} & \textrm{ if $s>t$,}
\end{array} \right.
&\text{ and } &
V_s^{IIt}  :=  \left\{ \begin{array}{ll}
V_s  & \textrm{ if $s\leq t$,}\\
V_t + V^{II}_{s}-V^{II}_{t} & \textrm{ if $s>t$,}
\end{array} \right.
\end{eqnarray}
where 
$V^I,V^{II}$
are given in~\eqref{eq:Def_VI_VII}.
In other words, for each 
$t\geq0$,
the Brownian motions 
$V^{It}$ and $V^{IIt}$
are arbitrary (but fixed) up to time
$t$  
and have increments 
equal to those of the candidate optimal 
Brownian motions after this time. 
We now consider two Bellman processes
$(B^I_t(V))_{t\in[0,T]}$
and
$(B^{II}_t(V))_{t\in[0,T]}$,
associated to Problems~\eqref{eq:sol_Td_K}-\eqref{eq:sol_Tu_K},
given by
\begin{eqnarray}
\label{eq:Bellman_processes_TRacking}
B^I_t(V) := \psi_{K,n}^{(I)}(R^{K,n}_t(V),Z^n_t,T-t)
&\text{and} &
B^{II}_t(V) := \psi_{K,n}^{(II)}(R^{K,n}_t(V),Z^n_t,T-t).
\end{eqnarray}

The definitions in~\eqref{eq:Def_VI_VII}
of
$V^I,V^{II}$,
together with Lemma~\ref{lem:_Zand_W_indep},
imply that the processes
$(R(V^I),Z)$
and
$(R(V^{II}),Z)$
are Markov. 
The definition of the Brownian motion
$V^{It}$ %and $V^{IIt}$
in~\eqref{eq:Bellman_BMs}
and the properties of the function
$\psi_{K,n}^{(I)}$
therefore imply
\begin{eqnarray*}
\EE\left[\phi(R^{K,n}_T(V^{It}))\vert \cF_t\right] 
& = & 
\EE\left[\phi(R^{K,n}_T(V^{It}))I_{\{\tau_n<t\}} \vert \cF_t\right] +
\EE\left[\phi(R^{K,n}_T(V^{It}))I_{\{\tau_n\geq t\}} \vert \cF_t\right] \\
& = & \phi(R^{K,n}_{\tau_n}(V))I_{\{\tau_n<t\}}  +
\psi_{K,n}^{(I)}(R^{K,n}_t(V),Z^n_t,T-t)I_{\{\tau_n\geq t\}} \\
& = & 
\psi_{K,n}^{(I)}(R^{K,n}_t(V),Z^n_t,T-t).
\end{eqnarray*}
This equality, together with a 
similar argument based on the definitions
of
$V^{IIt}$ 
and 
$\psi_{K,n}^{(II)}$,
yields the following representations 
for the Bellman processes
\begin{eqnarray*}
%\label{eq:Bellman_processes_TRacking_Markov}
B^I_t(V) = \EE\left[\phi(R^{K,n}_T(V^{It}))\vert \cF_t\right]
&\text{and} &
B^{II}_t(V) = \EE\left[\phi(R^{K,n}_T(V^{IIt}))\vert \cF_t\right].
%B^I_t(V) = \psi_{K,n}^{(I)}(R^{K,n}_t(V),Z^n_t,T-t)
%&\text{and} &
%B^{II}_t(V) = \psi_{K,n}^{(II)}(R^{K,n}_t(V),Z^n_t,T-t).
\end{eqnarray*}
By Lemma~\ref{lem:Key_Props_I}
we can apply It\^o's formula for general semimartingales 
(see~\cite[Sec~II.7,~Thm.~33]{Protter:05}) to 
$B^I(V)$
and
$B^{II}(V)$.
Lemma~\ref{lem:Q_mart}
and inequalities~\eqref{eq:Firs_Der_Bound_I}
imply that the local martingale parts of these path
decompositions of processes
$B^I(V)$
and
$B^{II}(V)$
are true martingales. 
Therefore,
the fact that the quadratic covariation
$[R^{K,n}(V^{It}),Z^{n,i}]_t$
vanishes for all $t\geq0$
for
each component 
$Z^{n,i}$
of
$Z^n$,
together with 
Lemma~\ref{lem:HJB_I},
implies that,
for 
any
$V\in\cV$,
$B^I(V)$
is a submartingale and
$B^{II}(V)$
a supermartingale. 
Furthermore it follows from
the discussion above and
Lemma~\ref{lem:HJB_I}
that
$B^I(V^{I})$
and
$B^{II}(V^{II})$
are martingales.
This establishes the Bellman principle and solves
the optimisation problems in~\eqref{eq:sol_Td_K}
and~\eqref{eq:sol_Tu_K}.
Put differently, we have established the following inequalities 
for 
%all
%$V\in\cV$,
any starting points 
$r\in\bbR$,
$z\in\bbE$,
any
$K\in(0,\infty)$,
$n\in\bbN$
and all Brownian motions 
$V\in\cV$:
\begin{eqnarray}
\label{eq:Inequalities_K_n}
\EE_{r,z}\left[\phi(R^{K,n}_T(V^{I}))\right] \hspace{3mm}
\leq  
\hspace{3mm} \EE_{r,z}\left[\phi(R^{K,n}_T(V))\right] \hspace{3mm}
\leq \hspace{3mm}\EE_{r,z}\left[\phi(R^{K,n}_T(V^{II}))\right]
\end{eqnarray}

The next step in the proof of Theorem~\ref{thm:tracking_control}
requires two limiting arguments.
First, note that 
for any Brownian motion
$V\in\cV$
the definition of the process
$R^{K,n}_T(V)$
in~\eqref{eq:Def_process_K_n}
implies 
$$R^{K,\infty}_T(V)=\lim_{n\uparrow\infty}R^{K,n}_T(V)\quad
\text{$\PP_{r,z}$-a.s.}
$$
for any starting points
$r\in\bbR$
and
$z\in\bbE$.
Furthermore,
by Lemma~\ref{lem:Key_Props_I}~\eqref{Lem:Key_Props_I_i},
the random variables 
$\phi(R^{K,n}_T(V))$
are bounded in modulus by a constant uniformly in
$n\in\bbN$.
Therefore, the Dominated Convergence Theorem implies
that the inequalities in~\eqref{eq:Inequalities_K_n}
hold for 
$n=\infty$.

%for any fixed
%$V\in\cV$,
%$K\in(0,\infty)$
%and starting points 
%$r\in\bbR$,
%$z\in\bbE$.

For the second limiting argument, recall that 
$P$
denotes the semigroup of 
$Z$
and note 
first that the following inequalities hold
for any
$z\in\bbE$,
$t\in[0,T]$ 
and a non-negative function 
$f$:
\begin{eqnarray}
\nonumber
P_Tf(z) & = & \sum_{z'\in\bbE}P_{T-t}(z,z')P_t(z',y)f(y)
\geq P_{T-t}(z,z) P_t f(z)\\
&\geq &
\exp((T-t)Q(z,z)) P_t f(z)\geq
\exp(TQ(z,z)) P_t f(z),
\label{eq:Holding_Time_Estimate}
\end{eqnarray}
since the probability 
$\PP_z[Z_{T-t}=z]=P_{T-t}(z,z)$
%that the chain started at 
%$z$
%is at 
%$z$
%at time $(T-t)$
is greater than the probability that the exponential holding 
time at 
$z$ of the chain $Z$
is bigger than 
$T-t$.
Hence, 
by assumption~\eqref{eq:Int_Assump},
for the function
$f:=|\sigma_1|^p+|\sigma_2|^p:\bbE\to[0,\infty)$
and
$p\in[2,\infty)$ as in~\eqref{eq:Ass_phi}, 
we have
%and the semigroup property imply
%$P_tf(z)<\infty$
%for all
%$t\in[0,T]$,
%$z\in\bbE$,
%where
%$f:=|\sigma_1|^p+|\sigma_2|^p:\bbE\to[0,\infty)$
%and
%$p\in[2,\infty)$. 
%The Kolmogorov backward equation
%%for any 
%%$z\in\bbE$,
%implies for any $z\in\bbE$,
%\begin{eqnarray*}
%(P_Tf)(z)-f(z) &=& \int_0^T\sum_{z'\in\bbE}Q(z,z')(P_tf)(z')\dd t\\ 
%&=&
%\int_0^TQ(z,z)(P_tf)(z)\dd t+ 
%\int_0^T\sum_{z'\in\bbE\setminus\{z\}}Q(z,z')(P_tf)(z')\dd t\\
%& = & 
%\int_0^TQ(z,z)(P_tf)(z)\dd t +
%\sum_{z'\in\bbE\setminus\{z\}}Q(z,z')\int_0^T(P_tf)(z')\dd t\\
%& = &  \sum_{z'\in\bbE}Q(z,z')\int_0^T(P_tf)(z')\dd t,
%%\quad\text{for $z\in\bbE$,}
%\end{eqnarray*}
%where the second equality follows from the following facts:
%$Q(z,z')\in[0,\infty)$
%for all
%$z'\in\bbE\setminus\{z\}$,
%$P_tf(z)\in[0,\infty)$
%for all
%$t\in[0,T]$,
%$z\in\bbE$,
%and Fubini's theorem for non-negative functions.
%This, 
%together with 
%assumption~\eqref{eq:Int_Assump}
%and another application of Fubini's theorem, 
%yields
\begin{eqnarray}
\label{eq:Int_Assump_Chain}
\EE_z \int_0^T\left(|\sigma_1|^p+|\sigma_2|^p\right)(Z_t)\,\dd t<\infty 
\qquad\text{for $z\in\bbE$.}
\end{eqnarray}
Furthermore, it is clear 
from the definition of 
$R^{K,\infty}(V)$,
for any
$V\in\cV$,
that 
$$\lim_{K\to\infty} \phi(R^{K,\infty}_T(V))= \phi(R_T(V))
\quad\text{$\PP_{r,z}$-a.s.}
$$
The following almost sure inequality is 
a direct consequence 
of the definition in~\eqref{eq:Def_process_K_n}
\begin{equation}
\label{eq:S_sup_max}
-S\leq R^{K,\infty}_T(V)\leq S
\qquad\text{for all $K>0$, where $S:=\sup_{t\in[0,T]}|R_t(V)|$.}
\end{equation}
By assumptions~\eqref{eq:Ass_phi}
and~\eqref{eq:Add_Cond_on_phi}
%there exist
%$a,b>0$
%such that
%$|\phi(x)|<a |x|^p+b$
%for all
%$x\in\bbR$,
%and we are assuming that
%for some
%$\ell\in\bbR$
%$\ell\leq \phi(x)$
the following inequalities hold
for some constants $a,b>0$
and
$\ell\in\bbR$:
$$ |\phi(R^{K,\infty}_T(V))|\leq \max\{|\ell|,|\phi(S)|,|\phi(-S)|\}\leq \max\{|\ell|,a |S|^p+b\}
\leq a |S|^p+b+|\ell|.$$
The Burkholder-Davis-Gundy 
inequality~\cite[Thm~IV.4.1]{RevuzYor:99}
applied to the martingale
$R(V)$
at time
$T$,
together with inequality~\eqref{eq:Int_Assump_Chain},
implies that
$|S|^p$
is an integrable random variable.
The Dominated Convergence Theorem
therefore yields
the $L^1$-convergence 
for 
$\phi(R^{K,\infty}_T(V))\to \phi(R_T(V))$
as 
$K\to\infty$.
By~\eqref{eq:Inequalities_K_n}
for 
$n=\infty$,
we obtain the following inequalities for any
$V\in\cV$:
\begin{eqnarray}
\nonumber
\EE_{r,z}[\phi(R_T(V^{II}))]  = 
 \lim_{K\to\infty}\EE_{r,z}[\phi(R^{K,\infty}_T(V^{II}))] & \geq &
\lim_{K\to\infty}\EE_{r,z}[\phi(R^{K,\infty}_T(V))] =  
\EE_{r,z}[\phi(R_T(V))]  \\
& \geq & \lim_{K\to\infty}\EE_{r,z}[\phi(R^{K,\infty}_T(V^I))] = \EE_{r,z}[\phi(R_T(V^I))],  
\label{eq:Main_Ineq_Thm}
\end{eqnarray}
implying Theorem~\ref{thm:tracking_control}.
under the additional assumption in~\eqref{eq:Add_Cond_on_phi}.

In order to relax the assumption 
$\phi\in\cC^2(\bbR)$,
fix a non-negative
$g\in\cC^\infty(\bbR)$
with support
in
$[M,0]$,
for some $M\in(-\infty,0)$,
satisfying
$\int_{-\infty}^0 g(y)\,\dd y=1$.
For each
$n\in\bbN$,
define the convolution 
$$\phi_n(x):=\int_{-\infty}^0 \phi(x+y/n) g(y)\,\dd y,\quad x\in\bbR.$$
Note that 
$\phi_n:\bbR\to\bbR$
is a convex function,  
which satisfies both~\eqref{eq:Ass_phi} and~\eqref{eq:Add_Cond_on_phi}
(here we still assume that 
$\phi$ is bounded from below),
and the sequence 
$(\phi_n)_{n\in\bbN}$
converges point-wise to
$\phi$
as 
$n\uparrow\infty$
(see e.g.~\cite{RevuzYor:99}, proof of Theorem~VI.1.1 and 
Appendix~3).\footnote{We thank one of the referees for observing
that Theorems~\ref{thm:tracking} and~\ref{thm:tracking_control}
require neither smoothness nor boundedness from below
of the function $\phi$ and suggesting the argument presented here.}
Since 
$\phi$
satisfies~\eqref{eq:Ass_phi},
for any 
$x\in\bbR$
and
$n\in\bbN$
we have
\begin{eqnarray*}
\ell \leq \phi_n(x)  \leq  \max\{\phi(x+M/n),\phi(x)\} \leq a\max\{|x+M/n|^p,|x|^p\}+b \leq  A |x|^p +B,
\end{eqnarray*}
where the constants
$A,B>0$
are independent of both
$n$
and
$x$.
Since the random variable 
$|S|^p$
is integrable (see previous paragraph), where
$S$ is defined in~\eqref{eq:S_sup_max},
so is
$|R_T(V)|^p$
for any
$V\in\cV$.
The inequality above and the Dominated Convergence Theorem
imply
$$
\lim_{n\to\infty}\EE[\phi_n(R_T(V))] = \EE[\phi(R_T(V))]\qquad \text{for any $V\in\cV$,}
$$ 
which together with the inequalities in~\eqref{eq:Main_Ineq_Thm}, establishes Theorem~\ref{thm:tracking_control}
for $\phi$ that are bounded from below
and satisfy~\eqref{eq:Ass_phi}.

Since 
for any
$V\in\cV$
the processes 
$X$
and
$Y(V)$
are true martingales 
by~\eqref{eq:Mart_Integrab_Quad_VAr},
we may substitute 
$\phi$
with a function
$\phi^c(x):=\phi(x)+cx$,
$x\in\bbR$,
for any constant 
$c\in\bbR$,
without altering the solution of 
Problems~\eqref{eq:sol_Td}-\eqref{eq:sol_Tu}.
For any 
$\phi$
satisfying~\eqref{eq:Ass_phi}
there exists
some 
$c\in\bbR$
such that 
$\phi^c$
is bounded from below 
and hence Theorem~\ref{thm:tracking_control} follows.
\hfill \ensuremath{\Box}

\subsection{Non-Markovian Tracking}
\label{subsec:Non_Markov_Trackin}
The Markovian structure of $Z$ does not feature explicitly in 
the conclusion of 
Theorem~\ref{thm:tracking_control},
but only in its assumptions. It is therefore
natural to ask whether,
under some additional hypothesis, 
Theorem~\ref{thm:tracking_control}
can be generalised to a non-Markov 
volatility process
$Z$.
In this section we argue intuitively that,
for such a generalisation 
to hold for a large class of convex cost functions
$\phi$,
an underlying Markovian structure 
is in fact necessary
but show that it is possible
in the special case 
$\phi(x)=x^2$
(see Section~\ref{subsec:Non_Markov_counterexample}
for an explicit example of a process 
$Z$,
with a countable discrete state space 
$\bbE$
in
$\bbR$,
which is not $(\cF_t)$-Markov
and the conclusion of 
Theorem~\ref{thm:tracking_control}
fails).

Assume (in this section only) 
that the stochastic integrals 
$X$
and
$Y(V)$
are given by
\begin{eqnarray}
\label{eq:Non_MArkov_Tracking}
X_t  =   x+\int_0^tH_s\,\dd B_s
\qquad\text{and}\qquad
Y_t(V)  =  y+\int_0^tJ_s\,\dd V_s,
\end{eqnarray}
for some progressively measurable integrands
$H=(H_t)_{t\geq0}$
and
$J=(J_t)_{t\geq0}$
on
$(\Omega, (\cF_t)_{t\geq0}, \cF, \PP)$
and any 
$V\in\cV$.
As usual, we denote
the difference of 
$X$
and
$Y(V)$
by
$R(V)=X-Y(V)$.
The extremal Brownian motions 
$V^I$
and
$V^{II}$, defined in~\eqref{eq:Def_VI_VII},
can be generalised naturally by 
$ V^I_t=\int_0^t \sgn(H_sJ_s)\,\dd B_s $
and
$V_t^{II}=-V_t^I$.
Hence, for any fixed 
$V\in\cV$,
we can define the Brownian motions 
$V^{It}$
and
$V^{IIt}$
as in~\eqref{eq:Bellman_BMs}.
If the generalisation of 
Theorem~\ref{thm:tracking_control}
were to hold in this setting,
%The task now is to understand the stochastic
%dynamics of the corresponding 
the Bellman processes 
$B^I(V)$
and
$B^{II}(V)$,
defined 
in~\eqref{eq:Bellman_processes_TRacking},
would be a submartingale and a supermartingale,
respectively,
for any
$V\in\cV$.
We will focus on  
$B^I(V)$,
as the issues with 
$B^{II}(V)$
are completely analogous.
Representation~\eqref{eq:V_Rep} of
$V$
in Lemma~\ref{lem:R^V_rep}
and It\^o's formula yield
\begin{eqnarray*}
\phi\left(R_T(V^{It})\right) & = &  \phi\left(R_0(V^{It})\right) 
+ M^I_T + 
\frac{1}{2}
\int_0^t \phi''\left(R_s(V)\right) 
\left(H_s^2-2C_sH_sJ_s+J_s^2\right)\,\dd s\\
& + & \frac{1}{2}
\int_t^T \phi''\left(R_s(V^I)-R_t(V^I)+R_t(V)\right) 
\left(|H_s|-|J_s|\right)^2\,\dd s,
\end{eqnarray*}
where
$M^I$
is a local martingale, which we assume to be a true martingale.
The process 
$B^I_t(V)= \EE\left[\phi(R_T(V^{It}))\vert \cF_t\right]$
is a submartingale if and only if the conditional 
expectation 
$\EE[B^I_{t'}(V)-B^I_{t}(V)\vert\cF_t]$,
proportional to 
%\begin{eqnarray*}
%\EE\left[B^I_{t'}(V)-B^I_{t'}(V)\Big\vert\cF_t\right] 
%& =  &
%\frac{1}{2} 
%\EE\left[\int_t^{t'}\phi''\left(R_s(V^{It})\right)\left(H_s^2-2C_sH_sJ_s+J_s^2\right)
%- \phi''\left(R_s(V^I)-R_t(V^I)+R_t(V)\right)  
%\left(|H_s|-|J_s|\right)^2\,\dd s
%\Big\vert\cF_t\right] \\
%& + & \frac{1}{2}
%\EE\left[
%\int_{t'}^T 
%\left(\phi''\left(R_s(V^I)-R_{t'}(V^I)+R_{t'}(V)\right) -
%\phi''\left(R_s(V^I)-R_{t}(V^I)+R_{t}(V)\right) \right)
%\left(|H_s|-|J_s|\right)^2\,\dd s,
%\Big\vert\cF_t\right].
%\end{eqnarray*}
\begin{eqnarray*}
&\EE\left[\right.
\int_{t'}^T 
\left[\phi''\left(R_s(V^I)-R_{t'}(V^I)+R_{t'}(V)\right) -
\phi''\left(R_s(V^I)-R_{t}(V^I)+R_{t}(V)\right) \right]
\left(|H_s|-|J_s|\right)^2\dd s
&
\\
&   
\left.
+ \int_t^{t'}\phi''\left(R_s(V)\right)\left(H_s^2-2C_sH_sJ_s+J_s^2\right)
- \phi''\left(R_s(V^I)-R_t(V^I)+R_t(V)\right)  
\left(|H_s|-|J_s|\right)^2\dd s 
\Big\vert\cF_t\right]&
\end{eqnarray*}
by the formula above, 
is non-negative for all $0\leq t<t'\leq T$.
Hence
$B^I(V)$
is a submartingale for general integrands
$J$
and
$H$
if 
$\phi''$
does not depend on the state,
i.e. 
when the cost criterion $\phi$
is quadratic,
and we obtain:

\begin{proposition}
Let
$R(V)=X-Y(V)$,
where 
$X,Y(V)$
are as in~\eqref{eq:Non_MArkov_Tracking},
and
$T>0$.
Then we have 
$$
\EE\left[(X_T-Y_T(V^{I}))^2\right] \hspace{3mm}
\leq  
\hspace{3mm} \EE\left[(X_T-Y_T(V))^2\right] \hspace{3mm}
\leq \hspace{3mm}\EE\left[(X_T-Y_T(V^{II}))^2\right]
\quad \text{for any $V\in\cV$.}
$$
\end{proposition}

This proposition is consistent with
an argument based on 
It\^o's isometry:
the variance  
of a stochastic integral is equal to 
the expectation of its quadratic variation
and hence minimising/maximising its variance is equivalent
to locally minimising/maximising the Radon-Nikodym derivative of
its quadratic variation. 
Furthermore, it is also clear from the representation
above that in the absence of an underlying Markovian 
structure, 
for a general convex
$\phi$,
the process 
$B^I(V)$
may fail to be a submartingale 
and hence
the strategy in Theorem~\ref{thm:tracking_control}
is not optimal for general non-Markovian integrands
(see Section~\ref{subsec:Non_Markov_counterexample}
for an explicit example demonstrating this phenomenon).

%========================================================================================
\section{Coupling}
\label{sec:Coupling} 

In this section we consider the problems of 
\textit{minimising}
and 
\textit{maximising} 
the coupling time
of the processes
$X$
and
$Y(V)$
defined in~\eqref{eq:Proc_Def_X},
where the controller is free to choose the driving Brownian motion
$V$
in the integral
$Y(V)$
and the volatility is driven by a continuous-time $(\cF_t)$-Markov chain
$Z$. 
Put differently, we seek sharp upper and lower bounds
for the probability of the event that the coupling 
of
$X$
and
$Y(V)$
occurs after a fixed time 
$T$.
The couplings are characterised 
by the stochastic extrema 
of 
the stopping time
$\tau_0(X-Y(V)):=\inf\{t\geq0:X_t=Y_t(V)\}$
(with convention
$\inf\emptyset = \infty$).
More precisely, 
for any fixed
$T>0$,
we consider the following problems:
\begin{eqnarray*}
\hspace{5mm} & & \text{minimise}\quad
\PP\left[\tau_0(X-Y(V))>T\right]\quad\text{over}\quad V\in\cV,\\
\hspace{5mm} & & \text{maximise}\quad
\PP\left[\tau_0(X-Y(V))>T\right]\quad\text{over}\quad V\in\cV.
\end{eqnarray*}

\begin{theorem}
\label{thm:coupling_control}
Let 
$V^I$
and 
$V^{II}$
be as given by~\eqref{eq:Def_VI_VII}
and
$Z$
%satisfies~\eqref{eq:Int_Assump}.
satisfy~\eqref{eq:Mart_Integrab_Quad_VAr},~\eqref{eq:Z_Chain_Assumption} and~\eqref{eq:Coupling_Int_Assump}. 
%and~\eqref{eq:Int_Assump}
Then for any 
$T>0$
we have
\begin{eqnarray}
\label{eq:sol_Cd}
\inf_{V\in\cV} \PP\left[\tau_0(X-Y(V))>T\right] & = &
\PP\left[\tau_0(X-Y(V^{II}))>T\right],\\
\label{eq:sol_Cu}
\sup_{V\in\cV} \PP\left[\tau_0(X-Y(V))>T\right] & = &
\PP\left[\tau_0(X-Y(V^{I}))>T\right]. 
\end{eqnarray}
\end{theorem}

In this section we prove 
Theorem~\ref{thm:coupling_control},
which clearly implies
Theorem~\ref{thm:coupling},
and hence 
solves
Problem~\textbf{(C)}
for a continuous-time $(\cF_t)$-Markov chain
$Z$. 
The aim is to minimise and maximise the coupling time of the martingales 
$X$ 
and
$Y(V)$ given in~\eqref{eq:Proc_Def_X}.
%We
%start by noting that, 
Due to the symmetry in
Problem~\text{\textbf{(C)}},
we may therefore assume without loss of generality that
the starting points of the processes 
$X_0=x$
and
$Y_0(V)=y$
satisfy the inequality
\begin{eqnarray}
\label{eq:start_Point_Assum}
x&\leq& y.
\end{eqnarray}

The candidate value functions in 
Problems~\eqref{eq:sol_Cd} and~\eqref{eq:sol_Cu}
will be functionals of the law of the Markov processes 
%with the state space in
%$\bbR\times\bbR^d$,
$(R(V^{II}),Z)$
and
$(R(V^{I}),Z)$,
respectively,
where 
$R(V)$
is given in~\eqref{eq:R^V_Rep}
and the Brownian motions
$V^{II}$
and
$V^{I}$
are defined in~\eqref{eq:Def_VI_VII}.
%Theorem~\ref{thm:coupling}.
The first step in the proof of Theorem~\ref{thm:coupling_control}
is to localise 
Problems~\eqref{eq:sol_Cd} and~\eqref{eq:sol_Cu}.
With this in mind, for
any
$n\in\bbN$
recall definition~\eqref{eq:Stopped_Chain}
of the stopping time
$\tau_n$
and the stopped chain 
$Z^n$.
Unlike in Section~\ref{sec:Tracking},
in the case of coupling it is important 
to localise the process
$R(V)$
by stopping only the integrand.
The process
$R^n(V):=(R^n_t(V))_{t\geq0}$
is therefore given by
\begin{equation}
\label{eq:R_n_V_def}
R^n_t(V):=r+\int_0^t\sigma_1(Z^n_s)\,\dd B_s-\int_0^t\sigma_2(Z^n_s)\,\dd V_s,  \qquad r\leq0,
\end{equation}
where 
$B$
is the fixed Brownian motion 
and
$V\in\cV$
any Brownian motion on our probability space.
As in the previous section,
%Section~\ref{sec:Tracking},
in this circumstance it is also natural to identify the limit
$(R^\infty(V),Z^\infty)$
with the process
$(R(V),Z)$.
For 
$n\in\bbN\cup\{\infty\}$,
we define 
the first entry time 
of the process
$R^n(V)$
into the positive half-line
by
\begin{eqnarray}
\tau_0^+(R^n(V)) & := & \inf\{t\geq0\,:\,R^n_t(V)>0\}\qquad\text{(with $\inf\emptyset = \infty$)}.
\label{eq:Def_Entrance}
\end{eqnarray}

The localisation procedure will allow us to reduce the problem
%proof of Theorem~\ref{thm:coupling_control}
to the case where the generator of the volatility chain 
$Z$
is bounded, which will in turn make it possible to 
establish sufficient regularity of the candidate value
functions and conclude
that certain processes are true martingales 
(see Section~\ref{subsec:StochTime_Change}).
The two Markov processes
$(R^{IIn},Z^n)$
and
$(R^{In},Z^n)$,
which play a key role in the solution of Problems~\eqref{eq:sol_Cd} and~\eqref{eq:sol_Cu},
%proof of Theorem~\ref{thm:coupling_control},
are defined by
\begin{eqnarray}
\label{eq:Def_of_R}
R^{IIn}_t  :=  r + \int_0^{t} \Sigma_{II}(Z^n_s)\,\dd B_s
\qquad\text{and}\qquad
R^{In}_t  :=  r + \int_0^{t} \Sigma_{I}(Z^n_s)\,\dd B_s,
%\quad r\leq0,
\end{eqnarray}
for any
$r\leq0$,
where 
$B$
and
$Z^n$
are as above and the functions 
$\Sigma_{II},\Sigma_{I}:\bbE\to\bbR$
are given by
\begin{eqnarray}
\label{eq:Sigma_II}
\Sigma_{II}(z) & :=  & \sigma_1(z) + \sgn(\sigma_1(z)\sigma_2(z))\sigma_2(z)\quad \text{$\forall z\in\bbE$,}\\
\Sigma_I(z) & := & \sigma_1(z)-\sgn(\sigma_1(z)\sigma_2(z))\sigma_2(z)\quad \text{$\forall z\in\bbE$.}
\label{eq:Sigma_I}
\end{eqnarray}
Note that,
according to our definitions,
we have
$R^n(V^{II})\neq R^{IIn}$
and
$R^n(V^I)\neq R^{In}$
for any
$n\in\bbN$,
since 
the Brownian motions 
$V^I$
and 
$V^{II}$,
defined in~\eqref{eq:Def_VI_VII},
are given in terms of 
$Z$
and not
$Z^n$.
However, if we define 
the Brownian motions 
$V^{In}$
and 
$V^{IIn}$
by~\eqref{eq:Def_VI_VII}
with 
$Z$
replaced by 
$Z^n$,
then the equalities
$R^n(V^{IIn})=R^{IIn}$
and
$R^n(V^{In})= R^{In}$
hold.

The proof of 
Theorem~\ref{thm:coupling_control}
can now be carried out in three steps. 
First, we formulate a pair of ``approximate'' coupling 
problems (for each
$n\in\bbN$):
\begin{eqnarray}
\label{eq:sol_Cd_n}
  \text{minimise}\quad \PP_{r,z}\left[\tau_0^+(R^n(V))>T\right] \quad\text{over}\quad V\in\cV, &  & \\
\label{eq:sol_Cu_n}
\text{maximise}\quad
\PP_{r,z}\left[\tau_0^+(R^n(V))>T\right] \quad\text{over}\quad V\in\cV, &  & 
\end{eqnarray}
for a fixed
$T>0$
and any starting points
$r\leq0$,
$z\in\bbE$.
The following probabilistic representations
for the candidate value functions 
of Problems~\eqref{eq:sol_Cd_n} and~\eqref{eq:sol_Cu_n} 
play an important role in their solutions: 
\begin{eqnarray}
\label{eq:Val_Fun_II}
\zeta_n^{(II)}(r,z,t)&:=& \PP_{r,z}\left[\tau_{0}^+\!\left(R^{IIn}\right)>t\right], \\ 
\zeta_n^{(I)}(r,z,t)&:=& \PP_{r,z}\left[\tau_{0}^+\!\left(R^{In}\right)>t\right]. 
\label{eq:Val_Fun_I}
\end{eqnarray}

The second step, described in Section~\ref{subsec:StochTime_Change},
solves  Problems~\eqref{eq:sol_Cd_n} and~\eqref{eq:sol_Cu_n}. 
Lemmas~\ref{lem:measures_II_I} and~\ref{lem:rr_positive}
establish the necessary analytical properties of the candidate value functions
$\zeta_n^{(II)}$ 
and
$\zeta_n^{(I)}$, 
which 
enable us to prove 
(see
Lemma~\ref{lem:Coupling_n_Solution})
the optimality of
the Brownian motions
$V^{IIn}$
and
$V^{In}$.
%Problems~\eqref{eq:sol_Cd_n} and~\eqref{eq:sol_Cu_n}. 
More precisely,
the representations in~\eqref{eq:Val_Fun_II}-\eqref{eq:Val_Fun_I}
are used to 
establish the required differentiability 
of the functions
$\zeta_n^{(II)}$
and
$\zeta_n^{(I)}$,
which allows us to study the pathwise
evolution of the corresponding Bellman processes.
The optimality 
of
$V^{IIn}$
and
$V^{In}$,
established 
in 
Lemma~\ref{lem:Coupling_n_Solution},
%Problems~\eqref{eq:sol_Cd_n} and~\eqref{eq:sol_Cu_n},
%respectively,
is a consequence of the non-positivity of the second derivatives
$\frac{\partial^2 \zeta_n^{(II)}}{\partial r^2}$
and
$\frac{\partial^2 \zeta_n^{(I)}}{\partial r^2}$
proved in Lemma~\ref{lem:rr_positive}.

The third step in the proof of 
Theorem~\ref{thm:coupling_control},
given in Section~\ref{subsec:Proof_Coupling},
applies  approximation arguments,
which establish the Brownian motions
$V^{II}$
and
$V^{I}$
as the solutions of 
Problems~\eqref{eq:sol_Cd} and~\eqref{eq:sol_Cu}. 

Finally, Section~\ref{subsec:Non_Markov_Coupling}
discusses the issues that arise 
with a direct approach, 
based on the Dambis, Dubins-Schwarz theorem
(see e.g.~\cite[Thm~V.1.6]{RevuzYor:99}),
to the coupling problems 
in~\eqref{eq:sol_Cd} and~\eqref{eq:sol_Cu}.

\subsection{The stochastic time-change}
\label{subsec:StochTime_Change}
Throughout this section we fix
$n\in\bbN$.
Let 
$\Sigma_{II}:\bbE\to\bbR$
be as in~\eqref{eq:Sigma_II}
and note that
our standing assumption
$(|\sigma_1|+|\sigma_2|)(z)>0$
implies 
$\Sigma_{II}^2(z)>0$
for all
$z\in\bbE$.
Therefore, 
the stochastic time-change
$A^{II}=(A^{II}_t)_{t\geq0}$,
given by
\begin{equation}
\label{eq:Time_Change_A}
A^{II}_t := \int_0^t\Sigma_{II}^2(Z^n_s)\,\dd s,
\end{equation}
is a differentiable, strictly increasing process.
Furthermore, the definition of 
$Z^n$
and~\eqref{eq:Time_Change_A}
imply that 
the almost sure limit
$\lim_{t\uparrow\infty}A^{II}_t=\infty$
holds.
Hence, the inverse
$E^{II}=(E^{II}_s)_{s\geq0}$,
defined as the unique solution of  
$$A^{II}_{E^{II}_s}=s,\quad s\geq0,  
\qquad\text{also satisfies}\qquad
E^{II}_{A^{II}_t}=t
\qquad\text{for all
$t\geq0$,}
$$
and is a strictly increasing process
with differentiable trajectories.
Since 
$Z^n$
is an $(\cF_t)$-Markov chain,
%is 
%$(\cF_t)$-adapted, 
it is 
by Lemma~\ref{lem:_Zand_W_indep}
independent  of the
$(\cF_t)$-Brownian motion 
$B$
in~\eqref{eq:Def_of_R}.
Therefore 
the laws of the processes
$(R^{IIn},Z^n)$
and
$(r+B_{A^{II}},Z^n)$
coincide, where 
$B_{A^{II}}$
denotes the Brownian motion 
$B$
time-changed by the increasing process
$A^{II}$.

Let
$\Sigma_I:\bbE\to\bbR$
be as in~\eqref{eq:Sigma_I}
and assume further that 
$|\sigma_1|(z)\neq|\sigma_2|(z)$
for all 
$z\in\bbE$.
This implies the inequality
$\Sigma_{I}^2(z)>0$
for all
$z\in\bbE$.
Define, in an analogous way to~\eqref{eq:Time_Change_A}, the 
strictly increasing continuous time-change
$A^{I}=(A^{I}_t)_{t\geq0}$
and its inverse 
$E^{I}=(E^{I}_s)_{s\geq0}$,
and note that the processes
$(R^{In},Z^n)$
and
$(r+B_{A^{I}},Z^n)$
have the same law.
We can now state and prove Lemma~\ref{lem:measures_II_I}.
 
\begin{lemma}
\label{lem:measures_II_I}
Pick any
$r\leq 0$
and define
the stopping time 
$\tau^B_r:=\inf\{s:B_s=-r\}$ 
(with $\inf\emptyset =\infty$).
%for any $r\leq 0$,
Recall that the function
$G(r,t):=\PP\left[\tau^B_r>t\right]$,
for any
$t\geq0$,
%and
%$g(r,t):=-\frac{\partial G}{\partial t}(r,t)$ 
takes the form
$$
%=2N\left(-\frac{r}{\sqrt{t}}\right)-1
%\quad\text{for $r\leq0$ and 
%$t\geq0$},\qquad\
%\qquad\text{and}\qquad
%g(r,t)=-\frac{r}{t^{3/2}}\, n\left(-\frac{r}{\sqrt{t}}\right),
%G(r,t):=-\frac{r}{t^{3/2}}\, n\left(\frac{r}{\sqrt{t}}\right),
%\te^{-\frac{r^2}{2t}}/\sqrt{2\pi}
%\qquad\text{where $t\geq0$, $r<0$}
G(r,t)  =  \left\{ \begin{array}{ll}
2N\left(-\frac{r}{\sqrt{t}}\right)-1 & \textrm{ if $r<0,\>t\geq0$,}\\
0 & \textrm{ if $r=0,\>t\geq0$,}
\end{array} \right.
$$
where
$N(\cdot)$
%,
%$n(\cdot)$
denotes the standard normal cdf. %, pdf respectively. 
For any 
$n\in\bbN$
the following holds.
%Fix
%$n\in\bbN$.
\begin{enumerate}[(a)]
\item 
\label{lem:item_a_Finite}
For any
$z\in\bbE$
the following representation holds:
$$
\zeta_n^{(II)}(r,z,t) = 
\EE_z\left[G(r,A^{II}_t)\right]\qquad\text{for $r\leq0,\>t\geq0$.} 
%= 
%\int_{t}^{\infty}  \EE_z\left[g(r,A^{II}_s)(|\sigma_1|+|\sigma_2|)^{2}\left(Z_{s}\right)\right]\,\dd s.
$$
Hence
the partial derivatives 
$%(r,t)\quad\mapsto \quad
\frac{\partial \zeta_n^{(II)}}{\partial r}(r,z,t),
\frac{\partial^2 \zeta_n^{(II)}}{\partial r^2}(r,z,t), 
\frac{\partial \zeta_n^{(II)}}{\partial t}(r,z,t) $
exist 
for 
$r<0, t>0$.
\item 
\label{lem:item_b_Finite}
Assume further that
$|\sigma_1|(z')\neq|\sigma_2|(z')$
for all 
$z'\in\bbE$.
Then for any
$z\in\bbE$
we have
$$
\zeta_n^{(I)}(r,z,t) =
\EE_z\left[G(r,A^{I}_t)\right]\qquad \text{for $r\leq0,\>t\geq0$,}
%= 
%\int_{t}^{\infty}  \EE_z\left[g(r,A^{I}_s)(|\sigma_1|-|\sigma_2|)^{2}\left(Z_{s}\right)\right]\,\dd s
$$
and the partial derivatives
$
\frac{\partial \zeta_n^{(I)}}{\partial r}(r,z,t),
\frac{\partial^2 \zeta_n^{(I)}}{\partial r^2}(r,z,t), 
\frac{\partial \zeta_n^{(I)}}{\partial t}(r,z,t) 
$
exist
for any
$r<0, t>0$.
\end{enumerate}
\end{lemma}

\begin{proof} 
We first establish~\eqref{lem:item_a_Finite}. 
Recall the definition of 
the time-change process
$A^{II}$
and its inverse 
$E^{II}$
introduced above
and note that the following equalities hold almost surely by
the definition of the stopping time
$\tau^B_r$: 
\begin{equation*}
E^{II}_{\tau^B_r} =   \inf\{E^{II}_s:B_{s}=-r\} =\inf\{t:B_{A^{II}_t}=-r\} \qquad\quad
\text{(with $\inf\emptyset = \infty$).}
\end{equation*}
Therefore,
since the processes 
$(R^{IIn},Z^n)$
and
$(r+B_{A^{II}},Z^n)$
are equal in law, so are 
the random variables
$\tau_0^+(R^{IIn})$
and
$E^{II}_{\tau^B_r}$.
Since 
$E^{II}$
is a strictly increasing continuous 
inverse of 
$A^{II}$,
we have
\begin{eqnarray}
\nonumber
\PP_{r,z}\left[t<\tau_0^+(R^{IIn})\right] 
& = & \PP_{z}\left[A^{II}_t<\tau^B_r\right] \\
& = & \EE_z\left[G(r,A^{II}_t) \right].
\label{eq:REp_first_pas_Basic_Eq}
\end{eqnarray}
This, together with definition~\eqref{eq:Val_Fun_II},
implies the representation
of
$\zeta_n^{(II)}$
in part~\eqref{lem:item_a_Finite}
of the lemma.

The required differentiability of 
$\zeta_n^{(II)}$
in 
%$t$
%and 
$r$
%is also a consequence of~\eqref{eq:REp_first_pas_Basic_Eq}. 
%The differentiability of 
%$\zeta_n^{(II)}$
%in 
%$r$
follows from~\eqref{eq:REp_first_pas_Basic_Eq},
along the same lines as in the proof of Lemma~\ref{lem:Key_Props_I}.
An application of the Dominated Convergence Theorem, 
the mean value theorem 
and the smoothness and boundedness of the functions
$\frac{\partial G}{\partial r}$
and
$\frac{\partial^2 G}{\partial r^2}$
on a rectangle 
$(r-\varepsilon,r+\varepsilon)\times(0,\infty)$ %\subset(-\infty,0)\times(0,\infty)$
for any fixed
$r<0$
and small 
$\varepsilon>0$,
such that 
$\varepsilon+r<0$,
together imply the existence of
$\frac{\partial \zeta_n^{(II)}}{\partial r}(r,z,t)$
and
$\frac{\partial^2 \zeta_n^{(II)}}{\partial r^2}(r,z,t)$.

The differentiability 
of
$\zeta_n^{(II)}$
in
$t$
is more delicate as it is intimately related to the
integrability of the chain 
$Z^n$
and the unboundedness of the function
$\Sigma_{II}$.
We start with the following observation.

\noindent \textbf{Claim.} The stopping time 
$\tau_n$,
defined in~\eqref{eq:Stopped_Chain},
is a continuous random variable and 
\begin{eqnarray}
\label{eq:Stopped_VAriable_L_one}
\EE_z\left[I_{\{\tau_n\leq s\}}\Sigma^2_{II}\left(Z_{\tau_n}\right)\right] < \infty 
& & 
\text{for any $z\in\bbE$ and $s\geq0$.}
\end{eqnarray}
Since
$\PP_z\left[\tau_n>t\right] = \PP_{z}\left[Z^n_t\in U_n\cap\bbE\right]$,
the continuity of 
$\tau_n$
follows (the definition of the sets 
$U_n$
is given above equation~\eqref{eq:Stopped_Chain}).
To prove~\eqref{eq:Stopped_VAriable_L_one}, 
note first that 
$$
(Q_n\Sigma_{II}^2)(z)=(Q\Sigma_{II}^2)(z),
%\quad\text{for}
\quad
z\in\bbE\cap U_n,
\qquad\text{and}\qquad
(Q_n\Sigma_{II}^2)(z)=0,
%\quad\text{for}
\quad
z\in\bbE\setminus U_n.
$$
The assumption in~\eqref{eq:Coupling_Int_Assump}
and definition~\eqref{eq:Sigma_II}
imply that 
%the $Q$-matrix
$Q_n\Sigma_{II}^2$
is a bounded function,
even though neither
$Q\Sigma_{II}^2$
nor
$\Sigma_{II}^2$
necessarily are:
\begin{eqnarray}
\label{eq:Q_n_estimate}
\|Q_n\Sigma_{II}^2\|_{\infty}:=\sup_{z\in\bbE}|(Q_n\Sigma_{II}^2)(z)|<\infty.
\end{eqnarray}
Definition~\eqref{eq:Stopped_Chain}
yields the following inequalities 
$$I_{\{\tau_n\leq s\}}\Sigma^2_{II}\left(Z_{\tau_n}\right)
\leq  \Sigma^2_{II}\left(Z_{s\wedge \tau_n}\right)
=\Sigma_{II}^2\left(Z^n_s\right)\qquad
\text{for any $s\geq0$}.$$
Hence, to prove~\eqref{eq:Stopped_VAriable_L_one}, 
we need to show 
$\EE_z \Sigma_{II}^2\left(Z^n_s\right)<\infty$
for all states
$z\in\bbE$ 
and times
$s \geq 0$.
Recall, from the definition of
$Q_n$
in~\eqref{eq:Stopped_Chain_Gen},
that 
$Q_n$
is a bounded operator 
on the Banach space 
$\ell_\infty(\bbE)$
of bounded real functions mapping 
$\bbE$
into 
$\bbR$.
Let 
$\|Q_n\|_\infty<\infty$
denote its norm and recall that the norm satisfies
$\|Q_n^k\|_\infty\leq \|Q_n\|^k_\infty$
for all
$k\in\bbN$.
We can therefore use the exponential series to define
a bounded operator 
$\exp(sQ_n)$
and express the semigroup of 
$Z^n$
as follows:
$\EE_z \Sigma_{II}^2\left(Z^n_s\right)
=\left(\exp\left(sQ_n\right)\Sigma_{II}^2\right)(z)$.
Hence,
by~\eqref{eq:Q_n_estimate},
we find
$$
\EE_z \left[\Sigma_{II}^2\left(Z^n_s\right)\right]
%=
%\overline \Sigma_{II}^2\left(z\right) + \sum_{k=0}^\infty \frac{(sQ_n)^k}{(k+1)!} 
%\left(Q_n\Sigma_{II}^2\right)(z)
\leq
\Sigma_{II}^2\left(z\right) +s \sum_{k=0}^\infty \frac{(s\|Q_n\|_\infty)^k}{(k+1)!} 
\|Q_n\Sigma_{II}^2\|_\infty<\infty,
$$
for all
$z\in\bbE$
and
$s\geq0$.
This implies~\eqref{eq:Stopped_VAriable_L_one}
and proves the claim.

In order to prove that 
$\zeta_n^{(II)}$
is differentiable in 
time,
fix
$t>0$,
$r<0$,
$z\in\bbE$
and, for any
$\Delta t>0$,
define the random variable
$$ D_{\Delta t}(r,z,t):=
%\frac{1}{A^{II}_{t+\Delta t}-A^{II}_{t}}
\left[G(r,A^{II}_{t+\Delta t})-G(r,A^{II}_t)\right]/
\left(A^{II}_{t+\Delta t}-A^{II}_{t}\right). %^{-1}
$$
Since
$t>0$
(resp.
$\Delta t>0$),
we have 
$A^{II}_{t}>0$
(resp.
$(A^{II}_{t+\Delta t}-A^{II}_{t})>0$)
$\PP_z$-a.s.
Note also that 
the random variable
$|D_{\Delta t}(r,z,t)|$
is bounded by a constant 
uniformly in
$\Delta t>0$.
This follows from the existence of a uniform bound on 
$\frac{\partial G}{\partial t}(r,\cdot)$
in the second variable for any fixed 
$r<0$
and the mean value theorem. 
Furthermore the following limits hold:
\begin{eqnarray}
\label{eq:Almos_Sure_D_delta}
\lim_{\Delta t\to0} D_{\Delta t}(r,z,t) = 
\frac{\partial G}{\partial t}(r,A^{II}_{t})\quad\text{$\PP_z$-a.s.},
\qquad%\text{and}\quad
\lim_{\Delta t\to0}\frac{A^{II}_{t+\Delta t}-A^{II}_{t}}{\Delta t}=
\Sigma_{II}^2(Z^n_t)
\quad\text{$\PP_z$-a.s.}.
\end{eqnarray}
The quotient 
$(\zeta_n^{(II)}(r,z,t+\Delta t)-\zeta_n^{(II)}(r,z,t))/\Delta t$
now takes the form
\begin{eqnarray}
\label{eq:Time_Der_Zeta_II}
\EE_z\left[
D_{\Delta t}(r,z,t)
\left(A^{II}_{t+\Delta t}-A^{II}_{t}\right)/\Delta t
%\frac{A^{II}_{t+\Delta t}-A^{II}_{t}}{\Delta t}
%\left[G(r,A^{II}_{t+\Delta t})-G(r,A^{II}_t)\right]/
%\left(A^{II}_{t+\Delta t}-A^{II}_{t}\right)/\Delta t
\right]
&=&
\EE_z\left[
D_{\Delta t}(r,z,t)
I_{\{\tau_n\leq t\}}
\left(A^{II}_{t+\Delta t}-A^{II}_{t}\right)/\Delta t
\right]\\
\nonumber
&+&
\EE_z\left[
D_{\Delta t}(r,z,t)
I_{\{\tau_n\geq t+\Delta t\}}
\left(A^{II}_{t+\Delta t}-A^{II}_{t}\right)/\Delta t
\right]\\
&+&
\EE_z\left[
D_{\Delta t}(r,z,t)
I_{\{t<\tau_n< t+\Delta t\}}
\left(A^{II}_{t+\Delta t}-A^{II}_{t}\right)/\Delta t
\right].
\nonumber
\end{eqnarray}

Since
$I_{\{\tau_n\leq t\}}
\left(A^{II}_{t+\Delta t}-A^{II}_{t}\right)/\Delta t
=I_{\{\tau_n\leq t\}}\Sigma_{II}^2(Z_{\tau_n})$,
inequality~\eqref{eq:Stopped_VAriable_L_one}
in the claim above, the Dominated Convergence Theorem,
boundedness of 
$D_{\Delta t}(r,z,t)$
and~\eqref{eq:Almos_Sure_D_delta}
imply that the first expectation on the right-hand side
of~\eqref{eq:Time_Der_Zeta_II}
converges to
$\EE_z\left[
\frac{\partial G}{\partial t}(r,A^{II}_{t})
I_{\{\tau_n\leq t\}}\Sigma_{II}^2(Z_{\tau_n})
\right]$
as 
$\Delta t\to0$.

The random variable 
$I_{\{\tau_n\geq t+\Delta t\}} \left(A^{II}_{t+\Delta t}-A^{II}_{t}\right)/\Delta t$
is bounded by a constant for all 
$\Delta t$,
since, on the event
$\{\tau_n\geq t+\Delta t\}$,
the chain 
$Z$
has not left the finite state space
$U_n\cap\bbE$
by the time
$t+\Delta t$.
Therefore, by the Dominated Convergence Theorem, 
the second expectation on the right-hand side of~\eqref{eq:Time_Der_Zeta_II}
converges to
$\EE_z\left[
\frac{\partial G}{\partial t}(r,A^{II}_{t})
I_{\{\tau_n> t\}} \Sigma_{II}^2(Z_{t})
\right]$
as 
$\Delta t\to0$.

We will now prove that the third expectation 
on the right-hand side of~\eqref{eq:Time_Der_Zeta_II}
converges to $0$
as
$\Delta t\to0$.
By decomposing the path of 
$Z^n$
at 
$\tau_n$
on the event 
$\{t<\tau_n< t+\Delta t\}$
and applying the arguments used in the previous two paragraphs
to each of the two parts of the trajectory of 
$Z^n$,
there exists a constant 
$C^+>0$
such that 
\begin{eqnarray*}
\EE_z\left[
\frac{|D_{\Delta t}(r,z,t)|}{C^+}
I_{\{t<\tau_n< t+\Delta t\}}
\frac{A^{II}_{t+\Delta t}-A^{II}_{t}}{\Delta t}
\right]
& \leq & 
\EE_z\left[
\frac{\tau_n -t}{\Delta t}
I_{\{t<\tau_n< t+\Delta t\}}
\right]\\
& + &
\EE_z\left[
\frac{t+\Delta t-\tau_n}{\Delta t}
\Sigma_{II}^2(Z_{\tau_n})
I_{\{t<\tau_n< t+\Delta t\}}
\right]\\
& \leq & 
\PP_z\left[ t<\tau_n< t+\Delta t\right] + 
\EE_z\left[\Sigma_{II}^2(Z_{\tau_n})
I_{\{t<\tau_n< t+\Delta t\}}
\right].
\end{eqnarray*}
The probability 
$\PP_z\left[ t<\tau_n< t+\Delta t\right]$ 
tends to zero 
as
$\Delta t\to0$
by the claim and %the random variable
$\Sigma_{II}^2(Z_{\tau_n}) I_{\{t<\tau_n< t+\Delta t\}}$
is, for 
$\Delta t\in(0,1)$,
bounded above by the random variable
$\Sigma_{II}^2(Z_{\tau_n}) I_{\{\tau_n< t+1\}}$,
which is integrable by~\eqref{eq:Stopped_VAriable_L_one}.
Therefore, another application of the Dominated Convergence Theorem
implies that the function
$\zeta_n^{(II)}$
is right-differentiable in time.
In the case
$\Delta t<0$,
analogous arguments to the ones described above 
yield the
left-differentiability  of
$\zeta_n^{(II)}$.
The limits in~\eqref{eq:Almos_Sure_D_delta}
and their counterparts for
$\Delta t<0$
imply that the left- and right-derivatives in 
$t$
of
$\zeta_n^{(II)}$
coincide and part~\eqref{lem:item_a_Finite} follows.

For the proof of part~\eqref{lem:item_b_Finite},
note that, under the assumption 
$|\sigma_1|(z)\neq|\sigma_2|(z)$
for all 
$z\in\bbE$,
we have 
$\Sigma_I^2(z)=(|\sigma_1|-|\sigma_2|)^2(z)>0$
for all 
$z\in\bbE$.
Therefore, a completely analogous 
argument to the one that established
the equality in~\eqref{eq:REp_first_pas_Basic_Eq},
based on the stochastic time-change 
$A^I$
and the fact that 
the laws of the processes
$(R^{In},Z^n)$
and
$(r+B_{A^{I}},Z^n)$
coincide, where 
$B_{A^{I}}$
denotes the Brownian motion 
$B$
time-changed by the increasing process
$A^{I}$,
implies the representation of
$\zeta_n^{(I)}$
given in part~\eqref{lem:item_b_Finite}
of the lemma.
The differentiability of 
$\zeta_n^{(I)}$
follows along the same lines
as in part~\eqref{lem:item_a_Finite}.
The details of the arguments are now straightforward
and are left to the reader.
\end{proof}

Lemma~\ref{lem:rr_positive}
shows that the functions
$\zeta_n^{(II)}$
and
$\zeta_n^{(I)}$
solve the 
HJB equations
that correspond to the 
Problems~\eqref{eq:sol_Cd_n} and~\eqref{eq:sol_Cu_n}. 

\begin{lemma}
\label{lem:rr_positive}
Let 
$\zeta_n^{(II)}$ 
and
$\zeta_n^{(I)}$ 
be given by
\eqref{eq:Val_Fun_II}-\eqref{eq:Val_Fun_I}.
\begin{enumerate}[(a)]
%\item
%Then,
%for any 
%$z\in\bbE$,
%the partial derivatives 
%$$(r,t)\quad\mapsto \quad
%\frac{\partial \zeta_n^{(II)}}{\partial r}(r,z,t),\quad 
%\frac{\partial^2 \zeta_^{(II)}}{\partial r^2}(r,z,t), \quad
%\frac{\partial \zeta_n^{(II)}}{\partial t}(r,z,t) %:(-\infty,0)\times(0,\infty)\to\bbR
%$$
%exist %and are finite 
%at any point 
%$(r,t)$
%in the open subset 
%$(-\infty,0)\times(0,\infty)$
%of 
%$\bbR^2$.
%
%Furthermore,
%$r\mapsto\frac{\partial^2 \zeta_n^{(II)}}{\partial r^2}(r,z,t)$ 
%is a continuous function on the interval
%$(-\infty,0)$
%for any
%$z\in\bbE$
%and
%$t>0$.
\item The modulus of the partial derivative
$\lvert \frac{\partial \zeta_n^{(II)}}{\partial r}\rvert$ 
is bounded on the set 
$(-\infty,-\varepsilon)\times\bbE\times (0,\infty)$
for any 
$\varepsilon>0$
and the second derivative in space of $\zeta_n^{(II)}$ 
satisfies
\begin{eqnarray}
\label{eq:Negative_Second_Der_II}
\frac{\partial^2 \zeta_n^{(II)}}{\partial r^2}(r,z,t) & \leq & 0\qquad\text{for all
$(r,z,t)\in(-\infty,0)\times\bbE\times(0,\infty).$}
\end{eqnarray}
If 
$|\sigma_1|(z)\neq|\sigma_2|(z)$
for all 
$z\in\bbE$,
then 
the modulus
$\lvert \frac{\partial \zeta_n^{(II)}}{\partial r}\rvert$ 
is bounded on 
$(-\infty,-\varepsilon)\times\bbE\times (0,\infty)$,
for any 
$\varepsilon>0$,
and we have
\begin{eqnarray}
\label{eq:Negative_Second_Der_I}
\frac{\partial^2 \zeta_n^{(I)}}{\partial r^2}(r,z,t) & \leq & 0\qquad\text{for all
$(r,z,t)\in(-\infty,0)\times\bbE\times(0,\infty)$.}
\end{eqnarray}
\label{lem_negative:item_i}
\item For any $T>0$
the following holds
for all $r<0, t\in[0,T)$ and $z\in\bbE$
\begin{eqnarray}
\label{eq:HJB_II_Finite}
\inf_{c\in[-1,1]}\left\{\left[\cL^c\left(\zeta_n^{(II)}(\cdot,\cdot,T-t)\right)\right](r,z) -  
\frac{\partial \zeta_n^{(II)}}{\partial t}(r,z,T-t) \right\} & = & 0, %\qquad \text{for $r<0, t\in[0,T)$ and $z\in\bbE$,} 
\end{eqnarray}
where the function
$\cL^c\left(\zeta_n^{(II)}(\cdot,\cdot,T-t)\right)$
is defined in~\eqref{eq:Def_cL}
with 
$Q$
substituted by
$Q_n$
from~\eqref{eq:Stopped_Chain_Gen}.
Furthermore, we have
\begin{eqnarray*}
\label{eq:HJB_II_Boundary}
\zeta_n^{(II)}(r,z,0)  & = & 1\qquad\text{for all $(r,z)\in(-\infty,0)\times\bbE$,}\\
\zeta_n^{(II)}(0,z,t)  & = & 0\qquad\text{for all $(z,t)\in\bbE\times[0,\infty)$.}
\end{eqnarray*}
If 
$|\sigma_1|(z)\neq|\sigma_2|(z)$
for all 
$z\in\bbE$,
then for all $r<0, t\in[0,T)$ and $z\in\bbE$
we have
\begin{eqnarray}
\label{eq:HJB_I_Finite}
\sup_{c\in[-1,1]}\left\{\left[\cL^c\left(\zeta_n^{(I)}(\cdot,\cdot,T-t)\right)\right](r,z) -  
\frac{\partial \zeta_n^{(I)}}{\partial t}(r,z,T-t) \right\} & = & 0 %\qquad \text{for $r<0, t\in[0,T)$ and $z\in\bbE$,} 
\end{eqnarray}
(as above 
$\cL^c\left(\zeta_n^{(I)}(\cdot,\cdot,T-t)\right)$
is defined in~\eqref{eq:Def_cL}
with 
$Q$
substituted by
$Q_n$
from~\eqref{eq:Stopped_Chain_Gen})
and
%where the function
%$\cL^c\left(\zeta^{(I)}(\cdot,\cdot,T-t)\right)$
%is given in~\eqref{eq:Def_cL}.
%Furthermore, we have
\begin{eqnarray*}
\label{eq:HJB_I_Boundary}
\zeta_n^{(I)}(r,z,0)  & = & 1\qquad\text{for all $(r,z)\in(-\infty,0)\times\bbE$,}\\
\zeta_n^{(I)}(0,z,t)  & = & 0\qquad\text{for all $(z,t)\in\bbE\times[0,\infty)$.}
\end{eqnarray*}
\label{lem_negative:item_ii}
\end{enumerate}
\end{lemma}

\begin{proof}\eqref{lem_negative:item_i}
Let 
$G(r,t)$
be as defined in Lemma~\ref{lem:measures_II_I}.
Since
$n'(x)=-xn(x)$,
where 
$n(\cdot)$
is the standard normal pdf, we have
\begin{eqnarray}
\label{eq:G_first_Der}
\frac{\partial G}{\partial r}(r,t) & = & -\frac{2}{\sqrt{t}} n\left(-\frac{r}{\sqrt{t}}\right),\\
\frac{\partial^2 G}{\partial r^2}(r,t) & = & 2\frac{r}{t^{3/2}} n\left(-\frac{r}{\sqrt{t}}\right)\leq 0,
\label{eq:G_second_Der}
\end{eqnarray}
for all $r<0$, $t>0$.
The derivatives
$\frac{\partial^i G}{\partial r^i}$,
$i=1,2$,
are bounded on 
$(r-\varepsilon,r+\varepsilon)\times(0,\infty)$
for any
$r<0$
and small enough
$\varepsilon>0$
and hence, as in the proof of Lemma~\ref{lem:measures_II_I}, 
the Dominated Convergence Theorem implies
\begin{eqnarray*}
\frac{\partial \zeta_n^{(II)}}{\partial r}(r,z,t)  =  
\EE_z\left[\frac{\partial G}{\partial r}(r,A^{II}_t)\right] &\text{and} &
\frac{\partial^2 \zeta_n^{(II)}}{\partial r^2}(r,z,t)  =  
\EE_z\left[\frac{\partial^2 G}{\partial r^2}(r,A^{II}_t)\right] 
%\qquad\text{for all $r<0$, $z\in\bbE$, $t>0$.}
\end{eqnarray*}
for all $r<0$, $z\in\bbE$, $t>0$.
Inequality~\eqref{eq:Negative_Second_Der_II}
now follows from the inequality in~\eqref{eq:G_second_Der}
and the boundedness of
$\lvert \frac{\partial \zeta_n^{(II)}}{\partial r}\rvert$ 
on the product
$(-\infty,-\varepsilon)\times\bbE\times (0,\infty)$
is a consequence of~\eqref{eq:G_first_Der}.
Under the assumption that 
$|\sigma_1|(z)\neq|\sigma_2|(z)$
for all 
$z\in\bbE$,
the properties of the partial derivatives in space of 
$\zeta_n^{(I)}$
%the proof of~\eqref{eq:Negative_Second_Der_I}
follow from Lemma~\ref{lem:measures_II_I}~\eqref{lem:item_b_Finite}
and~\eqref{eq:G_first_Der}-\eqref{eq:G_second_Der}
along the same lines.

\noindent \eqref{lem_negative:item_ii}
In order to prove that 
$\zeta_n^{(II)}$
satisfies the HJB equation above, define a bounded martingale 
$M^{II}=(M^{II}_t)_{t\in[0,T]}$,
where
$$
M_t^{II}:=\PP_{r,z}\left[\tau_0^+(R^{IIn})>T\vert\cF_t\right],\qquad r\leq0,z\in\bbE,t\in[0,T],
$$
where the process
$R^{IIn}$,
started at
$R^{IIn}_0=r$,
is given in~\eqref{eq:Def_of_R} 
and the corresponding 
first-passage time 
$\tau_0^+(R^{IIn})$ %:=\tau_0^+(R(V^{II}))$
is defined in~\eqref{eq:Def_Entrance}.
The Markov property of the process
$(R^{IIn},Z^n)$
and the definition of 
$\zeta_n^{(II)}$
in~\eqref{eq:Val_Fun_II}
imply the equality 
\begin{eqnarray}
\label{eq:Markov_REp_Zeta_II}
\zeta_n^{(II)}\left(R^{IIn}_{\tau_0^+(R^{IIn})\wedge t},Z^n_{\tau_0^+(R^{IIn})\wedge t},T-\tau_0^+(R^{IIn})\wedge t\right)& = & \PP_{r,z}\left[\tau_0^+(R^{IIn})>T\vert\cF_t\right]=M_t^{II},
\end{eqnarray}
for all $r<0,z\in\bbE, t\in[0,T]$.

Note that, 
by~\eqref{eq:G_first_Der},
the modulus %first derivative
$\lvert\frac{\partial G}{\partial r}\rvert$
is globally bounded
on the set
$(-\infty,-\varepsilon]\times(0,\infty)$
for any
$\varepsilon>0$.
Let
$r<0$,
pick 
$\varepsilon\in(0,-r)$
and consider the stopped martingale
$M^{\varepsilon}=(M^{\varepsilon}_t)_{t\in[0,T]}$,
defined by
$$
M^{\varepsilon}_t:=M^{II}_{\tau_{-\varepsilon}^+\wedge t},\qquad\text{where $\tau_{-\varepsilon}^+:=\inf\{s\geq0:R^{IIn}_s=-\varepsilon\}$.}
$$
%and the stopped process
%$R^\varepsilon=(R^\varepsilon_t)_{t\geq0}$
%given by
%$R^\varepsilon_t:=R_{\tau_{-\varepsilon}^+\wedge t}$.
%\tau_{-\varepsilon}^+\leq \tau_{0}^+$
%$\PP_{r,z}$-a.s.,
It\^o's formula for general semimartingales~\cite[Sec~II.7,~Thm.~33]{Protter:05}
applied to the 
representation in~\eqref{eq:Markov_REp_Zeta_II} 
of the martingale
$M^{\varepsilon}$,
Lemma~\ref{lem:measures_II_I}~\eqref{lem:item_a_Finite},
Lemma~\ref{lem:Q_mart}
applied for the process
$U=(R^{IIn}_{t\wedge\tau_{-\varepsilon}^+})_{t\in[0,T]}$
and the bounded function 
$\zeta_n^{(II)}$,
and the facts that the quadratic covariation
$[R^{IIn},Z^{n,i}]_t=0$
vanishes for all times
$t$
and coordinates $Z^{n,i}$
of the chain
$Z^n$
(recall that
$\bbE\subset\bbR^d$),
$\frac{\partial \zeta_n^{(II)}}{\partial r}$
is bounded on
$(-\infty,-\varepsilon]\times\bbE\times(0,\infty)$
and
$\PP_{r,z}\left[R^{IIn}_{t\wedge\tau_{-\varepsilon}^+}\leq-\varepsilon, \,\forall t\geq0\right]=1$
together yield that the process
$N^\varepsilon=(N_t^\varepsilon)_{t\in[0,T)}$,
defined by
\begin{eqnarray*}
N_t^\varepsilon & := & \int_0^{t\wedge \tau_{-\varepsilon}^+}\left[\frac{1}{2}(|\sigma_1|+|\sigma_2|)^2(Z^n_s) 
\frac{\partial^2 \zeta_n^{(II)}}{\partial r^2}(R^{IIn}_s,Z^n_s,T-s)\right.\\
&  & \qquad + \left. (Q_n\zeta_n^{(II)}(R^{IIn}_s,\cdot,T-s))(Z^n_s)-
\frac{\partial \zeta_n^{(II)}}{\partial t}(R^{IIn}_s,Z^n_s,T-s)
\right]\,\dd s, 
\end{eqnarray*}
is a continuous martingale. Hence, since the quadratic variation of
$N^\varepsilon$
vanishes,
we have
$N_t^\varepsilon=0$
for all times
$t<T$
%identically equal to zero for all 
and starting
points
$(r,z)$
with 
$r<-\varepsilon$.
Since 
$\varepsilon>0$
is arbitrarily small, 
%and, by assumption~\eqref{eq:Z_Chain_Assumption}, the process 
%$(R^{IIn},Z^n)$
%visits neighbourhoods of all points in its state space
%with positive probability, 
%this implies that 
for all $r<0$, $z\in\bbE$
and
$t\in[0,T)$
we have:
\begin{eqnarray}
\label{eq:PDE_HJB}
& &
\frac{1}{2}(|\sigma_1|+|\sigma_2|)^2(z)
\frac{\partial^2 \zeta_n^{(II)}}{\partial r^2}(r,z,T-t)
+(Q_n\zeta_n^{(II)}(r,\cdot,T-t))(z)-
\frac{\partial \zeta_n^{(II)}}{\partial t}(r,z,T-t)  =  0
\end{eqnarray}
(here we also apply the fact that
for any $z\in\bbE$
we have 
$\PP_z[Z^n_t=z,\> \forall t\leq T]>0$
and on this event 
%$\tau_n\geq T$
%$\PP_z$-a.s.
the process
$R^{IIn}$
%R^{K,n}(V^{I})$
is by~\eqref{eq:Def_of_R} and~\eqref{eq:Sigma_II}
equal to 
%the constant 
%$r$
%(if $\sigma_1(z)=\sigma_2(z)$)
%or 
a Brownian motion 
%stopped when it exits 
%$(-K,K)$,
%Since,  
which,
with positive probability,
leaves the interval
$(-\infty,-\varepsilon)$
%stays in this interval 
after
$T$ and
visits a neighbourhood of any fixed
point in 
$(-\infty,-\varepsilon)$ before 
$T$).
%$(-K,K)$).
%the fact that 
%$N_t=0$
%for all 
%$t\in[0,T]$
%and starting
%points
%$(r,z)$
%implies
%the equality 
%\begin{eqnarray}
%\label{eq:PDE_HJB_Tracking_I}
%& &
%\frac{1}{2}(|\sigma_1|-|\sigma_2|)^2(z)
%\frac{\partial^2 \psi_{K,n}^{(I)}}{\partial r^2}(r,z,T-t)
%+(Q\psi_{K,n}^{(I)} (r,\cdot,T-t))(z)-
%\frac{\partial\psi_{K,n}^{(I)}}{\partial t}(r,z,T-t)  =  0
%\end{eqnarray}
%for all 
%$(r,z,t)\in(-K,K)\times\left(\bbE\cap U_n\right)\times[0,T)$.

To prove the first HJB equation above,
note that for any
$c\in[-1,1]$
the following inequality holds
\begin{eqnarray*}
(\sigma_1^2-2c\sigma_1\sigma_2+\sigma_2^2)(z)
\frac{\partial^2 \zeta_n^{(II)}}{\partial r^2}(r,z,T-t) 
& \geq &
(|\sigma_1|+|\sigma_2|)^2(z) %(\sigma_1^2+\sigma_2^2)(z)
\frac{\partial^2 \zeta_n^{(II)}}{\partial r^2}(r,z,T-t),
\end{eqnarray*}
for all
$r<0$,
$z\in\bbE$
and
$t\in[0,T)$
since
$\frac{\partial^2 \zeta_n^{(II)}}{\partial r^2}(r,z,T-t)\leq0$
by~\eqref{eq:Negative_Second_Der_II}.
This inequality,
the definition of
$\cL^c\zeta_n^{(II)}$
in~\eqref{eq:Def_cL}
and identity~\eqref{eq:PDE_HJB}
imply~\eqref{eq:HJB_II_Finite}.
The boundary behaviour 
of the function
$\zeta_n^{(II)}$,
stated in the lemma,
at
$t=0$
and at
$r=0$
follows directly from the representation of 
$\zeta_n^{(II)}$
given in~\eqref{eq:Val_Fun_II}.

In the case of the function
$\zeta_n^{(I)}$,
by~\eqref{eq:Negative_Second_Der_I}
it follows that 
%we have
%$\frac{\partial^2 \psi^{(I)}}{\partial r^2}\geq0$
%on $(-\infty,0]\times\bbE$
%and hence the following inequality
%holds
\begin{eqnarray*}
\frac{1}{2}(|\sigma_1|-|\sigma_2|)^2(z) 
\frac{\partial^2 \zeta_n^{(I)}}{\partial r^2}(r,z,T-t)
& \leq &
\frac{1}{2}(\sigma_1^2-2c\sigma_1\sigma_2+\sigma_2^2)(z)
\frac{\partial^2 \zeta_n^{(I)}}{\partial r^2}(r,z,T-t) 
\end{eqnarray*}
for any
$c\in[-1,1]$
and
all
$r<0$,
$z\in\bbE$,
$t\in[0,T)$.
An analogous argument to the one in the case of
$\zeta_n^{(II)}$
establishes the HJB equation in~\eqref{eq:HJB_I_Finite}
and the required boundary behaviour.
This concludes the proof of the lemma.
\end{proof}

We can now prove that 
$\zeta_n^{(II)}$
and
$\zeta_n^{(I)}$
are the value functions for 
Problems~\eqref{eq:sol_Cd_n} and~\eqref{eq:sol_Cu_n}.

\begin{lemma}
\label{lem:Coupling_n_Solution}
Pick a time horizon 
$T>0$
and, 
for any
$V\in\cV$,
let
$R^n(V)$
and
$\tau_0^+(R^n(V))$
be as in~\eqref{eq:R_n_V_def}
and~\eqref{eq:Def_Entrance}
respectively.
\begin{enumerate}[(a)]
\item 
\label{lem:item_Val_Fun_II_n}
The function 
$\zeta_n^{(II)}$,
defined in~\eqref{eq:Val_Fun_II},
satisfies the following:
\begin{eqnarray*}
\zeta_n^{(II)}(r,z,T) = \inf_{V\in\cV} \PP_{r,z}\left[\tau_0^+(R^n(V))>T\right] 
\qquad\text{for any $r\leq0,z\in\bbE$.}
\end{eqnarray*}
\item Assume that 
$|\sigma_1|(z)\neq|\sigma_2|(z)$
for all 
$z\in\bbE$.
Then the function 
$\zeta_n^{(I)}$,
given in~\eqref{eq:Val_Fun_I},
satisfies
\begin{eqnarray*}
\zeta_n^{(I)}(r,z,T) = \sup_{V\in\cV} \PP_{r,z}\left[\tau_0^+(R^n(V))>T\right] 
\qquad\text{for any $r\leq0,z\in\bbE$.}
\end{eqnarray*}
\label{lem:item_Val_Fun_I_n}
\end{enumerate}
\end{lemma}

\begin{proof}
\noindent \eqref{lem:item_Val_Fun_II_n}
Pick any Brownian motion
$V\in\cV$
and, 
for any
$t\in[0,T]$,
define the corresponding 
Brownian motion 
$V^{IInt}=(V_s^{IInt})_{s\geq0}\in\cV$
by
\begin{eqnarray}
\label{eq:Bellman_BM_Coupling}
V_s^{IInt}  :=  \left\{ \begin{array}{ll}
V_s  & \textrm{ if $s\leq t$,}\\
V_t + V^{IIn}_{s}-V^{IIn}_{t} & \textrm{ if $s>t$,}
\end{array} \right.
\end{eqnarray}
where 
$V^{IIn}\in\cV$
is given in~\eqref{eq:Def_VI_VII}
with 
$Z$
substituted by the stopped chain
$Z^n$.
For any
$r\leq0,z\in\bbE$, 
the Bellman process
$S^{II}=(S^{II}_t)_{t\in[0,T]}$
is defined by
\begin{eqnarray}
\label{eq:Technical_Reason}
S^{II}_t  & := &  \PP_{r,z}\left[\tau_0^+(R^n(V^{IInt}))\geq T\vert\cF_t\right],\quad t\in[0,T].
\end{eqnarray}
In this definition we use $\geq$ instead of $>$
for technical reasons (see Remark after this proof).
%which will allow us to establish the
%required inequality in~\eqref{lem:item_Val_Fun_II_n}.
%Note that the stopping time 
%$\tau_0^+$
%on the right-hand side of the equality 
%is given by
%$\tau_{0}^+:=\tau_{0}^+\left(R^n(V)\right)$,
%and hence does depend on the choice of 
%$V$.
%This is not explicitly stated 
%in the formula
%for brevity of notation.
Let
$\tau_{0}^+:=\tau_{0}^+\left(R^n(V)\right)$
and note that for any
$t\in[0,T)$ 
the equality 
$S^{II}_t=S^{II}_{t\wedge \tau_0^+}$
holds.
Hence we have
\begin{eqnarray*}
S^{II}_{t}  & = &  
\PP_{R^n_{\tau_{0}^+\wedge t}(V),Z^n_{\tau_{0}^+\wedge t}}\left[\tau_0^+(R^{IIn})\geq T-s\right]\vert_{s=t\wedge\tau_0^+} = 
\zeta_n^{(II)}\left(R^n_{\tau_{0}^+\wedge t}(V),Z^n_{\tau_{0}^+\wedge t},T-(t\wedge \tau_{0}^+)\right),
\end{eqnarray*}
by the strong Markov property 
and
definitions~\eqref{eq:Val_Fun_II},
\eqref{eq:Bellman_BM_Coupling} and~\eqref{eq:Def_of_R}
of the candidate value function
$\zeta_n^{(II)}$,
the Brownian motion 
$V^{IInt}$
and the process
$R^{IIn}$ respectively
(note that 
$\PP_{r,z}[R^{IIn}_u=0]=0$
for any $r\leq0$,
$z\in\bbE$
and
$u>0$, 
implying 
$\PP_{r,z}\left[\tau_0^+(R^{IIn})= u\right]=0$ and hence the second equality above).

\noindent \textbf{Claim.} The process
$(S^{II}_{t})_{t\in[0,T)}$ 
is a bounded c\`adl\`ag $(\cF_t)$-submartingale on the interval
$[0,T)$.

\noindent The process is c\`adl\`ag on $[0,T)$ 
by Lemma~\ref{lem:measures_II_I}\eqref{lem:item_a_Finite} 
and Assumption~\eqref{eq:Z_Chain_Assumption}. It is bounded by definition. 
To see that 
$(S^{II}_{t})_{t\in[0,T)}$ 
is a submartingale,
define a stopping time
$\tau_{-\varepsilon}^+:=\inf\{t\geq0:R_t^n(V)=-\varepsilon\}$,
for any small
$\varepsilon>0$,
and note that 
$\tau_{-\varepsilon}^+<\tau_{0}^+$.
Hence,
for any
$r\leq0,z\in\bbE,t\in[0,T]$,
we have
\begin{eqnarray}
\label{eq:SII_epsilon}
S^{II}_{t\wedge \tau_{-\varepsilon}^+}  & = &  
\zeta_n^{(II)}\left(R^n_{\tau_{-\varepsilon}^+\wedge t}(V),Z^n_{\tau_{-\varepsilon}^+\wedge t},T-(t\wedge \tau_{-\varepsilon}^+)\right).
\end{eqnarray}
%takes the form
%where the second equality follows from 
%by the Markov property 
%and
%definitions~\eqref{eq:Val_Fun_II} 
%and~\eqref{eq:Bellman_BM_Coupling}
%of the candidate value function
%$\zeta_n^{(II)}$
%and the Brownian motion 
%$V^{IIn}$.
%To establish the Bellman principle 
%applies, it is sufficient to 
%prove that the process 
%$S^{II}$
%is a submartingale for any starting point
%$(r,z)$
%of the process
%$(R^n(V),Z^n)$. 
By Lemma~\ref{lem:measures_II_I}~\eqref{lem:item_a_Finite}, 
It\^o's formula for general semimartingales~\cite[Sec~II.7,~Thm.~33]{Protter:05} 
can be applied to 
$(S^{II}_{t\wedge \tau_{-\varepsilon}^+})_{t\in[0,T)}$
%on the (closed) stochastic interval
%$[0,T\wedge \tau_{-\varepsilon}^+]$
for any fixed small 
$\varepsilon>0$.
%we obtain the following pathwise representation of the process
%$S^{II}$:
In particular, for any 
$t\in[0,T)$,
we obtain
\begin{eqnarray}
\label{eq:S^II_rep}
S^{II}_{t\wedge \tau_{-\varepsilon}^+} & = &  \zeta_n^{(II)}(r,z,T) + 
N_{t\wedge \tau_{-\varepsilon}^+} + 
D_{t\wedge \tau_{-\varepsilon}^+} + 
M_{t\wedge \tau_{-\varepsilon}^+}, %\quad\text{where}
\end{eqnarray}
where 
the processes 
$N, D$
and
$M$
are defined on the stochastic interval
$[0,T\wedge \tau_0^+)$:
%are given by:
\begin{eqnarray*}
N_t & := & \int_0^t \frac{\partial \zeta_n^{(II)}}{\partial r}(R^n_s(V),Z^n_s,T-s)\,\dd R^n_s(V),\quad t\in[0,T\wedge \tau_0^+),\\
% \qquad t<\tau_0^+,\\ %\text{on the stochastic interval $ \tau_0^+]$}\\
\nonumber
D_t & := & \int_0^{t}\left[\left(\cL^{C_s}\zeta_n^{(II)}\right)(R^n_s(V),Z^n_s,T-s) 
- \frac{\partial \zeta_n^{(II)}}{\partial t}(R^n_s(V),Z^n_s,T-s)
\right]\,\dd s,\quad t\in[0,T\wedge \tau_0^+),\\
%\qquad t<\tau_0^+,
\nonumber
M_t & := &
\sum_{0<s\leq t}\left[\zeta_n^{(II)}(R^n_s(V),Z^n_s,T-s)-\zeta_n^{(II)}(R^n_s(V),Z^n_{s-},T-s)\right]\\
\nonumber
& - &
\int_0^{t}(Q_n\zeta_n^{(II)}(R^n_s(V),\cdot,T-s))(Z^n_{s-})\,\dd s,%\qquad t<\tau_0^+,
\quad t\in[0,T\wedge \tau_0^+).
\nonumber
\end{eqnarray*}
Here
$C=(C_t)_{t\geq0}$
is the stochastic correlation process 
from Lemma~\ref{lem:R^V_rep},
which corresponds to the Brownian motion
$V$,
and
$\cL^c\zeta_n^{(II)}$
is defined in~\eqref{eq:Def_cL} 
for any constant
$c\in[-1,1]$
with  
$Q$
substituted by 
$Q_n$
from~\eqref{eq:Stopped_Chain_Gen}.
The representation 
%of 
%$S^{II}$
in~\eqref{eq:S^II_rep}
relies on the fact that
the continuous part of the quadratic covariation
$[R^{n}(V),Z^{n,i}]_t$
vanishes for all times
$t$
and coordinates $Z^{n,i}$
of the chain
$Z^n$.

%\begin{eqnarray}
%\label{eq:Submart_Ineq_II_coupling}
%S^{II}_{t\wedge \tau_{-\varepsilon}^+} & \geq &  \zeta_n^{(II)}(r,z,T) + 
%N_{t\wedge \tau_{-\varepsilon}^+} + 
%M_{t\wedge \tau_{-\varepsilon}^+}. %\quad\text{where}
%\int_0^{t\wedge \tau_{-\varepsilon}^+}\frac{\partial \zeta_n^{(II)}}{\partial r}(R^n_s(V),Z^n_s,T-s)\,\dd R^n_s(V)\\
%\nonumber
%&+& 
%\sum_{0<s\leq t\wedge \tau_{-\varepsilon}^+}\left[\zeta_n^{(II)}(R^n_s(V),Z^n_s,T-s)-\zeta_n^{(II)}(R^n_s(V),Z^n_{s-},T-s)\right]\\
%&  - &
%\int_0^{t\wedge \tau_{-\varepsilon}^+}(Q_n\zeta_n^{(II)}(R^n_s(V),\cdot,T-s))(Z^n_{s-})\,\dd s.
%\nonumber
%\end{eqnarray}

Apply Lemma~\ref{lem:Q_mart},
with 
$F(s,r,z):=\zeta_n^{(II)}(r,z,T-s)$,
$U:=R^n(V)$
and the chain
$Z^n$
(with bounded generator
$Q_n$),
to conclude that 
$(M_{t\wedge \tau_{-\varepsilon}^+})_{t\in[0,T)}$
is a martingale. 
%the process given by the second 
%and third lines of the last display
%is a martingale.
The process
$(N_{t\wedge \tau_{-\varepsilon}^+})_{t\in[0,T)}$
%The first integral on the right-hand side of~\eqref{eq:Submart_Ineq_II_coupling} 
is clearly a local martingale 
(since the integrator
$R^n(V)$
is a martingale) 
%Define a stopping time
%$\tau_{-\varepsilon}^+:=\inf\{t\geq0:R_t^n(V)=-\varepsilon\}$
%for any small
%$\varepsilon>0$
%and note that 
%$\tau_{-\varepsilon}^+<\tau_{0}^+$.
with integrable quadratic variation 
%of the 
%local martingale 
%$$\left(\int_0^{t\wedge \tau_{-\varepsilon}^+}\frac{\partial \zeta_n^{(II)}}{\partial r}(R^n_s(V),Z^n_s,T-s)\,\dd R^n_s(V)\right)_{t\in[0,T]}$$
%is, by
%bounded above by 
%$A_\varepsilon$,
$$
\langle N \rangle_{t\wedge \tau_{-\varepsilon}^+}
%A_\varepsilon:=
=\int_0^{t\wedge \tau_{-\varepsilon}^+}
\left(\frac{\partial \zeta_n^{(II)}}{\partial r}(R^n_s(V),Z^n_s,T-s)\right)^2
(\sigma_1^2(Z^n_s)-2C_s\sigma_1(Z^n_s)\sigma_2(Z^n_s)+\sigma_2^2(Z^n_s))
\,\dd s 
$$
%an integrable random variable
(apply Lemma~\ref{lem:rr_positive}~\eqref{lem_negative:item_i}
and assumption~\eqref{eq:Mart_Integrab_Quad_VAr}).
Therefore this stochastic integral is also a martingale. 
Since 
$C_t\in[-1,1]$
for all 
$t\geq0$, 
equality~\eqref{eq:HJB_II_Finite}
implies that $D_{t\wedge \tau_{-\varepsilon}^+} \geq0$ 
$\PP_{r,z}$-a.s. 
and hence, by~\eqref{eq:S^II_rep},
the process 
$(S^{II}_{t\wedge \tau_{-\varepsilon}^+})_{t\in[0,T)}$
is a submartingale. 

In order to prove that 
$(S^{II}_{t})_{t\in[0,T)}$ 
is a submartingale, we first show that the following limit holds
\begin{equation}
\label{eq:LIm}
\lim_{\varepsilon\to0} 
S^{II}_{t\wedge \tau_{-\varepsilon}^+}
= S^{II}_{t\wedge \tau_{0}^+} \qquad\text{$\PP_{r,z}$-a.s. for any $t\in[0,T)$.}
\end{equation} %non-negative 
%and,
%by taking expectations on both sides of inequality~\eqref{eq:Submart_Ineq_II_coupling},
%we obtain
%$$
%\EE_{r,z}\left[S^{II}_{T\wedge \tau_{-\varepsilon}^+}\right]
%\geq
%\zeta_n^{(II)}(r,z,T)\qquad\text{for all $\varepsilon>0$.}
%$$
%By Lemma~\ref{lem:measures_II_I}\eqref{lem:item_a_Finite}
The paths of 
$R^n(V)$
are continuous and we have
$\tau_{-\varepsilon}^+\uparrow \tau_{0}^+$ $\PP_{r,z}$-a.s.
as $\varepsilon\downarrow0$,
and hence
$\lim_{\varepsilon\downarrow0}R_{t\wedge\tau_{-\varepsilon}^+}^n(V)= R_{t\wedge\tau_{0}^+}^n(V)$.
Since 
$Z$ is a Feller process with c\`adl\`ag paths 
by~\eqref{eq:Z_Chain_Assumption},
the stopped chain
$Z^n$
is also 
(i.e. the semigroup  $P^n$ of $Z^n$ is continuous at $t=0$ and, if
$f$ is a bounded function on 
$\bbE$
that tends to zero at infinity, then so is
$P^n_tf$
for every 
$t\geq0$),
and hence quasi left-continuous.
Therefore, as 
$\varepsilon\downarrow0$,
$\PP_{r,z}$-a.s.
we have
$Z^n_{t\wedge \tau_{-\varepsilon}^+}\to Z^n_{t\wedge \tau_{0}^+}$
(i.e. the chain does not jump at 
$\tau_{0}^+$).
By representation~\eqref{eq:SII_epsilon} and the continuity 
in 
$(r,s)\in (-\infty,0]\times [0,t]$ (recall that $t<T$)
of the function $(r,s)\mapsto \zeta_n^{(II)}(r,z,T-s)$
for each $z\in\bbE$,
implied by Lemma~\ref{lem:measures_II_I}, 
equality in~\eqref{eq:LIm} follows. 

The claim now follows by~\eqref{eq:LIm},
the boundedness and the submartingale property of 
$(S^{II}_{t\wedge \tau_{-\varepsilon}^+})_{t\in[0,T)}$,
the fact 
$S^{II}_t=S^{II}_{t\wedge \tau_0^+}$
$\PP_{r,z}$-a.s.
for any
$t\in[0,T)$
and Fatou's lemma: for any
$0\leq s<t<T$ we have
\begin{equation*}
\EE_{r,z}\left[S^{II}_{t} \vert \cF_s\right]
= 
\EE_{r,z}\left[\limsup_{\varepsilon\downarrow0}S^{II}_{t\wedge \tau_{-\varepsilon}^+} \vert \cF_s\right]
\geq 
\limsup_{\varepsilon\downarrow0}\EE_{r,z}\left[S^{II}_{t\wedge \tau_{-\varepsilon}^+} \vert \cF_s\right]
\geq 
\limsup_{\varepsilon\downarrow0} S^{II}_{s\wedge \tau_{-\varepsilon}^+} =S^{II}_{s}.
\end{equation*}

Doob's submartingale convergence theorem and the Claim above imply that the following limit holds
almost surely and in $L^1$:
$\lim_{t\uparrow T}S^{II}_{t} =: S_T$.
The random variable $S_T$ 
satisfies 
$\PP_{r,z}\left[S_T\in[0,1]\right]=1$
(as does
$S^{II}_{t}$ for all
$t<T$)
and
$S_T I_{\{\tau_0^+(R^n(V))<T\}}=0$
$\PP_{r,z}$-a.s.
($S_T$ is an almost sure limit of a process
which is 
equal to zero for all $t$ close to $T$
on the event 
$\{\tau_0^+(R^n(V))<T\}$).
Note that 
$S_T$
need not be equal to 
$S^{II}_{T}=I_{\{\tau_0^+(R^n(V))\geq T\}}$.  
However, the limit satisfies
$S_T\leq S^{II}_{T}$
$\PP_{r,z}$-a.s., implying the key inequality
\begin{equation}
\label{eq:key_inequality_min}
\zeta_n^{(II)}(r,z,T)\leq 
\EE_{r,z}\left[S_T\right] \leq
\PP_{r,z}\left[\tau_0^+(R^n(V))\geq T\right] 
\quad\text{for any $V\in\cV$ and any $T>0$.}
\end{equation}

In order to prove the equality in part~\eqref{lem:item_Val_Fun_II_n},
we apply the inequality in~\eqref{eq:key_inequality_min}
to the time horizon $T+\delta$,
where
$T$ is as in the statement of the lemma and 
$\delta>0$
is arbitrary:
$$
\zeta_n^{(II)}(r,z,T+\delta)\leq \PP_{r,z}\left[\tau_0^+(R^n(V))\geq T+\delta\right].  
$$
Since the equality
$\cup_{k\in\bbN}\{\tau_0^+(R^n(V))\geq T+1/k\}  = \{\tau_0^+(R^n(V))>T\}$
holds
and
$\zeta_n^{(II)}$ is continuous in time (in fact differentiable,
see Lemma~\ref{lem:measures_II_I}~\eqref{lem:item_a_Finite}) 
away from time zero, for any $V\in\cV$ we get (e.g. by the DCT):
$$\zeta_n^{(II)}(r,z,T) = \lim_{k\to\infty} \zeta_n^{(II)}(r,z,T+1/k) 
\leq
\lim_{k\to\infty} \PP_{r,z}\left[\tau_0^+(R^n(V))\geq T+1/k\right] = 
\PP_{r,z}\left[\tau_0^+(R^n(V))>T\right].$$
This concludes the proof of part~\eqref{lem:item_Val_Fun_II_n}
of the lemma.

\noindent \eqref{lem:item_Val_Fun_I_n}
For any Brownian motion
$V\in\cV$
and 
$t\in[0,T]$,
define 
$V^{Int}=(V_s^{Int})_{s\geq0}\in\cV$
by
\begin{eqnarray*}
%\label{eq:Bellman_BM_Coupling}
V_s^{Int}  :=  \left\{ \begin{array}{ll}
V_s  & \textrm{ if $s\leq t$,}\\
V_t + V^{In}_{s}-V^{In}_{t} & \textrm{ if $s>t$,}
\end{array} \right.
\end{eqnarray*}
where 
$V^{In}\in\cV$
is given in~\eqref{eq:Def_VI_VII}
with 
$Z^n$
in the place of
$Z$. 
In this case 
the Bellman process
$S^{I}=(S^{I}_t)_{t\in[0,T]}$
is given by
%\begin{eqnarray*}
$S^{I}_t  :=  
\PP_{r,z}\left[\tau_0^+(R^n(V^{Int}))>T\vert\cF_t\right] =
\zeta_n^{(I)}\left(R^n_{\tau_0^+\wedge t}(V),Z^n_{\tau_0^+\wedge t},T-(\tau_0^+\wedge t)\right)$,
%\PP_{r,z}\left[\tau_0^+(R^n(V^{Int}))>T\vert\cF_t\right],$
%\end{eqnarray*}
for any
$r\leq0,z\in\bbE,t\in[0,T],$
where again
$\tau_{0}^+:=\tau_{0}^+\left(R^n(V)\right)$
and the second equality holds by 
the Markov property 
and~\eqref{eq:Val_Fun_I}. 
The proof in this case is simpler than in 
part~\eqref{lem:item_Val_Fun_II_n},
as analogous arguments to part~\eqref{lem:item_Val_Fun_II_n}
imply that
$(S^{I}_t)_{t\in[0,T)}$
is a supermartingale with a limit at $T$
that is in this case smaller or equal to 
$S^{I}_T$ (cf. Remark below). Therefore 
the analogous inequality to~\eqref{eq:key_inequality_min}
states
\begin{equation*}
\zeta_n^{(I)}(r,z,T)\geq 
\PP_{r,z}\left[\tau_0^+(R^n(V)) > T\right] 
\quad\text{for any $V\in\cV$,}
\end{equation*}
removing the need for an additional limiting argument 
based on the perturbation of the maturity $T$
(cf. the final paragraph of the proof of part~\eqref{lem:item_Val_Fun_II_n}).
The details are left to the reader.  
%This concludes the proof of the lemma. 
%Analogous arguments to the ones used in the proof of 
%part~\eqref{lem:item_Val_Fun_II_n}
%can now be applied to conclude that the process
%$S^I$,
%appropriately stopped is a supermartingale.
%Then
%the optimality of the Brownian motion 
%$V^{In}\in\cV$
%follows by a limiting argument as in part~\eqref{lem:item_Val_Fun_II_n}.
%This concludes the proof of the lemma.
\end{proof}

\begin{remark}
The reason for defining the Bellman process 
$S^{II}$
in~\eqref{eq:Technical_Reason}
with $\geq$
rather than
$>$,
as is naturally suggested by our setting in 
part~\eqref{lem:item_Val_Fun_II_n}
of Lemma~\ref{lem:Coupling_n_Solution},
is as follows. 
With strict inequality,
$(S^{II}_t)_{t\in[0,T)}$
would still be a bounded convergent submartingale but its limit
$S_T$
would no longer necessarily satisfy 
$S_T\leq S^{II}_{T}=I_{\{\tau_0^+(R^n(V)) > T\}}$
$\PP_{r,z}$-a.s.
The problem arises on the event 
$\{\tau_0^+(R^n(V)) = T	\}$,
which need not have
probability $0$ for a general 
Brownian motion $V\in\cV$.
In 
part~\eqref{lem:item_Val_Fun_I_n}
of Lemma~\ref{lem:Coupling_n_Solution},
the same phenomenon of the atom 
$\{\tau_0^+(R^n(V)) = T	\}$
occurs, but the required inequality 
$S_T\geq S^{I}_{T}=I_{\{\tau_0^+(R^n(V)) > T\}}$
holds everywhere, including the atom at $T$, as the Bellman process
at $T$,
$S^{I}_{T}$,
takes value zero on
$\{\tau_0^+(R^n(V)) = T	\}$.
\end{remark}

\subsection{Proof of Theorem~\ref{thm:coupling_control}}
\label{subsec:Proof_Coupling}
We establish Theorem~\ref{thm:coupling_control}  in two steps.
The first step 
%in the proof of 
%Theorem~\ref{thm:coupling_control}
consists of generalising the result of Lemma~\ref{lem:Coupling_n_Solution}~\eqref{lem:item_Val_Fun_I_n},
\begin{eqnarray}
\label{eq:Final_upper_bound_n}
\zeta_n^{(I)}(r,z,T) = \sup_{V\in\cV} \PP_{r,z}\left[\tau_0^+(R^n(V))>T\right] 
\qquad\text{for any $r\leq0,z\in\bbE$ and $n\in\bbN$,}
\end{eqnarray}
to the case where
the assumption 
$|\sigma_1|(z)\neq|\sigma_2|(z)$
for all 
$z\in\bbE$
is not satisfied.
The function
$\zeta_n^{(I)}$
in this expression 
is given in~\eqref{eq:Val_Fun_I}
and
$R^n(V)$
and
$\tau_0^+(R^n(V))$
are defined in~\eqref{eq:R_n_V_def}
and~\eqref{eq:Def_Entrance} respectively.
The second step in the proof of 
Theorem~\ref{thm:coupling_control}
consists of a limiting argument that generalises 
Lemma~\ref{lem:Coupling_n_Solution} 
to volatility chains with possibly unbounded generator matrices.

Consider the case of general volatility functions
$\sigma_1,\sigma_2:\bbE\to\bbR$,
which are only assumed to satisfy
integrability condition~\eqref{eq:Mart_Integrab_Quad_VAr}. 
%needed to state our
%problems.
Then,
for any 
$\epsilon>0$,
there exists a function 
$\sigma_1^\epsilon:\bbE\to\bbR$
%which conincides with 
%$\sigma_1$
%on the set
%$\{z\in\bbE:|\sigma_1|(z)\neq|\sigma_2|(z)\}$
%and 
that satisfies~\eqref{eq:Mart_Integrab_Quad_VAr},
coincides with 
$\sigma_1$
on the set where the moduli of the original volatility functions
are already distinct,
\begin{eqnarray*}
\{z\in\bbE:\sigma^\epsilon_1(z)=\sigma_1(z)\} & = &  \{z\in\bbE:|\sigma_1|(z)\neq|\sigma_2|(z)\}
\end{eqnarray*}
and has the following properties:
\begin{eqnarray*}
|\sigma_1^\epsilon|(z)\neq|\sigma_2|(z),\qquad
\lvert\sigma_1^\epsilon(z)-\sigma_1(z)\rvert<\epsilon,\qquad
\sgn(\sigma_1^\epsilon(z)\sigma_2(z))= \sgn(\sigma_1(z)\sigma_2(z)) 
\qquad\text{for all $z\in\bbE$.}
\end{eqnarray*}
Note that, in order to define
$\sigma_1^\epsilon$,
we used the fact that
$|\sigma_1|+|\sigma_2|>0$,
which implies that if
$|\sigma_1|(z)=|\sigma_2|(z)$
for some 
$z\in\bbE$,
then 
$|\sigma_1|(z)>0$.

Define the process 
$R^{n,\epsilon}(V)$
by~\eqref{eq:R_n_V_def},
but
with 
$\sigma_1$
replaced by
$\sigma^\epsilon_1$,
and note that for any
$t\geq0$
we have
\begin{equation}
\label{eq:Difference_esp_n}
R^{n,\epsilon}_t(V)-R^n_t(V)=\int_0^t\left[\sigma_1^\epsilon(Z^n_s)-\sigma_1(Z^n_s)\right]\,\dd B_s.
\end{equation}
The chain 
$Z$
has 
c\`adl\`ag paths
in a state space with discrete topology
by assumption~\eqref{eq:Z_Chain_Assumption}
%has no instantaneous states 
and hence 
$Z^n$,
defined in~\eqref{eq:Stopped_Chain},
has only finitely many jumps,
say
$N_T(Z^n)\in\bbN\cup\{0\}$,
during the time interval 
$[0,T]$.
Therefore identity~\eqref{eq:Difference_esp_n}
implies the inequality 
%\begin{eqnarray*}
%R_t^{n,\epsilon}(V)-R^n_t(V) & \leq & \epsilon N_T(Z^n)\left(\sup_{s\in[0,T]}B_s-\inf_{s'\in[0,T]}B_{s'}\right)
%\qquad\text{for all $t\in[0,T]$.}
%\end{eqnarray*}
$\lvert R_t^{n,\epsilon}(V)-R^n_t(V)\rvert  \leq \epsilon (1+N_T(Z^n))(\sup_{s\in[0,T]}B_s-\inf_{s'\in[0,T]}B_{s'})$
for all $t\in[0,T]$.
Since the right-hand side of this inequality does not depend on 
$t\in[0,T]$,
the random variables
$S_T^\epsilon(V):=\sup_{t\in[0,T]}R_t^{n,\epsilon}(V)$
and
$S_T(V):=\sup_{t\in[0,T]}R_{t}^{n}(V)$
satisfy
\begin{eqnarray*}
\big\lvert S_T^\epsilon(V) -  S_T(V)
\big\rvert
& \leq & 
\epsilon (1+N_T(Z^n))\left(\sup_{s\in[0,T]}B_s-\inf_{s'\in[0,T]}B_{s'}\right)
\quad\text{and}\quad
\lim_{\epsilon\to0}S_T^\epsilon(V) = S_T(V)\>\PP_{r,z}\text{-a.s.}
\end{eqnarray*}
This implies 
$I_{\{S_T(V)<0\}}\leq \liminf_{\epsilon\to0}I_{\{S_T^\epsilon(V)<0\}}$.
Fatou's lemma and the fact that
$\{S_T(V)<0\} = \{\tau_0^+(R^n(V))>T \}$
therefore imply
\begin{eqnarray}
\nonumber
\PP_{r,z}\left[\tau_0^+(R^n(V))>T \right] & \leq & 
\liminf_{\epsilon\to0} \PP_{r,z}\left[ S_T^\epsilon(V)<0\right]
 =   
\liminf_{\epsilon\to0} 
\PP_{r,z}\left[\tau_0^+(R^{n,\epsilon}(V))>T \right] \\
& \leq &  \liminf_{\epsilon\to0} \PP_{r,z}\left[\tau_{0}^+\!\left(R^{In,\epsilon}\right)>t\right],  
\label{eq:First_Step_proof}
%\zeta^{(I)}_{n,\epsilon}(r,z,T),
\end{eqnarray}
where the process 
$R^{In,\epsilon}$
%$\zeta^{(I)}_{n,\epsilon}(r,z,T)$
is defined in~\eqref{eq:Def_of_R}
with 
$\sigma_1$
substituted by
$\sigma_1^\epsilon$
and the last inequality follows by
Lemma~\ref{lem:Coupling_n_Solution}~\eqref{lem:item_Val_Fun_I_n}.

Define a strictly increasing process
$A^{I,\epsilon}=(A^{I,\epsilon}_t)_{t\geq0}$ 
and a non-decreasing process
$A^{I}=(A^{I}_t)_{t\geq0}$,
analogous to~\eqref{eq:Time_Change_A},
by
\begin{equation*}
%\label{eq:Time_Change_A}
A^{I,\epsilon}_t := \int_0^t(|\sigma^\epsilon_1|-|\sigma_2|)^2(Z^n_s)\,\dd s,\qquad
A^{I}_t := \int_0^t(|\sigma_1|-|\sigma_2|)^2(Z^n_s)\,\dd s.
\end{equation*}
The properties of 
$\sigma^\epsilon_1$
imply that 
$A^{I,\epsilon}_t\geq A^{I}_t$
$\PP_z$-a.s.
for all 
$t\geq0$.
As in the proof of Lemma~\ref{lem:measures_II_I},
the independence of 
$B$
and
$Z$
(by Lemma~\ref{lem:_Zand_W_indep})
implies that 
the processes
$(R^{In,\epsilon},Z^n)$
and
$(r+B_{A^{I,\epsilon}},Z^n)$
are equal in law,
where 
$B_{A^{I,\epsilon}}$
denotes the Brownian motion 
$B$
time-changed by
the precess
$A^{I,\epsilon}$.
Similarly, we have that the laws of
$(R^{In},Z^n)$
and
$(r+B_{A^{I}},Z^n)$
coincide,
where 
$R^{In}$
is given in~\eqref{eq:Def_of_R}.
These observations imply the almost sure inequality,
%\begin{eqnarray*}
$\inf\{t\geq0\,:\, B_{A^{I,\epsilon}_t}=-r\} 
 \leq 
\inf\{t\geq0\,:\, B_{A^{I}_t}=-r\}$,
%\end{eqnarray*}
and the fact that the random variable on the left-hand side of this inequality 
has the same
law as 
$\tau_0^+(R^{n,\epsilon}(V))$
while the one on the right-hand side is distributed as
$\tau_0^+(R^{n}(V))$.
This therefore implies the inequality
\begin{eqnarray*}
\nonumber
\PP_{r,z}\left[\tau_0^+(R^{In,\epsilon})>T \right] & \leq & 
\PP_{r,z}\left[\tau_0^+(R^{In})>T \right] 
\end{eqnarray*}
which, together with~\eqref{eq:First_Step_proof} 
and the definition of
$\zeta^{(I)}$
in~\eqref{eq:Val_Fun_I},
yields~\eqref{eq:Final_upper_bound_n}
and hence concludes  step one of the proof of Theorem~\ref{thm:coupling_control}.

In the second step of the proof we assume that the volatility process
$Z$
is a general $(\cF_t)$-Markov chain with state space
$\bbE\subset\bbR^d$, defined in Section~\ref{sec:i}.
For any
$n\in\bbN$,
in~\eqref{eq:Stopped_Chain}
we defined a stopping time
$\tau_n$
and a chain 
$Z^n$,
which is
equal to 
$Z$
up to the time
$\tau_n$.
Lemma~\ref{lem:Coupling_n_Solution}~\eqref{lem:item_Val_Fun_II_n},
equality~\eqref{eq:Final_upper_bound_n}
and the definitions 
of the functions
$\zeta_n^{(II)}$
and
$\zeta_n^{(I)}$
in~\eqref{eq:Val_Fun_II}-\eqref{eq:Val_Fun_I}
imply the following inequalities
for any Brownian motion 
$V\in\cV$,
\begin{eqnarray}
\label{eq:Final_Inequalities}
\PP_{r,z}\left[\tau_0^+(R^{IIn})>T \right]  \leq  
\PP_{r,z}\left[\tau_0^+(R^n(V))>T \right]  \leq  
\PP_{r,z}\left[\tau_0^+(R^{In})>T \right], %\quad\text{for any $V\in\cV$,}
\end{eqnarray}
where 
$R^{n}(V)$
is given in~\eqref{eq:R_n_V_def}
and 
$R^{In}, R^{IIn}$
are defined in~\eqref{eq:Def_of_R}.
Furthermore,
for any 
$t$
in the stochastic interval
$[0,\tau_n]$
the following equalities hold:
\begin{equation*}
R^{n}_t(V)=R_t(V),
\qquad
R^{IIn}_t = R_t(V^{II}),
\qquad
R^{In}_t = R_t(V^{I}),
\end{equation*}
where the process
$R(V)$
is defined in~\eqref{eq:R^V_Rep}
and the Brownian motions 
$V^I$
and
$V^{II}$
are given in~\eqref{eq:Def_VI_VII}.
Therefore,
%for any
%$\omega\in\Omega$
we have
that,
on the event
$\{\tau_n>T\}$,
the random variables 
$I_{\{\tau_0^+(R^n(V))>T\}}$
and
$I_{\{\tau_0^+(R(V))>T\}}$
coincide.
The same holds true for the pairs
$I_{\{\tau_0^+(R^{IIn})>T\}}$
and
$I_{\{\tau_0^+(R(V^{II}))>T\}}$,
and
$I_{\{\tau_0^+(R^{In})>T\}}$
and
$I_{\{\tau_0^+(R(V^{I}))>T\}}$.
Since 
$(\tau_n)_{n\in\bbN}$
is a non-decreasing sequence of stopping times,
such that
$\tau_n\nearrow\infty$
$\PP_{z}$-a.s.
as
$n\to\infty$,
we obtain the following almost sure limits: 
\begin{eqnarray*}
\lim_{n\to\infty} I_{\{\tau_0^+(R^{IIn})>T\}}
 = 
I_{\{\tau_0^+(R(V^{II}))>T\}},
& & 
\lim_{n\to\infty} I_{\{\tau_0^+(R^{In})>T\}}
 = 
I_{\{\tau_0^+(R(V^{I}))>T\}}, \\
\lim_{n\to\infty} I_{\{\tau_0^+(R^n(V))>T\}}
 =  
I_{\{\tau_0^+(R(V))>T\}}. & & 
\end{eqnarray*}
These equalities, a final application of the Dominated Convergence Theorem 
and the inequalities in~\eqref{eq:Final_Inequalities}
imply~\eqref{eq:sol_Cd}-\eqref{eq:sol_Cu}.
This concludes the proof. 
\hfill \ensuremath{\Box}

\subsection{Time-varying extremal couplings}
\label{subsec:Non_Markov_Coupling}
It is tempting to try to prove/generalise  the result in 
Theorem~\ref{thm:coupling_control}
via a direct argument
based on the 
Dambis, Dubins-Schwartz (DDS)-Brownian motion~\cite[Thm~V.1.6]{RevuzYor:99},
avoiding the Bellman principle. 
Let
$\Sigma^{(1)}=(\Sigma^{(1)}_t)_{t\geq0}$
and
$\Sigma^{(2)}=(\Sigma^{(2)}_t)_{t\geq0}$
be two progressively measurable processes
on 
$(\Omega, (\cF_t)_{t\geq0}, \cF, \PP)$,
such that
$\int_0^t\EE\left(\Sigma^{(i)}_s\right)^2\dd s<\infty$
for
$i=1,2$
and any
$t\geq0$.
As usual, for any
$V\in\cV$,
define the difference process
$R(V)=(R_t(V))_{t\geq0}$
by 
$
R_t(V) := r + \int_0^t\Sigma^{(1)}_s\,\dd B_s-
\int_0^t\Sigma^{(2)}_s\,\dd V_s$,
$r\leq0,\> t\geq0$.
Let the candidate extremal Brownian motions
$V^{II}=(V^{II}_t)_{t\geq0}$
and
$V^{I}=(V^{I}_t)_{t\geq0}$
be given by
\begin{equation}
\label{eq:GEn_Mir_Coupl_Non_MArkov}
V^{II}_t := -\int_0^t\sgn\left(\Sigma^{(1)}_s\Sigma^{(2)}_s\right)\,\dd B_s
\qquad\text{and}\qquad
V^{I}_t := \int_0^t\sgn\left(\Sigma^{(1)}_s\Sigma^{(2)}_s\right)\,\dd B_s.
\end{equation}

Under these assumptions the process
$R(V)$
is a martingale for each 
$V\in\cV$.
Hence, by~\cite[Thm~V.1.6]{RevuzYor:99},
there exists a 
(DDS)-Brownian motion
$W^V$,
adapted to the filtration 
$(\cF_{E_u(V)})_{u\geq0}$,
where
the processes
$A(V)=(A_t(V))_{t\geq0}$
and
$E(V)=(E_u(V))_{u\geq0}$
are defined by
$$A_t(V):=\int_0^t\left((\Sigma^{(1)}_s)^2-2C_s \Sigma^{(1)}_s\Sigma^{(2)}_s+(\Sigma^{(2)}_s)^2\right)\,\dd s
\qquad \text{and}\qquad
E_u(V):=\inf\{s\,:\, A_s(V)>u\}$$
and 
$C=(C_t)_{t\geq0}$
is the stochastic correlation between the Brownian motions 
$B$
and
$V$
from~\eqref{eq:V_Rep} in Lemma~\ref{lem:R^V_rep},
and the following representation holds
\begin{eqnarray*}
R_t(V) & = & r + W^V_{A_t(V)}\qquad \text{for all $t\geq0$.}
\end{eqnarray*}
It is clear from these definitions that the following inequalities 
hold almost surely for all times
$t\geq0$:
\begin{eqnarray}
\label{eq:Inequalities_DDS_Time_Change}
A^{I}_t := \int_0^t\left(\lvert\Sigma^{(1)}_s\rvert- \lvert\Sigma^{(2)}_s\rvert\right)^2\,\dd s
\leq  A_t(V) \leq 
 \int_0^t\left(\lvert\Sigma^{(1)}_s\rvert+ \lvert\Sigma^{(2)}_s\rvert\right)^2\,\dd s=: A^{II}_t.
\end{eqnarray}
Let
$\tau_0^+(R(V))$,
$\tau_0^+(r + W^V_{A^{II}})$
and
$\tau_0^+(r + W^V_{A^{I}})$
denote the first-passage times over zero of the processes
$R(V)$,
$r + W^V_{A^{II}}$
and
$r + W^V_{A^{I}}$,
respectively,
and note that 
the inequalities in~\eqref{eq:Inequalities_DDS_Time_Change}
imply
\begin{eqnarray}
\label{eq:Path_Wise_Inequalities_stopping_times_Coupling}
\tau_0^+(r + W^V_{A^{II}})
\leq  
\tau_0^+(R(V))
\leq  
\tau_0^+(r + W^V_{A^{I}})
\end{eqnarray}
on the entire probability space
$\Omega$
for every Brownian motion 
$V\in\cV$.

It is tempting to conclude from this that 
the processes
$r + W^V_{A^{II}}$
and
$R(V^{II})$,
where the Brownian motion 
$V^{II}$
is defined in~\eqref{eq:GEn_Mir_Coupl_Non_MArkov}
have the same law (ditto for the pair
$r + W^V_{A^{I}}$
and
$R(V^{I})$),
which would 
together with~\eqref{eq:Path_Wise_Inequalities_stopping_times_Coupling},
yield a generalisation 
or an alternative proof 
of 
Theorem~\ref{thm:coupling_control}.
However, the counterexample in Section~\ref{subsubsec:GEneral_Fails}
demonstrates that the generalised mirror coupling in~\eqref{eq:GEn_Mir_Coupl_Non_MArkov}
can be suboptimal in this setting. 
The counterexamples to
Theorem~\ref{thm:coupling_control},
based on the continuous-time Markov chains in Section~\ref{subsection:NOn_MArkov_Chain}, 
which are adapted non-Markovian processes with respect to the filtration 
$(\cF_t)_{t\geq0}$,
clearly show that this approach cannot be used as an alternative
proof of 
Theorem~\ref{thm:coupling_control}, because it only requires
the volatility processes 
to be $(\cF_t)$-adapted.
We should stress here however, that 
in the case of deterministic integrands
$\Sigma^{(1)}$
and
$\Sigma^{(2)}$,
Proposition~\ref{prop:Time_Determ}
can be established.~\footnote{We would like to thank David Hobson for this 
observation.}

\begin{proposition}
\label{prop:Time_Determ}
Let $\Sigma^{(1)},\Sigma^{(2)}$
be deterministic processes (i.e. measurable functions of time)  that satisfy
the integrability condition above,
$\lvert\Sigma^{(1)}_s\rvert, \lvert\Sigma^{(2)}_s\rvert>0$
for all 
$s\geq0$
and
$A_t^{II}, A_t^{I}\nearrow\infty$
as
$t\nearrow\infty$.
Then for any 
time horizon
$T>0$
and 
Brownian motion 
$V\in\cV$,
the following inequalities hold:
\begin{eqnarray*}
\PP_{r}\left[\tau_0^+(R(V^{II}))>T\right]
\leq  
\PP_{r}\left[\tau_0^+(R(V))>T\right]
\leq  
\PP_{r}\left[\tau_0^+(R(V^{I}))>T\right].
\end{eqnarray*}
\end{proposition}

\begin{proof}
The integrability assumption 
$\int_0^t(\Sigma^{(i)}_s)^2\dd s<\infty$,
$i=1,2$,
from the beginning of Section~\ref{subsec:Non_Markov_Coupling}
implies that 
$A^{II}$
is a well-defined, finite, strictly increasing differentiable function. 
Its inverse
$E^{II}$,
which is defined on 
$[0,\infty)$
since the limit
$A^{II}$
tends to infinity with increasing time,
is also strictly increasing and differentiable and 
satisfies the following ODE:
\begin{equation}
\label{eq:Def_inverse_E}
E^{II}_u = 
\int_0^u\left(\lvert\Sigma^{(1)}_{E^{II}_s}\rvert+ \lvert\Sigma^{(2)}_{E^{II}_s}\rvert\right)^{-2}\,\dd s.
\end{equation}

Since the left-hand side of~\eqref{eq:Def_inverse_E} is
finite for all 
$u\geq0$, 
for any
$V\in\cV$
the process
$W^{IIV}=(W^{IIV}_t)_{t\geq0}$,
\begin{equation}
\label{eq:Def_BM_E}
W^{IIV}_t := \int_0^{A_t^{II}}
\left(\lvert\Sigma^{(1)}_{E^{II}_u}\rvert+ \lvert\Sigma^{(2)}_{E^{II}_u}\rvert\right)^{-1}\,\dd W^V_u,
\end{equation}
is well-defiend for all
$t\geq0$,
where 
$W^V$
denotes the (DDS)-Brownian motion introduced above.
The quadratic variation of the
continuous local martingale 
$W^{IIV}$
is by~\eqref{eq:Def_inverse_E}
equal to
$[W^{IIV}]_t=E^{II}_{A_t^{II}} = t$,
making
$W^{IIV}$
a Brownian motion by L\'evy's characterisation theorem.
By~\eqref{eq:Def_BM_E}
we obtain
$
\dd W^{IIV}_{E^{II}_s} = 
\dd W^V_s/
(\lvert\Sigma^{(1)}_{E^{II}_s}\rvert+ \lvert\Sigma^{(2)}_{E^{II}_s}\rvert) $ 
and 
$
W^V_u = 
\int_0^{u}
(\lvert\Sigma^{(1)}_{E^{II}_v}\rvert+ \lvert\Sigma^{(2)}_{E^{II}_v}\rvert)\,\dd W^{IIV}_{E^{II}_v} = 
\int_0^{E^{II}_u}
(\lvert\Sigma^{(1)}_{s}\rvert+ \lvert\Sigma^{(2)}_{s}\rvert)\,\dd W^{IIV}_{s},
$
where the last equality follows by~\cite[Prop~V.1.4]{RevuzYor:99}.
Hence 
we find
$  W^V_{A^{II}_t} = 
\int_0^{t}
(\lvert\Sigma^{(1)}_{s}\rvert+ \lvert\Sigma^{(2)}_{s}\rvert)\,\dd W^{IIV}_{s}$
for all 
$t\geq0$.
Since
$\Sigma^{(1)}$
and
$\Sigma^{(2)}$
are non-zero everywhere by assumption, 
the process 
$W=(W_t)_{t\geq0}$,
given by
$W_t:=\int_0^t\sgn(\Sigma^{(1)}_{s}) \dd W^{IIV}_{s}$,
is a Brownian motion and the equalities 
$\lvert\Sigma^{(1)}_{s}\rvert= \sgn(\Sigma^{(1)}_{s})\Sigma^{(1)}_{s}$
and
$\sgn(\Sigma^{(1)}_{s} \Sigma^{(2)}_{s})=
\sgn(\Sigma^{(1)}_{s}) \sgn(\Sigma^{(2)}_{s})$
hold.
Therefore, the processes 
$R(V^{II})$,
where 
$V^{II}$
is given in~\eqref{eq:GEn_Mir_Coupl_Non_MArkov},
and 
$r + W^V_{A^{II}}$
are equal in law and hence~\eqref{eq:Path_Wise_Inequalities_stopping_times_Coupling}
implies the first inequality in the proposition. The second inequality follows
along the same lines. 
\end{proof}

\begin{remarks}
\begin{enumerate}[(i)]
\item It is important to note that the Brownian motion 
$W^{IIV}$,
introduced in~\eqref{eq:Def_BM_E},
is not an element of the set
$\cV$
as it is in general not adapted to the original filtration 
$(\cF_t)_{t\geq0}$.
In fact, 
$W^{IIV}$
is an 
$(\cF_t)$-Brownian motion only in the case 
$V=V^{II}$.
\item The final step in the proof of Proposition~\ref{prop:Time_Determ}
relies on the fact that the stochastic integrals
$$
\int_0^\cdot\left( \Sigma^{(1)}_s +\sgn\left(\Sigma^{(1)}_s\Sigma^{(2)}_s\right)\Sigma^{(2)}_s\right)\dd B_s,
\quad
\int_0^\cdot\left( \Sigma^{(1)}_s+
\sgn\left(\Sigma^{(1)}_s\Sigma^{(2)}_s\right)\Sigma^{(2)}_s\right)\sgn\left(\Sigma^{(1)}_s\right)\dd W^{IIV}_s,
$$
where 
$B$
is a fixed Brownian motion and 
$W^{IIV}$
is defined in~\eqref{eq:Def_BM_E},
are equal in law, which holds since 
$\Sigma^{(1)}$
and
$\Sigma^{(2)}$
are deterministic. Assume that both processes
$\Sigma^{(1)}, \Sigma^{(2)}$
non-deterministic, but
adapted to 
$(\cF_t)_{t\geq0}$
and independent of the Brownian motion 
$B$.
Then,
it is not clear whether one can define the second stochastic integral,
since 
$W^{IIV}$
is not 
$(\cF_t)$-Brownian motion.
Even if this were possible, the laws of the two integrals would 
in general not coincide since the integrand and the integrator are
independent in the former and dependent in the latter integral.
\end{enumerate}
\end{remarks}

\section{Counterexamples}
\label{sec:CounterEx}

\subsection{The presence of drift}
\label{subsec:Drift}
If either of the processes 
$X$
and
$Y(V)$
in~\eqref{eq:Proc_Def_X}
can have drift, 
the conclusion of 
Theorem~\ref{thm:coupling}
fails as the following example demonstrates.

Let $R(V)$
be the difference of 
$X$
and
$Y(V)$
and assume that it takes the form
$$
R_t(V)=r+\mu t+ B_t-\bar\sigma V_t,
$$
where 
$B$
is the fixed $(\cF_t)$-Brownian motion, 
$V\in\cV$ an arbitrary $(\cF_t)$-Brownian motion, 
$\bar \sigma$
a volatility parameter different from
$1$,
$r$
a strictly negative starting point 
and 
$\mu$
a constant positive drift.
Then the candidate extremal Brownian motions in~\eqref{eq:Def_VI_VII}
are given by
$V^I=B$
and
$V^{II}=-B$
and the following lemma holds.

\begin{lemma}
\label{lem:Drift} For any starting point 
$r<0$,
time horizon 
$T>0$,
volatility 
$\bar \sigma>0$
and positive drift
$\mu>0$,
the inequality 
$\PP_{r}\left[\tau_0^+(R(V^{I}))>T\right]  < 
\PP_{r}\left[\tau_0^+(R(V^{II}))>T\right]$
holds.
\end{lemma}

Lemma~\ref{lem:Drift}
implies that 
Theorem~\ref{thm:coupling}
cannot hold for processes with drift. An intuitive explanation 
for this phenomenon is as follows: in the presence of a large drift
upwards, it is better to reduce the volatility as much as possible
(in this case to the level
$|1-\bar \sigma|$),
instead of increasing it to its maximal value
(equal to $(1+\bar \sigma)$),
since the drift 
makes the processes 
$X$
and
$Y(V)$
couple before time 
$T$.

\begin{proof}
Fix
$r<0$,
$T>0$,
$\bar \sigma>0$,
$\mu>0$
and define the function 
$F:(0,\infty)\to[0,1]$
by
$$F(v):=N\left(-\frac{r+\mu T}{\sqrt{T}}v\right) 
- \te^{-2\mu r v^2 } N\left(-\frac{r-\mu T}{\sqrt{T}}v \right),\qquad v>0,$$
and recall that
$\PP_{r}\left[\tau_0^+(R(V^{I}))>T\right]=F\left(1/|1-\bar \sigma|\right)$,
$\PP_{r}\left[\tau_0^+(R(V^{II}))>T\right]=F\left(1/(1+\bar \sigma)\right)$
(see e.g.~\cite[II.2.1,~Eq.~1.1.4]{BorodinSalminen:02}),
where
$N(\cdot)$
denotes the normal cdf.
To establish the lemma it is sufficient to 
%find 
%$\mu>0$
%such that the corresponding 
show that
$F$
is strictly decreasing on the bounded interval
$[1/(1+\bar \sigma), 1/|1-\bar \sigma|]$.
Since the derivative takes the form
$
F'(v)=-2\mu\sqrt{T} n\left(-\frac{r+\mu T}{\sqrt{T}}v \right)+4\mu r v \te^{-2\mu r v^2 } 
N\left(-\frac{r-\mu T}{\sqrt{T}}v \right)
$
and clearly satisfies 
$F'(v)<0$
for all 
$v>0$,
the lemma follows.\footnote{We thank one of the referees for this simplification 
of our original argument.}
%It is easy to see that
%for  any
%$a<0$
%it holds
%$N(a)>-n(a)a/(1+a^2)$, where
%$n(\cdot)$
%%is the normal pdf.
%After a slightly tedious calculation and an application
%of this inequality, it follows that there exists 
%a real constant 
%$C'$,
%which depends on our bounded interval,
%such that the inequality holds:
%$$
%F'(v) \leq  n\left(\frac{r+\mu T}{\sqrt{T}}v\right) \left(\frac{2r}{\sqrt{T}}+\frac{C'}{\mu}\right)\qquad
%\text{for all}\quad
%v\in
%%[1/(1+\bar \sigma), 1/|1-\bar \sigma|].
%$$
%There clearly exists a large
%$\mu>0$,
%such that the right-hand side of this inequality is strictly negative on 
%the entire interval (recall $r<0$) and the lemma follows.
\end{proof}

\subsection{$(\cF_t)$-adapted non-$(\cF_t)$-Markov processes on a discrete state space}
\label{subsection:NOn_MArkov_Chain}
In this section we construct two continuous-time 
$(\cF_t)$-adapted processes with a countable discrete state space,
neither of which are $(\cF_t)$-Markov, and show that in both cases the strategies 
in Theorems~\ref{thm:tracking} and~\ref{thm:coupling} are suboptimal. 
In the first (resp. second) example, Section~\ref{subsec:Non_Markov_counterexample} 
(resp. Section~\ref{subsubsec:Time_Homogeneous}),
the constructed process is semi-Markov (resp. Markov) 
with respect to its natural filtration.
%Markov chains 
%on our probability space, both of which are 
%$(\cF_t)$-adapted but do not satisfy the Markov property
%with respect to the filtration 
%$(\cF_t)_{t\geq0}$.
This demonstrates that the assumption that the chain 
$Z$
is an $(\cF_t)$-Markov process,
not just a Markov process with respect to its ``natural'' filtration,
is indeed necessary in 
Theorems~\ref{thm:tracking} and~\ref{thm:coupling}. 

\subsubsection{$(\cF_t)$-semi-Markov process}
\label{subsec:Non_Markov_counterexample}
Recall that
$B$
is $(\cF_t)$-Brownian motion, fix
$\epsilon\in(0,1)$
and then let the random times 
$T_n$,
$n\in\bbN\cup\{0\}$,
be given 
by
$T_0:=0$
and
$$
T_{n}:=\inf\{t\geq T_{n-1}\,:\,|B_{t}-B_{T_{n-1}}|=\epsilon\}\qquad
\text{for 
$n\geq1$.}
$$
Define the processes
$N=(N_t)_{t\geq0}$
and
$W=(W_t)_{t\geq0}$
by
$$
N_t:=\max\{n\in\bbN\cup\{0\}\,:\,T_n\leq t\}
\qquad\text{and}\qquad
W_t:=B_{T_{N_t}}.
$$
For every
$t>0$
we have
$\{T_n\leq t\}\in\cF_t$
for all
$n\in\bbN$
and hence the process
$W$
is 
$(\cF_t)$-adapted. 
Furthermore
%, in the filtration generated by 
%the pair $(W,B)$, 
$W$ is a continuous-time %time-inhomogeneous simple symmetric random walk
semi-Markov process (i.e. the pair 
$(W,B)$ is $(\cF_t)$-Markov)
with state space
$\epsilon\bbZ$
and c\`adl\`ag trajectories. 
In particular, 
$W$
has only finitely many jumps on any compact interval.
Let
$$Z:=z_0\cE (W),\qquad\text{for a fixed $z_0>0$,}$$
where 
$\cE$
denotes the Dol\'eans-Dade stochastic exponential~\cite[Sec~II.7,~Thm.~37]{Protter:05}.
Therefore, by definition, we have
\begin{equation*}
%\label{eq:DoleansDade_MC}
Z_t = z_0 + \int_0^tZ_{s-}\,\dd W_s = z_0 + \int_0^{T_{N_t}}Z_{s-}\,\dd W_s= Z_{T_{N_t}}, 
\end{equation*}
where the second equality follows from the facts that 
$T_{N_t}\leq t$,
and that there are no jumps of 
$W$
during the time interval
$(T_{N_t},t]$.
The process
$Z$
has a countable state space,\footnote{An additional 
bijection is needed to define a chain with a state space
that is a discrete subspace of a Euclidean space. \label{footnote_5}}\label{page:footnote_bijection}
which can be expressed as 
%$\bbE:=z_0\left(\cup_{m,n\in\bbN}\left\{z_0(1-\epsilon)^n(1+\epsilon)^m:{m,n\in\bbN}\right\}\right)\subset(0,\infty),$
$\bbE:=\left\{z_0(1-\epsilon)^n(1+\epsilon)^m:m,n\in\bbN\right\}\subset(0,\infty)$
and is a continuous-time semi-Markov process 
(as before, $(Z,B)$ is $(\cF_t)$-Markov).
%(since its holding times are not exponential)
%with respect to the filtration generated
%by
%$(Z,B)$,
%which is clearly contained in 
%and is non-Markovian with respect to 
%$(\cF_t)_{t\geq0}$.
%Furthermore, it is clear from the construction that the pair
%$(B,Z)$
%is $(\cF_t)$-Markov. 

Consider the stochastic integral 
$\int_0^\cdot Z_{s}\dd B_s$
and note that
the equality
$W_{T_n}-W_{T_n-}= B_{T_n}-B_{T_{n-1}}$
holds
for all
$n\in\bbN$.
Hence the stochastic 
integral 
can be expressed as follows:
\begin{eqnarray*}
\int_0^t Z_{s}\dd B_s & = & 
\int_0^{T_{N_t}} Z_{s}\dd B_s  +
\int_{T_{N_t}}^t Z_{s}\dd B_s  
=\sum_{n=1}^{N_t} Z_{T_{n-1}}(B_{T_n}-B_{T_{n-1}}) + Z_{T_{N_t}}(B_t-B_{T_{N_t}})\\
& = & (Z_{T_{N_t}}-z_0) + Z_{T_{N_t}}(B_t-B_{T_{N_t}}) = 
Z_t(1+(B_t-B_{T_{N_t}})) -z_0.
\end{eqnarray*}
Therefore, 
since by definition we have
$|B_t-B_{T_{N_t}}|<\epsilon$
and
$Z_t>0$,
the following inequalities hold:
\begin{equation}
\label{eq:Inequality_For_Z}
-z_0\leq (1-\epsilon)Z_t - z_0
\leq
\int_0^t Z_{s}\dd B_s
%\leq
%(1+\epsilon)Z_t - z_0.
\qquad\text{for all $t\geq0$.}
\end{equation}

As in Section~\ref{subsubsec:GEneral_Fails},
define
$\sigma_i:\bbE\to \bbR$
by
$\sigma_i(z):=-iz$
for any $z\in\bbE$ and $i=1,2$,
and note that  by~\eqref{eq:Def_VI_VII} we have
$V^{I}=B$
and
$V^{II}=-B$.
Hence,
for any starting points
$x,y\in\bbR$,
definition~\eqref{eq:Proc_Def_X}
and
inequality~\eqref{eq:Inequality_For_Z}
yield the following almost sure inequalities:
$$
X_t-Y_t(V^I) = x-y + \int_0^t Z_{s}\dd B_s \geq x-y-z_0,
\quad %\text{and}\qquad
X_t-Y_t(V^{II}) = x-y -3 \int_0^t Z_{s}\dd B_s \leq x-y+3z_0.
$$
For any time horizon
$T>0$,
counterexamples to the Conjecture in Section~\ref{subsection:Solution}
(for both Problems~\textbf{(T)} and~\textbf{(C)})
can now be constructed in the same way as in Section~\ref{subsubsec:GEneral_Fails}.

\subsubsection{Non-$(\cF_t)$-Markov Markov chain}
\label{subsubsec:Time_Homogeneous}
In order to define a process
$Z$,
which is an $(\cF_t)$-adapted, time-homogeneous Markov chain in its own filtration
and has properties analogous to the ones in the previous section,
we sample the path of the Brownian motion
$B$
at a sequence of jump times of a Poisson process $N^\epsilon$. 
The key idea is to use the increments of $B$ over the jump times
of $N^\epsilon$ to construct a certain compound Poisson process (in its own filtration), which
is coupled with $B$ and $(\cF_t)$-adapted. The corresponding 
stochastic exponential will then serve as an example of $Z$ with the required properties.

Fix a small
$\epsilon>0$
and assume that $N^\epsilon$
is an $(\cF_t)$-Poisson process\footnote{$N^\epsilon$ is a L\'evy process started at $0$
with state space $\bbN\cup\{0\}$, such that 
$N^\epsilon_t-N^\epsilon_s$
is independent of $\cF_s$
and Poisson distributed with parameter $(t-s)/\epsilon$
for any times $0\leq s<t$.}
with intensity $1/\epsilon$.
Note that 
$N^\epsilon$ 
is necessarily independent of $B$
by Lemma~\ref{lem:_Zand_W_indep}.
%For any $n\in\bbN\cup\{0\}$,
Define $(\cF_t)$-stopping times
\begin{equation} 
\label{eq:N_hitting_time}
T_n:=\inf\{t\geq0:N^\epsilon_t=n\}\quad\text{for any $n\in\bbN$ and $T_{-1}:=T_0:=0$.} 
\end{equation} 
Recall that 
$T_n-T_{n-1}$, $n\in\bbN$,
are IID exponentially distributed with with mean $\epsilon$
and note that
\begin{equation}
\label{eq:TNt_smaller_t}
N^\epsilon_t=\max\{n\in\bbN\cup\{0\}\,:\,T_n\leq t\},
\quad\text{ implying }\quad
T_{N^\epsilon_t}\leq t< T_{N^\epsilon_t+1}\quad \forall t\geq0.
\end{equation}

Define $h(\epsilon):=\exp(-1/\epsilon^2)$
and the function $g_\epsilon:\bbR\to\bbR$ by the formula
$$
g_\epsilon(x):= h(\epsilon)\lfloor 1+ x/h(\epsilon)\rfloor I_{\{x>0\}} +
 h(\epsilon)\lfloor x/h(\epsilon)\rfloor I_{\{x<0\}}, 
%\left\{ \begin{array}{ll}
%\max\{y\in h(\epsilon)\bbZ%\cup\{h(\epsilon)2^{-n}:n\in\bbN\}
%:y\leq x\}
%& \textrm{ if $x\in(-1+h(\epsilon),1)$,}\\
%0 & \textrm{ if $x\in\bbR\setminus(-1+h(\epsilon),1)$.}
%\end{array} \right.
%\qquad\text{where 
%$h(\epsilon):=\exp(-1/\epsilon^2)$}.
$$
%Note that 
%$g_\epsilon(x)=h(\epsilon)\lfloor x/h(\epsilon)\rfloor$
%for all
%$x\in(-1+h(\epsilon),1)$,
where 
$\lfloor y\rfloor$
denotes the largest
element in 
$\bbZ$
smaller than
$y\in\bbR$.
The function 
$g_\epsilon$
satisfies
%is supported in 
%$[-1,1]$
%and satisfies
\begin{equation}
\label{eq:g_epsilon_bound}
\left\vert g_\epsilon(x)-x\right\vert \leq h(\epsilon)\quad \forall x\in\bbR\qquad \text{ and }\qquad
g_\epsilon(x)=0 \iff x = 0.
\end{equation}
We now discretize the increments of $B$ 
%jumps of $V^\epsilon$
using $g_\epsilon$: 
define the process $W^\epsilon=(W^\epsilon_t)_{t\geq0}$
by
$$
W^\epsilon_t:= \sum_{n=0}^{N^\epsilon_t}g_\epsilon\left(B_{T_n}-B_{T_{n-1}}\right),\quad t\geq0.
$$
The process $W^\epsilon$
is $(\cF_t)$-adapted, i.e.
the r.v.
$
W^\epsilon_t= \sum_{m=0}^\infty I_{\{N^\epsilon_t=m\}}\sum_{n=0}^{m}I_{\{T_n\leq t\}}g_\epsilon\left(B_{t\wedge T_n}-B_{t\wedge T_{n-1}}\right)
$
is $\cF_t$-measurable for every 
$t\geq0$
($N^\epsilon$ is $(\cF_t)$-adapted and recall that for any stopping times $\tau\leq\rho$,
the r.v. $B_{\tau}$ is $\cF_\rho$-measurable).
Furthermore, the state space of $W^\epsilon$ 
is $h(\epsilon)\bbZ$ (recall~\eqref{eq:g_epsilon_bound} and $B_0=0$), its trajectories are piecewise constant
and its jumps are IID with distribution $g_\epsilon(B_{T_1})$.
The jump times of 
$W^\epsilon$ 
are given by the sequence of times
$T_n$, $n\in\bbN$,
for the following reason:
$T_1$
is independent of $B$ and 
exponentially distributed making the r.v. $B_{T_1}$ continuous. 
Hence, by~\eqref{eq:g_epsilon_bound}, $g_\epsilon(B_{T_1})\neq0$ almost surely
implying that 
$W^\epsilon$ 
jumps if and only if 
there is a jump in
$N^\epsilon$.
Hence $W^\epsilon$
is a c\'adl\'ag 
$(\cF_t)$-semimartingale,
equal to the sum of its jumps, 
which is a continuous-time random walk in its own filtration. 

\begin{remark}
It is intuitively clear that 
$W^\epsilon$ 
cannot be $(\cF_t)$-Markov: the part of the Brownian path 
over the time interval 
$[T_{N^\epsilon_t},t]$ (recall~\eqref{eq:TNt_smaller_t})
is not contained in the 
$\sigma$-field
generated by $W^\epsilon$ up to time $t$ (but, of course, is in $\cF_t$) \textit{and} provides additional information 
about e.g. the distribution of the random variable
$W^\epsilon_s$ 
%in the interval $[T_{N_t},T_{N_s}]$
for any $s>t$.
The example in this section and Theorems~\ref{thm:tracking} and~\ref{thm:coupling}
imply that 
$W^\epsilon$ 
is indeed non-$(\cF_t)$-Markov if
$\epsilon$ is small enough. A direct rigorous argument establishing this fact
(for any $\epsilon>0$), based on the intuitive description in this remark,
can also be constructed.
\end{remark}

Let 
$Z^\epsilon:=z_0\cE(W^\epsilon)$
be the stochastic exponential of the 
$(\cF_t)$-semimartingale
$W^\epsilon$
(see~\cite[Sec~II.7,~Thm.~37]{Protter:05} for definition).
$Z^\epsilon$
is a time-homogeneous continuous-time Markov chain with
a countable state space and 
c\`adl\`ag paths 
(footnote~\ref{footnote_5} on page~\pageref{page:footnote_bijection}
also applies here).
%where
%it holds 
%$\cE(W^\epsilon)=\prod_{s\leq t}(1+\Delta W^\epsilon_s)$,
%since 
%$W^\epsilon$
%is a pure jumps process.
%(as defined by the first equality in~\eqref{eq:DoleansDade_MC}
%with 
%$W$
%replaced by
%$W^\epsilon$).
%If we can show that,
%It is clear from this construction that 
If
for some
$T>0$
\begin{equation}
\label{eq:Key_Apprx_MC}
\lim_{\epsilon\to0}\int_0^T\EE\left[\left(Z^\epsilon_t-Z_t\right)^2\right]\dd t=0%\qquad\text{$\PP_{z_0}$-a.s.,}
\end{equation}
holds,
where
$Z$
is defined in~\eqref{eq:Z_dependent_BM_Counter},
then,
since the stochastic exponentials 
$Z$
and
$Z^\epsilon$
are square integrable on compact intervals (see Lemma~\ref{lem:For_Gronwal}),
the Burkholder-Davis-Gundy inequality~\cite[Sec~IV.4,~Thm.~48]{Protter:05}
implies the following almost sure convergence
\begin{equation}
\label{eq:Indicators_converge_no_atom}
\phi\left(X^\epsilon_T-Y^\epsilon_T(V)\right)
\to
\phi\left(X_T-Y_T(V)\right),
\qquad
I_{\{\tau_0(X^\epsilon-Y^\epsilon(V))>T\}}
\to
I_{\{\tau_0(X-Y(V))>T\}},
\end{equation}
as
$\epsilon\to0$,
for any Brownian motion 
$V\in\{B,-B\}\subset \cV$,
cost function 
%$\phi\in\cC^2(\bbR)$
$\phi$
and volatility functions 
$\sigma_1,\sigma_2$
given in Section~\ref{subsubsec:GEneral_Fails}
(the processes 
$X^\epsilon,Y^\epsilon(V)$
are defined in~\eqref{eq:Proc_Def_X}
with 
$Z$
replaced by
$Z^\epsilon$
and the stopping time 
$\tau_0(X^\epsilon-Y^\epsilon(V))$
is equal to 
$\inf\{t\geq0:X^\epsilon_t=Y^\epsilon_t(V)\}$).\footnote{We note that $X-Y(V)$
in Section~\ref{subsubsec:GEneral_Fails},
for $V\in\{B,-B\}$, is a geometric Brownian motion (plus a constant). Hence
the distribution of $X_T-Y_T(V)$ does not have atoms, implying in particular that 
$\PP[\tau_0(X-Y(V))=T]=0$. We thank one of the referees for noting that this is 
necessary for the almost sure convergence of the indicators in~\eqref{eq:Indicators_converge_no_atom}
to follow from the BDG inequality, which implies the a.s. convergence 
$\lim_{\epsilon\to0}\|X^\epsilon-Y^\epsilon(V)\|_\infty=0$
for a subsequence.}
The counterexamples from 
Section~\ref{subsubsec:GEneral_Fails}
%(with a
%bounded
%$\phi\in\cC^2(\bbR)$)
show that the conjecture 
in Section~\ref{subsection:Solution}
fails 
(for both Problems~\textbf{(T)} and~\textbf{(C)})
in the case of the process
$Z^\epsilon$
if 
$\epsilon>0$
is small enough. 
In order to complete our counterexample, we need to prove 
that the limit in~\eqref{eq:Key_Apprx_MC} holds. 
To this end we establish the following lemma.

\begin{lemma}
\label{lem:For_Gronwal}
Fix a time horizon $T>0$.
Let the processes 
$N^\epsilon$ and $Z^\epsilon$
be as defined above and let
$Z$
be given by~\eqref{eq:Z_dependent_BM_Counter}.
\begin{enumerate}[(a)]
\item 
\label{lem:a_Prob}
For any 
$\delta>0$ 
and stopping times %$T_n$, $n\in\bbN$,
in~\eqref{eq:N_hitting_time}
we have:
%\begin{equation*}
$\lim_{\epsilon\downarrow0}\sup_{t\in[0,T]} \PP\left[T_{N^\epsilon_t}<t-\delta\right]=0.$
%\end{equation*}
\item 
\label{lem:b_Prob}
$Z^\epsilon$
and
$Z$
are square integrable on compact intervals and 
there exists a constant 
$C_0>0$
such that the following holds:
%Coming back to the proof of~\eqref{eq:Key_Apprx_MC}, note first that Gronwall's inequality
%and Lemma~\ref{lem:For_Gronwal} yield
$$
\EE\left[\left|Z_t- Z^\epsilon_t\right|^2\right]\leq \alpha(t,\epsilon)+\int_0^t\exp(C_0(t-s))\alpha(s,\epsilon)\,\dd s
\qquad
\text{for all $t\geq0$ and small $\epsilon>0$,}
$$
where $\alpha(t,\epsilon)\in[0,\infty)$ satisfies
$ \lim_{\epsilon\downarrow0}\alpha(t,\epsilon)=0$ %\qquad\text{for any $t\geq0$.} $
for any $t\geq0$.
Furthermore,  
%on any interval
%$[0,T]$,
%$T<\infty$,
the function 
$\alpha(\cdot,\epsilon):[0,T]\to\bbR$
can be chosen to be bounded uniformly in all small
$\epsilon>0$.
\end{enumerate}
\end{lemma}

\begin{proof}
\noindent \eqref{lem:a_Prob} Pick any small $\delta>0$ and fix
$\delta_1,\delta_2>0$ such that 
$\delta/2>\delta_1T+\delta_2$.
By the Chebyshev inequality,
the event
$A_{t,\epsilon}:=\{N^\epsilon_t\geq (1-\delta_1)t/\epsilon\}$ satisfies
\begin{equation*}
1-\PP\left[A_{t,\epsilon}\right] 
\leq 
\PP\left[ N^\epsilon_t<(1-\delta_1)t/\epsilon\right] 
\leq 
\PP\left[\lvert N^\epsilon_t-t/\epsilon\rvert>\delta_1t/\epsilon\right] 
\leq \frac{\epsilon^2 t/\epsilon}{\delta_1^2 t^2}
\leq \frac{\epsilon}{\delta_1 T}
\end{equation*}
(recall that both the mean and the variance of 
$N^\epsilon_t$
are equal to
$\epsilon/t$).
Hence
$\lim_{\epsilon\downarrow0}\inf_{t\in[0,T]} \PP\left[A_{t,\epsilon}\right]=1$.
To establish~\eqref{lem:a_Prob}, it hence suffices to prove 
$\lim_{\epsilon\downarrow0}\sup_{t\in[0,T]} \PP\left[T_{N^\epsilon_t}<t-\delta, A_{t,\epsilon}\right]=0$.
Note that we have
\begin{eqnarray*}
\left\{T_{N^\epsilon_t}<t-\delta, A_{t,\epsilon}\right\}
& \subseteq &
\left\{T_{\lfloor (1-\delta_1)t/\epsilon\rfloor}<t-\delta\right\}
\subseteq 
\left\{T_{\lfloor (1-\delta_1)t/\epsilon\rfloor}<  (1-\delta_1)t-\epsilon-\delta_2 \right\} 
\\
& \subseteq &
\left\{\left\lvert T_{\lfloor (1-\delta_1)t/\epsilon\rfloor}- \epsilon \lfloor(1-\delta_1)t/\epsilon\rfloor\right\rvert >\delta_2\right\}
\end{eqnarray*}
for all $t\in[0,T]$
and any
$\epsilon\in[0,\delta/2)$
(recall that  $\delta_2$ satisfies 
$t-\delta<t(1-\delta_1)-\delta_2-\epsilon$
for all $t\in[0,T]$
and that the mean of 
$T_{\lfloor (1-\delta_1)t/\epsilon\rfloor}$
is 
$\epsilon \lfloor(1-\delta_1)t/\epsilon\rfloor >t(1-\delta_1)-\epsilon$).
Hence all we need to show is the equality
$\lim_{\epsilon\downarrow0}\sup_{t\in[0,T]} 
\PP\left[
\left\lvert T_{\lfloor (1-\delta_1)t/\epsilon\rfloor}- \epsilon \lfloor(1-\delta_1)t/\epsilon\rfloor\right\rvert >\delta_2\right]=0$.
Recall that the variance of 
$T_{\lfloor (1-\delta_1)t/\epsilon\rfloor}$
is $\epsilon^2 \lfloor (1-\delta_1)t/\epsilon\rfloor$
and apply Chebyshev's inequality:
$$
\PP\left[\left\lvert T_{\lfloor (1-\delta_1)t/\epsilon\rfloor}- \epsilon \lfloor(1-\delta_1)t/\epsilon\rfloor\right\rvert >\delta_2\right]\leq
\epsilon^2 \lfloor (1-\delta_1)t/\epsilon\rfloor/\delta_2^2\leq
\epsilon (1-\delta_1)T/\delta_2^2.
$$
This proves part~\eqref{lem:a_Prob}.

\noindent \eqref{lem:b_Prob}
Define the process 
$V^\epsilon=(V^\epsilon_t)_{t\geq0}$ 
by
$$
V^\epsilon_t:=B_{T_{N^\epsilon_t}} = \sum_{n=0}^{N^\epsilon_t}\left(B_{T_n}-B_{T_{n-1}}\right),\quad t\geq0.
$$
Note that, as in the case of 
$W^\epsilon$
defined above,
$V^\epsilon$ is an $(\cF_t)$-adapted
c\'adl\'ag semimartingale with piecewise constant paths. The jump times
of  
$V^\epsilon$ 
coincide with those of Poisson process 
$N^\epsilon$.
Hence the stochastic exponentials 
$\overline Z^\epsilon:=z_0\cE(V^\epsilon)$
and 
$Z^\epsilon=z_0\cE(W^\epsilon)$
are also c\'adl\'ag semimartingales and 
posses the following representations
(see e.g.~\cite[Sec~II.7,~Thm.~37]{Protter:05})
for any $t\geq0$:
$$
Z^\epsilon_t=z_0\prod_{n=0}^{N^\epsilon_t}\left(1+g_\epsilon\left(B_{T_{n}}-B_{T_{n-1}}\right)\right)
\qquad
\text{and}\qquad
\overline Z^\epsilon_t=z_0\prod_{n=0}^{N^\epsilon_t}\left(1+B_{T_{n}}-B_{T_{n-1}}\right).
$$
%where the product is taken to be one if 
%$N^\epsilon_t=0$.

Our first task is to control the difference 
$\EE\left[\lvert Z^\epsilon_t - \overline Z^\epsilon_t\rvert^2\right]$.
For any
$t\in[0,T]$
the equality 
\begin{equation}
\label{eq:Conditional_BM_Poisson}
\EE \left[I_{\{N^\epsilon_t>0\}} \prod_{i=1}^{N^\epsilon_t}f_i(B_{T_i}-B_{T_{i-1}})\Big\vert N^\epsilon_t, T_1,\ldots, T_{N^\epsilon_t}\right]=
\EE \left[I_{\{N^\epsilon_t>0\}} \prod_{i=1}^{N^\epsilon_t}F_i(T_i-T_{i-1})\Big\vert N^\epsilon_t, T_1,\ldots, T_{N^\epsilon_t}\right]
\end{equation}
holds for measurable functions
$f_i:\bbR\to\bbR_+$,
$i\in\bbN$,
such that 
$F_i(s):=\EE[f_i(B_s)]<\infty$
for all
$i\in\bbN$
and
$s\leq T$.
This is because
the processes 
$B$
and 
$N^\epsilon$
are independent and hence, 
conditional on the path of 
$N^\epsilon$ up to time $t$,
the increments of $B$
over the holding time intervals of $N^\epsilon$
are independent normal random variables. 
Let
$K_i:=B_{T_{i}}-B_{T_{i-1}}-g_\epsilon\left(B_{T_{i}}-B_{T_{i-1}}\right)$
and
note that by~\eqref{eq:g_epsilon_bound}
we have 
$|K_i|\leq h(\epsilon)$.
Let 
$\mathcal{P}_{n}$
denote the power set of 
$\{1,\ldots,n\}$
and, for any
$S\in\mathcal{P}_{n}$,
let 
$|S|$ be the cardinality of $S$
and $S^c\in\mathcal{P}_{n}$ 
the complement of $S$.
Using this notation 
and 
the elementary inequality  
$(\sum_{i=1}^Na_i)^2\leq N^2 \sum_{i=1}^Na_i^2$
for any non-negative sequence
$(a_i)_{i=1,\ldots,N}$,
we find:
\begin{eqnarray*}
\EE\left[\lvert Z^\epsilon_t - \overline Z^\epsilon_t\rvert^2\right] /z_0
& \leq & 
\EE\left[\left(\sum_{S\in\mathcal{P}_{N^\epsilon_t}\setminus\emptyset}\prod_{j\in S}|K_j| \prod_{i\in S^c}(1+|B_{T_i}-B_{T_{i-1}}|)\right)^2\right]\\
& \leq & 
\EE\left[2^{2N^\epsilon_t}\sum_{S\in\mathcal{P}_{N^\epsilon_t}\setminus\emptyset}\prod_{j\in S}|K_j|^2 \prod_{i\in S^c}(1+|B_{T_i}-B_{T_{i-1}}|)^2\right]\\
& \leq & 
\EE\left[2^{2N^\epsilon_t}\sum_{S\in\mathcal{P}_{N^\epsilon_t}\setminus\emptyset}
\EE\left[ h(\epsilon)^{2|S|}\prod_{i\in S^c}(1+|B_{T_i}-B_{T_{i-1}}|)^2\Big\vert N^\epsilon_t\right]\right]\\
& \leq & 
\EE\left[2^{2N^\epsilon_t}\sum_{S\in\mathcal{P}_{N^\epsilon_t}\setminus\emptyset}
h(\epsilon)^{2|S|} 2^{N^\epsilon_t-|S|} \EE\left[\prod_{i\in S^c}(1+|B_{T_i}-B_{T_{i-1}}|^2)\Big\vert N^\epsilon_t\right]\right].
\end{eqnarray*}
By the tower property, formula~\eqref{eq:Conditional_BM_Poisson}
(with 
$f_i(x):=1+x^2$,
and hence 
$F_i(s)=1+s$,
for 
$i\in S^c$
and 
$f_i(x):=1$
for 
$i\in \bbN\setminus S^c$)
and the fact that for $i\leq N^\epsilon_t$
we have
$T_i-T_{i-1}\leq t \leq T$,
we find
$$
\EE\left[\prod_{i\in S^c}(1+|B_{T_i}-B_{T_{i-1}}|^2)\Big\vert N^\epsilon_t\right]\leq  (1+T)^{N^\epsilon_t-|S|}.
$$
Recall that 
$h(\epsilon)<1$ and hence we have
\begin{eqnarray*}
\EE\left[\lvert Z^\epsilon_t - \overline Z^\epsilon_t\rvert^2\right] 
& \leq & 
z_0h(\epsilon)^2
\EE\left[2^{2N^\epsilon_t}\sum_{S\in\mathcal{P}_{N^\epsilon_t}\setminus\emptyset}
(2+2T)^{N^\epsilon_t-|S|}
 \right] = 
z_0h(\epsilon)^2
\EE\left[2^{2N^\epsilon_t}\sum_{i=1}^{N^\epsilon_t}
\binom{N^\epsilon_t}{i} (2+2T)^{N^\epsilon_t-i}
\right]\\
& \leq & 
z_0h(\epsilon)^2
\EE\left[(12+8T)^{N^\epsilon_t}\right]\leq 
h(\epsilon)^2
\EE\left[A_0\exp(A_1 N^\epsilon_t)\right]
\end{eqnarray*}
%The inequality in~\eqref{eq:g_epsilon_bound}
%and these representations yield
%\begin{equation}
%$\lvert Z^\epsilon_t - \overline Z^\epsilon_t\rvert
%\leq z_0 2^{N^\epsilon_t} h(\epsilon)\leq A_0\exp(A_1 N^\epsilon_t) h(\epsilon),$
%\end{equation}
for some positive constants 
$A_0, A_1$, independent of 
$\epsilon$
and
$t$.
Since 
$\EE[\exp(uN^\epsilon_t)]=\exp((\te^u-1)t/\epsilon)$,
for any $u>0$, 
and 
$h(\epsilon)=\exp(-1/\epsilon^2)$
we get
\begin{equation}
\label{eq:first_part_bound}
\EE\left[\lvert Z^\epsilon_t - \overline Z^\epsilon_t\rvert^2\right]
\leq A_2\exp\left(A_3t/\epsilon-2/\epsilon^2\right) \quad\text{$\forall t\in[0,T]$ and any $\epsilon>0$},
\end{equation}
where the positive constants
$A_2$
and
$A_3$
are independent of 
$\epsilon$
and
$t$.

In order to control the quantity 
$\EE\left[\lvert Z_t - \overline Z^\epsilon_t\rvert^2\right]$,
we apply the representations
of 
$Z=z_0\cE(B)$
and
$\overline Z^\epsilon=z_0\cE(V^\epsilon)$,
implied by the definition of the stochastic exponential~\cite[Sec~II.7,~Thm.~37]{Protter:05}:
$Z_t =z_0+ \int_0^tZ_s\,\dd B_s$ 
and 
$$
\overline Z^\epsilon_t=z_0+\int_0^t\overline Z_{s-}^\epsilon\,\dd V^\epsilon_s
= z_0+\sum_{n=1}^{N^\epsilon_t} \overline Z^\epsilon_{T_{n-1}}\left(B_{T_{n}}-B_{T_{n-1}}\right)
=z_0+\int_0^{T_{N^\epsilon_t}}\overline Z^\epsilon_{s-}\,\dd B_s, 
$$
where 
$\overline Z^\epsilon_{s-}:=\lim_{t\uparrow s}\overline Z^\epsilon_{t}$
if 
$s>0$
and 
$\overline Z^\epsilon_{0-}:=z_0$.
We find
$
Z_t-\overline Z^\epsilon_t = \int_0^{T_{N^\epsilon_t}}(Z_s-\overline Z^\epsilon_{s-})\,\dd B_s + 
\int_{T_{N^\epsilon_t}}^tZ_s\,\dd B_s,
$
%+ 
%\sum_{n=1}^{N_t} Z^\epsilon_{T_{n-1}}
%\left(B_{T_{n}}-B_{T_{n-1}}- g_\epsilon\left(B_{T_{n}}-B_{T_{n-1}}\right)\right).
implying the inequality 
%The definition of the Poisson process
%$N$,
%the definition of 
%$g_\epsilon$
%with the estimate in~\eqref{eq:g_epsilon_bound},
%the elementary fact
%$(a+b+c)^2\leq 9(a^2+b^2+c^2)$
%for any positive
%$a,b,c$,
%and the triangle inequality
%yield 
\begin{equation}
\label{eq:tirangle_bound_gepslion}
\EE\left|Z_t-\overline Z^\epsilon_t\right|^2 \leq 2\EE \left|\int_0^{T_{N^\epsilon_t}}(Z_s-\overline Z^\epsilon_{s-})\,\dd B_s\right|^2 + 
2\EE \left|\int_{T_{N^\epsilon_t}}^tZ_s\,\dd B_s\right|^2.  
%B_1\exp(B_0 N_t)h(\epsilon)^2, 
\end{equation}
%for some constants
%$A_0,B_0,B_1>0$.
%Since, by~\eqref{eq:iid_with_bounded_rate},
%the random variable 
%$N_t$
%%is Poisson distributed with the rate less or equal to
%$t/\epsilon$ 
%and %by definition 
%$h(\epsilon)=\exp(-1/\epsilon^2)$,
%we have
%\begin{eqnarray}
%\label{eq:Poisson_tends_zero}
%\lim_{\epsilon\downarrow0}\EE\left[\exp(B_0 N_t)h(\epsilon)^2\right]=0\qquad\text{for any 
%$t\geq0$.}
%\end{eqnarray}
By~\eqref{eq:TNt_smaller_t}
we have
$T_{N^\epsilon_t}\leq t$
and hence
$$
\left|\int_0^{T_{N^\epsilon_t}}(Z_s-\overline Z^\epsilon_{s-})\,\dd B_s\right|^2  
\leq 
\sup_{u\in[0,T_{N^\epsilon_t}]}
\left|\int_0^{u}(Z_s-\overline Z^\epsilon_{s-})\,\dd B_s\right|^2
\leq 
\sup_{u\in[0,t]}
\left|\int_0^{u}(Z_s-\overline Z^\epsilon_{s-})\,\dd B_s\right|^2,  
$$
which by the Burkholder-Davis-Gundy 
inequality~\cite[Sec~IV.4,~Thm.~48]{Protter:05}
implies the following bound
\begin{eqnarray}
\label{eq:Quadratic_Bound}
\EE\left[ \left|\int_0^{T_{N^\epsilon_t}}(Z_s-\overline Z^\epsilon_{s-})\,\dd B_s\right|^2  \right] 
& \leq & 
A_4\int_0^t\EE\left[\left|Z_s- \overline Z^\epsilon_s\right|^2\right]\,\dd s 
\end{eqnarray}
for some positive constant 
$A_4$
(here we use the fact that 
$\overline Z^\epsilon_{s-}=\overline Z^\epsilon_{s}$ $\PP$-a.s.
for all $s\geq0$).

In order to control the second term on the right-hand side of~\eqref{eq:tirangle_bound_gepslion},
pick an arbitrary 
$\delta>0$.
Then the following inequalities hold:
\begin{eqnarray*}
%\label{eq:Final_term_in_alpha}
\EE\left[\left|\int_{T_{N^\epsilon_t}}^tZ_s\,\dd B_s\right|^2  \right] & \leq &
\EE\left[\sup_{u\in[T_{N^\epsilon_t},t]} \left|\int_u^{t}Z_s\,\dd B_s\right|^2 I_{\{T_{N^\epsilon_t}\geq t-\delta\}} \right]
+\EE\left[\sup_{u\in[T_{N^\epsilon_t},t]} \left|\int_u^{t}Z_s\,\dd B_s\right|^2 I_{\{T_{N^\epsilon_t}< t-\delta\}} \right]\\
&\leq &
\EE\left[\sup_{u\in[t-\delta,t]} \left|\int_u^{t}Z_s\,\dd B_s\right|^2  \right]
+\EE\left[\sup_{u\in[T_{N^\epsilon_t},t]} \left|\int_u^{t}Z_s\,\dd B_s\right|^2 I_{\{T_{N^\epsilon_t}< t-\delta\}} \right]\\
&\leq &
A_5 \EE\left[\int_{t-\delta}^t Z^2_s\,\dd s  \right] 
+\EE\left[\sup_{u\in[0,T]} \left|\int_u^{T}Z_s\,\dd B_s\right|^4\right]^{1/2} 
\sup_{t\in[0,T]} \PP\left[T_{N^\epsilon_t}<t-\delta\right]^{1/2} \\
&\leq &
2z_0A_5 \te^T \delta + A_6 \sup_{t\in[0,T]} \PP\left[T_{N^\epsilon_t}<t-\delta\right]^{1/2}, 
\end{eqnarray*}
where 
$A_5,A_6$
are positive constants 
independent of
$t$,
$\epsilon$
and
$\delta$
(the third inequality follows by the BDG~\cite[Sec~IV.4,~Thm.~48]{Protter:05}
and the Cauchy-Schwartz inequalities, and the fourth is a consequence of 
the fact $Z_t=z_0\exp(B_t-t/2)$ and the 
BDG inequality~\cite[Sec~IV.4,~Thm.~48]{Protter:05} applied for $p=4$).
Part~\eqref{lem:a_Prob} of the lemma implies that 
%there exists a positive 
%constant 
%$A_7$,
%independent of
%$t$,
%$\epsilon$
%and
%$\delta$,
%such that 
\begin{eqnarray*}
%\label{eq:Final_term_in_alpha}
\EE\left[\left|\int_{T_{N^\epsilon_t}}^tZ_s\,\dd B_s\right|^2  \right] & \leq &
(2z_0A_5 \te^T + A_6)\delta \qquad\text{for all small $\epsilon>0$ and $\forall t\in[0,T]$.} 
\end{eqnarray*}
Since 
$\delta>0$
was arbitrary and the left-hand side does not depend on 
$\delta$, we must have the following limit uniformly in 
$t\in[0,T]$:
%An analogous argument yields
%\begin{eqnarray}
%\label{eq:Uniformly_Bounded_Alpha}
%\left|\int_{T_{N^\epsilon_t}}^tZ_s\,\dd B_s\right|^2 \leq  
%\sup_{u\in[0,t]}
%\left|\int_0^{u}Z_s\,\dd B_s\right|^2  
%\leq 2\sup_{u\in[0,t]} Z^2_u.  
%\end{eqnarray}
%Recall that the holding times of $N^\epsilon$,
%i.e. the sequence 
%$T_n-T_{n-1}$,
%$n\in\bbN$,
%are IID exponentials with variance $\epsilon^2$.
%Hence 
%$\lim_{\epsilon\downarrow0} \sup\{T_n-T_{n-1}:n\in\bbN\}=0$
%a.s. 
%The limit in~\eqref{eq:TNt_smaller_t},
%estimate~\eqref{eq:Uniformly_Bounded_Alpha},
%Doob's $L^2$-martingale inequality
%and the Dominated Convergence Theorem imply 
\begin{eqnarray}
\label{eq:Final_term_in_alpha}
\lim_{\epsilon\downarrow0}\overline\alpha(t,\epsilon)=0,\qquad\text{where}\quad \overline\alpha(t,\epsilon):=2\EE\left[\left|\int_{T_{N^\epsilon_t}}^tZ_s\,\dd B_s\right|^2  \right]. 
\end{eqnarray}
The following inequalities are a consequence of~\eqref{eq:tirangle_bound_gepslion},~\eqref{eq:Quadratic_Bound}
and~\eqref{eq:Final_term_in_alpha}:
$$
\EE\left[\left|Z_t- \overline Z^\epsilon_t\right|^2\right]
\leq  
2A_4 \int_0^t\EE\left[\left|Z_s- \overline Z^\epsilon_s\right|^2\right]\,\dd s +  \overline\alpha(t,\epsilon)\qquad \forall t\in[0,T].
$$
A well known elementary estimate (Gronwall's lemma) implies
$$
\EE\left[\left|Z_t- \overline Z^\epsilon_t\right|^2\right]\leq \overline 
\alpha(t,\epsilon)+\int_0^t\exp(2A_4(t-s))\overline \alpha(s,\epsilon)\,\dd s
\qquad
\text{for all $t\in[0,T]$ and small $\epsilon>0$.}
$$
Define 
%$C_0:=B_2/A_0$
%and
$\alpha(t,\epsilon):= 2\overline\alpha(t,\epsilon)+
2A_2\exp\left(A_3t/\epsilon-2/\epsilon^2\right) 
$
%\quad\text{$\forall t\in[0,T]$ and any $\epsilon>0$},
%\frac{1}{A_0}\EE\left[\left|\int_{T_{N^\epsilon_t}}^tZ_s\,\dd B_s\right|^2  \right] + 
%\frac{B_1}{A_1}\EE\left[\exp(B_0 N^\epsilon_t)h(\epsilon)^2\right].
and note that this inequality and~\eqref{eq:first_part_bound}
%,~\eqref{eq:tirangle_bound_gepslion},~\eqref{eq:Quadratic_Bound}
%and~\eqref{eq:Final_term_in_alpha}:
%The last statement in the lemma follows from the fact that 
%$N^\epsilon_t\leq N^\epsilon_T$
%for any
%$t\in[0,T]$,
%the bound in~\eqref{eq:Uniformly_Bounded_Alpha},
%the right-hand side of which is independent of 
%$\epsilon$,
%and Doob's $L^2$-martingale inequality.
yield
\begin{eqnarray*}
\EE\left[\left|Z_t- Z^\epsilon_t\right|^2\right]
& \leq  &
2\EE\left[\left|Z_t- \overline Z^\epsilon_t\right|^2\right]+
2\EE\left[\left|\overline Z^\epsilon_t- Z^\epsilon_t\right|^2\right]%\\ & \leq & 
\leq 
\alpha(t,\epsilon)+\int_0^t\exp(2A_4(t-s))\alpha(s,\epsilon)\,\dd s,
%\int_0^t\EE\left[\left|Z_s- Z^\epsilon_s\right|^2\right]\,\dd s + \alpha(t,\epsilon),\qquad \text{for all $t\in[0,T]$ and small $\epsilon>0$,}
\end{eqnarray*}
which concludes the proof of the lemma. 
%where $\alpha(t,\epsilon)\in[0,\infty)$ satisfies
%$ \lim_{\epsilon\downarrow0}\alpha(t,\epsilon)=0$ %\qquad\text{for any $t\geq0$.} $
%for any $t\geq0$.
\end{proof}

Going back to the equality in~\eqref{eq:Key_Apprx_MC}
for 
$T\in(0,\infty)$,
by Lemma~\ref{lem:For_Gronwal}~\eqref{lem:b_Prob}
we have 
$$
\int_0^T
\EE\left[\left|Z_t- Z^\epsilon_t\right|^2\right]\,\dd t
\leq \int_0^T\alpha(t,\epsilon)\dd t+\int_0^T\dd t\int_0^t\exp(C_0(t-s))\alpha(s,\epsilon)\,\dd s.
$$
Since
$T$
is fixed and 
$\alpha(\cdot,\epsilon)$
is bounded uniformly in 
$\epsilon$
on
$[0,T]$,
the DCT and
Lemma~\ref{lem:For_Gronwal}
imply that the right-hand side of this inequality tends to 
zero and the counterexample follows.

\subsection{$(\cF_t)$-Feller process $Z$ independent of $B$}
\label{subsection:Independence_Also_Does_NOt_Work}
The final counterexample shows that the ``tracking'' part of the
conjecture in Section~\ref{subsection:Solution}
fails for general Feller processes even if 
$Z$
and
$B$
are independent. 

Assume that
there exist an 
$(\cF_t)$-Brownian motion
$B^\perp\in \cV$,
independent 
of
$B$,
and define the $(\cF_t)$-Feller process
$Z:=z_0+B^\perp$
with state space
$\bbE:=\bbR$
for any starting point
$z_0\in\bbR$.
Let
$\sigma_1(z):=2z$
and
$\sigma_2(z):=z$,
for any
$z\in\bbR$,
and note that 
by~\eqref{eq:Def_VI_VII} we have
$V^I=B$.
We will now show that, for the cost function
$\phi(x):=x^4$,
the first inequality in Problem~\textbf{(T)} fails,
i.e. there exists a Brownian motion 
$V\in\cV$
such that  for any
$T>0$
\begin{equation}
\label{eq:Inequality_Tracking_Counter_Indep}
\EE_{r,z_0}\left[(R_T(V))^4\right]< \EE_{r,z_0}\left[(R_T(V^I))^4\right]
\end{equation}
holds,
where 
$R(V)=X-Y(V)$
(and $X$, 
$Y(V)$
given in~\eqref{eq:Proc_Def_X}
for any
$V\in\cV$)
and
$R_0(V)=r$,
$Z_0=z_0$.

To construct
such a process
$V$,
define the family 
$V^c=(V^c_t)_{t\geq0}$,
$c\in[-1,1]$,
of 
$(\cF_t)$-Brownian motions
by
$$
V^c_t:=\sqrt{1-c^2}B_t+c B^\perp_t,
$$
and note that 
$V^0=B=V^I$.
Therefore the difference process
$R(V^c)$
takes the form
$$
R_t(V^c) = r + \int_0^t\left(2Z_s\dd B_s-Z_s\dd V^c_s\right) 
= r +\left(2-\sqrt{1-c^2}\right)\int_0^tZ_s\,\dd B_s -c \int_0^tZ_s\,\dd B^\perp_s,
$$
and hence we find
$\dd [R(V^c),R(V^c)]_t=(5-4\sqrt{1-c^2})Z_t^2\dd t$
%where
%$k(c):=5-4\sqrt{1-c^2}$,
and 
%the covariation measure takes the form
$\dd [R(V^c),Z]_t=-cZ_t\dd t$.

\begin{lemma}
\label{lem:Final_Counter_Ex}
Define
$\psi^c(r,z,t):=\EE_{r,z}[(R_t(V^c))^4]$
for any
$r,z\in\bbR$
and
$t\geq0$.
Then we have 
\begin{eqnarray*}
\psi^c(r,z,t) & = & r^4 + 6k(c)r^2z^2 t+ 3k(c)(r^2+k(c)z^4-4crz^2)t^2\\
& & 
+ k(c)((7k(c)+8c^2)z^2-4cr)t^3
+ (7k^2(c)/4 +2c^2k(c))t^4,
\end{eqnarray*}
where 
$k(c):=5-4\sqrt{1-c^2}$
for any
$c\in[-1,1]$.
\end{lemma}

\begin{proof}
The representation in the lemma for the expectation 
$\psi^c(r,z,t)$
follows from martingale arguments and stochastic
calculus. Alternatively to verify the lemma, one can easily
check that the function 
$\varphi$,
given by the formula above,
satisfies the PDE
$$
\frac{1}{2}k(c)z^2\frac{\partial^2 \varphi}{\partial r^2}
-cz\frac{\partial^2 \varphi}{\partial r\partial z}+
\frac{1}{2}\frac{\partial^2 \varphi}{\partial z^2} =
\frac{\partial \varphi}{\partial t}, 
$$
with boundary condition 
$\varphi(r,z,0)=r^4$
and polynomial growth in 
$r$ and $z$.
An application of the Feynman-Kac formula then yields
$\psi^c=\varphi$.
\end{proof}

Note that 
$k'(0)=0$
and hence 
the derivative in
$c$
at
$c=0$
of the value function 
$\psi^c(r,z_0,T)$
equals
$$
\frac{\partial \psi^c}{\partial c}(r,z_0,T)\Big\lvert_{c=0} = -r(12z_0^2T+4T^3).
$$
Since this quantity is non-zero for 
any
$r\neq0$,
inequality~\eqref{eq:Inequality_Tracking_Counter_Indep}
is satisfied
(by Lemma~\ref{lem:Final_Counter_Ex})
for some 
$V=V^c$
with $c\neq0$
(recall that 
$V^0=B=V^I$).
An analogous argument can be used to show that the second 
inequality in Problem~\textbf{(T)}
also fails in this setting.

\appendix
%========================================================================================
\section{Proofs of Lemmas~\ref{lem:R^V_rep} and~\ref{lem:Q_mart}}
\label{sec:App}
\subsection{Proof of Lemma~\ref{lem:R^V_rep}}
\label{sec:App_BM_representation_sketch}
It is clear that Lemma~\ref{lem:R^V_rep}
follows from~\eqref{eq:Proc_Def_X}
and the basic properties of stochastic integrals if,
for any 
$V\in\cV$,
we can find a progressively measurable process
$C$
and
$W \in\cV$,
such that 
$-1\leq C_t\leq 1$
for all
$t\geq0$
$\PP$-a.s.,
$W$
and
$B$
independent and 
\begin{eqnarray}
\label{eq:rep_Final_W}
V_t=\int_0^t C_s\,\dd B_s + \int_0^t (1-C_s^2)^{1/2}\,\dd W_s.
\end{eqnarray}

By the Kunita-Watanabe inequality~\cite[Sec~II.6,~Thm.~25]{Protter:05},
the signed random measure 
$\dd [V,B]_t$
on the predictable 
$\sigma$-field
is absolutely continuous with respect to 
Lebesgue measure
$\dd [B,B]_t=\dd t$.
Hence, there exists a predictable process 
$C=(C_t)_{t\geq0}$,
such that 
$\dd [V,B]_t=C_t \dd t$,
and for any
$s<t$
we have
$|[V,B]_t-[V,B]_s|\leq t-s$.
Therefore, we may assume that 
$|C_t|\leq1$
and 
define the processes
$D_t:=(1-C_t^2)^{1/2}$
and
$M_t:=V_t-\int_0^tC_s\dd B_s$.
Note  that  the equalities
$[M,B]_t=0$,
$[M,M]_t=\int_0^t D_s^2\dd s$
and
$\int_0^tI_{\{D_s>0\}}D_s^{-2}\dd [M,M]_s\leq t$
hold.
Therefore the continuous local martingale 
$W$,
given by
$$
W_t : = 
\int_0^tI_{\{D_s>0\}}D_s^{-1}\dd M_s  
+ 
\int_0^tI_{\{D_s=0\}}\dd B^\perp_s ,
$$
is well-defined, where 
$B^\perp\in\cV$
is a Brownian motion independent of 
$B$.
L\'evy's characterisation theorem 
applied to 
$W$
now yields the representation in~\eqref{eq:rep_Final_W}
and hence implies 
Lemma~\ref{lem:R^V_rep}.
\hfill \ensuremath{\Box}

\subsection{Proof of Lemma~\ref{lem:Q_mart}}
\label{sec:App_Lemma_Q_mart}
The assumptions on 
$Q$
and
$F$
imply that 
$\EE[|M^U_t|]<\infty$
for all 
times
$t\geq0$.
The additive structure of the process
$M^U$
implies that it is sufficient to prove the following
almost sure equality:
\begin{eqnarray}
\label{eq:Mart_Req_U}
\EE_z\left[\sum_{t<s\leq t'}\left[F(s,U_s,Z_s)-F(s,U_s,Z_{s-})\right]\Bigg\vert\cF_t\right] & = &
\EE_z\left[\int_t^{t'}(QF(s,U_s,\cdot))(Z_{s-})\,\dd s\bigg\vert\cF_t\right],
\end{eqnarray}
for any
$0<t<t'$
and
$z\in\bbE$.
The jump-chain holding-time description of the continuous-time chain 
$Z$,
the continuity of the process
$U$
and the continuity and boundedness of the function
$F$
imply
\begin{eqnarray}
\label{eq:local_mart_prop}
\EE_z\left[\sum_{u<s\leq u+\Delta u}\left[F(s,U_s,Z_s)-F(s,U_s,Z_{s-})\right]\Bigg\vert\cF_u\right] & = &
\Delta u (QF(u,U_u,\cdot))(Z_{u})+o(\Delta u),
\end{eqnarray}
for any 
$u>0$
and small
$\Delta u>0$.
In this expression, 
for each 
$\Delta u$,
$o(\Delta u)$
represents an 
$\cF_u$-measurable random variable 
which is bounded in modulus 
by 
$C\Delta u$,
for some constant 
$C>0$
independent of 
$\Delta u$
(here we use assumption~\eqref{eq:sup_Assum_Q} and the 
boundedness of
$F$),
and 
$\lim_{\Delta u\downarrow0}\frac{o(\Delta u)}{\Delta u}=0$
almost surely.

We now decompose the left-hand side of~\eqref{eq:Mart_Req_U}
into a sum over the time intervals of length
$\Delta t>0$,
where
$\frac{t'-t}{\Delta t}\in\bbN$,
and apply~\eqref{eq:local_mart_prop} to each summand:
\begin{eqnarray}
\nonumber
\EE_z\left[\sum_{t<s\leq t'}\left[F(s,U_s,Z_s)-F(s,U_s,Z_{s-})\right]\Bigg\vert\cF_t\right]  = 
  \sum_{i=0}^{\frac{t'-t}{\Delta t}-1}\EE_z\left[\sum_{i<\frac{s-t}{\Delta t}\leq i+1}\left[F(s,U_s,Z_s)-F(s,U_s,Z_{s-})\right]\Bigg\vert\cF_t\right] 
\end{eqnarray}
\vspace{-4mm}
\begin{eqnarray}
\label{eq:Interm_Step}
& =  & \frac{o(\Delta t)}{\Delta t}  +  \Delta t \sum_{i=0}^{\frac{t'-t}{\Delta t}-1}\EE_z\left[(QF(t+i\Delta t,U_{t+i\Delta t},\cdot))(Z_{{t+i\Delta t}})\big\vert\cF_t\right]. 
\end{eqnarray}
The properties of the random variables 
$o(\Delta t)$
listed in the paragraph above, the Dominated Convergence Theorem applied to the right-hand side of~\eqref{eq:Interm_Step}
as
$\Delta t\downarrow0$,
the definition of the Lebesgue integral and the fact that 
$Z$
jumps only finitely many times during the time interval
$[t,t']$
together imply the equality in~\eqref{eq:Mart_Req_U}.
This concludes the proof of the lemma.
\hfill \ensuremath{\Box}

\section{Two classes of examples of Markov Chains}
\label{app:Examples}

We first construct a chain $Z$ that does not satisfy~\eqref{eq:Coupling_Int_Assump}
but satisfies~\eqref{eq:Int_Assump}.\footnote{We thank the referee for pointing 
out a potential issue with the relation between the assumptions in~\eqref{eq:Int_Assump}
and~\eqref{eq:Coupling_Int_Assump}.}
Let
$\bbE=\{1,2,\ldots\}$
and define the $Q$-matrix by 
$Q(1,1)=-1/2$,
$$
Q(1,n)=1/2^n, \quad 
Q(n,1)=-Q(n,n)=\beta_n>0
\quad 
\text{for}\quad
n\geq2,\quad\text{such that}\quad
\sum_{n=2}^\infty 1/\beta_n<\infty,
$$
and zero everywhere else. The idea is that $Z$ 
makes very big jumps with small intensity and then
very quickly jumps back to $1$ since $\beta_n$ are large. 
The process is stationary with invariant distribution
$\pi$
given by the detailed balance equations
$\pi_n Q(n,1)=\pi_1 Q(1,n)$, 
i.e. 
$\pi_n>0$
for all 
$n\in\bbE$
and
$\pi_n=\pi_1/(2^n\beta_n)$.
For the function 
$f:\bbE\to\bbR$,
given by
$f(n)=2^n$,
and any 
$m\in\bbE$
we have 
$$\pi_m \EE_m[f(Z_1)]\leq \sum_{n=1}^\infty \pi_n \EE_n[f(Z_1)]=\EE_\pi[f(Z_1)] =\sum_{n=1}^\infty f(n)\pi_n=
\pi_1\sum_{n=1}^\infty 1/\beta_n<\infty,$$
and hence~\eqref{eq:Int_Assump} holds for $T=1$
and $\sigma_i$, $i=1,2$,
such that 
$|\sigma_1|^2+|\sigma_2|^2=f$.
Note however that 
$$
Qf(1)=\sum_{n=1}^\infty Q(1,n)f(n)=\infty
$$
and~\eqref{eq:Coupling_Int_Assump} 
fails.

To construct a chain $Z$ 
such that~\eqref{eq:Coupling_Int_Assump} 
holds and~\eqref{eq:Int_Assump}
fails,
pick a 
$Q$-matrix 
$Q$
on an infinite state space 
with the properties that 
$Z$
is irreducible and
only finitely many elements in each row 
of $Q$
are non-zero (e.g. a birth-catastrophe process with
$\bbE=\{1,2,\ldots\}$
and 
$Q(n,n+1)=Q(n,1)=1$,
$Q(n,n)=-2$ for 
$n\geq2$
and $Q(1,2)=-Q(1,1)=1$).
Then the function 
$f:\bbE\to\bbR$,
given by
$f(n)=1/P_1(1,n)$,
is finite at every 
$n\in\bbE$.
It hence satisfies~\eqref{eq:Coupling_Int_Assump} 
but clearly has the property
$P_1f(1)=\infty$,
which violates~\eqref{eq:Int_Assump}.

%========================================================================================
%\nocite{*}
%\bibliographystyle{chicago}
%\bibliographystyle{amsplain}
%\bibliographystyle{plain}
\bibliographystyle{abbrv}
\bibliography{refs}
\end{document}